\documentclass[12pt]{extarticle}
\usepackage[utf8]{inputenc}
\usepackage[T1]{fontenc}

\usepackage{amssymb,amsfonts}
\usepackage{amsthm}
\usepackage{graphicx}
\usepackage[all,arc]{xy}
\usepackage[colorlinks=true, allcolors=blue]{hyperref}
\usepackage{enumerate}
\usepackage{bbm}
\usepackage{dsfont}
\usepackage{mathrsfs}
\usepackage{tikz-cd}
\usepackage{import}
\usepackage[utf8]{inputenc}
\usepackage{hyperref}
\usepackage{tikz} 
\usepackage{xspace}
\usepackage{ytableau}
\usepackage{subcaption} 

\usepackage[left=1in,top=1in,right=1in,bottom=1in]{geometry}

\usepackage{mathtools}
\usepackage{comment}
\usepackage{makecell} 

\newtheorem{thm}{Theorem}[section]
\newtheorem{cor}[thm]{Corollary}
\newtheorem{prop}[thm]{Proposition}
\newtheorem{lem}[thm]{Lemma}
\newtheorem{conj}[thm]{Conjecture}

\newtheorem{innercustomthm}{Theorem}
\newenvironment{customthm}[1]
{\renewcommand\theinnercustomthm{#1}\innercustomthm}
{\endinnercustomthm}

\theoremstyle{definition}

\newtheorem{example}[thm]{Example} 

\newtheorem{rmk}[thm]{Remark} 

\newtheorem{quest}[thm]{Question} 

\newtheorem{expe}[thm]{Experiment} 

\def\R{\mathbb{R}}
\def\C{\mathbb{C}}
\def\T{\mathbb{T}} 
\renewcommand{\P}{\mathbb{P}}
\newcommand{\M}{\mathcal{M}} 

\newcommand{\mycomment}[1]{}

\definecolor{claudiaColor}{RGB}{102, 0, 204}

\definecolor{martaColor}{RGB}{50, 168, 143}

\definecolor{alheydisColor}{RGB}{150, 18, 43}

\title{\bf Positive del Pezzo Geometry}
\date{}

\author{Nick Early,
	Alheydis Geiger, Marta Panizzut, \\
	Bernd Sturmfels and Claudia He Yun}

\begin{document}
	
	\maketitle
	
	\begin{abstract} \noindent
		Real, complex, and tropical algebraic geometry join forces 
		in a  new branch of mathematical physics called positive geometry.
		We develop the positive geometry of 
		del Pezzo surfaces  and their moduli spaces, viewed as
		very affine varieties.
		Their connected components are derived from polyhedral spaces
		with Weyl group symmetries.		
		We study their canonical forms and scattering amplitudes,
		and we solve the likelihood equations.
	\end{abstract}
	
	\section{Introduction}
	
	It has been known for two centuries
	that a cubic surface contains $27$ lines. If the surface is obtained
	by blowing up six  real points in $\mathbb{P}^2$, then the surface  
	and its $27$ lines are real. We consider the very affine cubic surface
	obtained by removing the $27$ lines.
	That  {\bf real} surface has $130$ connected components, namely
	$10$ triangles, $90$ quadrilaterals and $30$ pentagons.
	This appears in our Theorem~\ref{thm:subdivision}.
	The {\bf complex} surface has Euler characteristic $90$,
	as shown in Lemma~\ref{lem:1690}.
	This is the number of solutions to the equations
	known as likelihood equations in statistics and as scattering
	equations in physics. The {\bf tropical} surface is a balanced polyhedral
	complex which has the cograph of the Schl\"afli graph at~infinity. This is the
	$10$-regular graph with $27$ vertices that records
	intersections among the $27$ lines; see \cite{RSS2}.
	
	This article studies moduli of del Pezzo surfaces through the lens 
	of positive geometry~\cite{ABL,ABHY}.
	It builds on work of  Sekiguchi and Yoshida \cite{Sekolder, Sek, SY}
	and Hacking, Keel and Tevelev~\cite{HKT}.
	Positive geometry highlights the trinity of
	real, complex and tropical geometry.
	We  saw this for the cubic surface
	in the previous paragraph,
	and we next summarize other main results.
	
	Let $Y(3,n)$ denote the moduli space of configurations of $n$ 
	points in general position in the complex projective plane $\P^2$. 
	For us, this means that no three points are  collinear and no six lie on a conic. 
	When $n\leq 7$, the space $Y(3,n)$ parametrizes marked del 
	Pezzo surfaces of degree $9-n$ as the blow up of $\P^2$ at the $n$ points.   
	From $Y(3,8)$ one obtains the moduli space by requiring that the eight points are not on a cubic that is
	singular at one of the~points. 	
	
	The Euler characteristics of $Y(3,6)$ and $Y(3,7)$ over $\mathbb{C}$ were recently computed  in \cite{Bergvall,BG}
	by means of cohomological methods. We here present an alternative derivation:

	\begin{customthm}{\ref{thm:euler}}
		The Euler characteristic of the {\bf complex} moduli space $Y(3,n)$ is $32$ for $n=6$ and it is
		$3600$ for $n= 7$.  	For $n=8$, a numerical computation gives $4884387$ as a lower bound for the Euler characteristic.
	\end{customthm}
	
	\begin{customthm}{\ref{thm:2polytopes}}
		The {\bf real} moduli space $Y(3,6)$ has $432$ connected components, all $W(E_6)$ equivalent, and
		the closure of each is homeomorphic as a cell-complex to a simple
		$4$-polytope with f-vector $(45,90,60,15)$. The real moduli space $Y(3,7)$ has $60480$ connected components, all $W(E_7)$ equivalent,
		and the closure of each is homeomorphic as a cell-complex to a simple $6$-dimensional homology ball with f-vector $(579,1737,2000,1105,$ $297,34)$.
	\end{customthm}
	
	These  complex and real results complement those on
	{\bf tropical} moduli spaces in \cite{RSS1, RSS2}. For instance, it is known that
	the tropicalization of $Y(3,6)$ is a simplicial fan with f-vector $(76,630,1620,1215)$. 
	The identification of the components in Theorem \ref{thm:2polytopes}
	and the transitive Weyl group action
	are due to Sekiguchi and Yoshida \cite{Sekolder, SY}.
	Our contribution is a combinatorial study of these objects,
	which we call  {\em pezzotopes}. Also new are their f-vectors. We use the term pezzotope for two objects: the $W(E_n)$-equivalent connected components of $Y(3,n)$, for $n=6,7$; and in the case of $Y(3,6)$, the polytope that is homeomorphic to those components. We add the word ``curvy'' to emphasize that we are referring to the components.
	
	One motivation for studying the real moduli spaces $Y(3,n)$ comes from physics.
	We aim to show that the pezzotopes are positive geometries in the
	sense of \cite{ABL, ABHY}.  This involves identifying the canonical form
	and the resulting scattering amplitude.
	In Sections \ref{sec8} and \ref{sec9}, we compute the scattering amplitude for the $E_6$ pezzotope, see Equation \eqref{eq:E6amplitude}, and a system of perfect $u$-equations, in the sense of \cite{AHLT, HLRZ}, see Equation \eqref{eq:perfectuE6}. For $n=7$, see Theorem \ref{thm:E7pezzo}.
	
	Figure \ref{fig:zwei} depicts the $E_6$ pezzotope. It
	offers
	a colorful illustration of Theorems 
	\ref{thm:2polytopes} and~\ref{thm: E6amplitude}.
	
	\begin{figure}[ht] 
		\centering
		\includegraphics[scale=0.64]{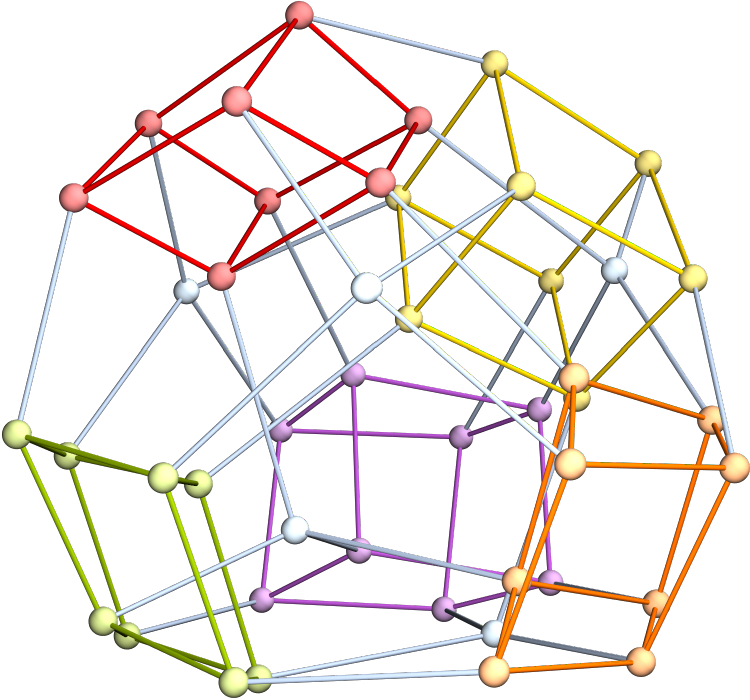} 
		\vspace{-0.14cm}	
		\caption{The real moduli space $Y(3,6)$ is glued from $432$ copies of
			a simple $4$-polytope with $f= (45,90,60,15)$.
			The picture shows its edge graph.
			The amplitude is a rational function, given as the sum of
			$45$ reciprocal monomials, one for each vertex.
			Singularities correspond~to  facets: five cubes (in color) and ten associahedra.
			The Weyl group $W(E_6)$ acts on this data.
			\label{fig:zwei}}
		\vspace{-0.2cm}		
	\end{figure}
	
	Our final highlights concern the formal definition of 
	positive geometry by
	Arkani-Hamed, Bai and Lam \cite{ABL, Lam},
	which stipulates the existence of a canonical differential form $\Omega$ like~(\ref{eq:Omega}). We prove the following for $Y(3,6)$ and conjecture the same for $Y(3,7)$.
	
	\begin{customthm}{\ref{thm:Y3nPosGeom}}
		The moduli space $Y(3,6)$ is a positive geometry for any
		of its $432$ regions, each of which is a curvy $E_6$ pezzotope.
	\end{customthm}

In Theorem \ref{thm:pos_surj_par}
 we single out one of these $432$ regions,  to be denoted $Y_+(3,6)$ and
 called the positive moduli space of del Pezzo surfaces, and we provide a parametrization for it. 

	We now define concepts that are used throughout the paper, starting with 
	a realization of $Y(3,n)$ as a very affine variety, i.e. as a closed subvariety of an algebraic torus.
	Homogeneous coordinates for the $n$ points are the columns of
	a $3\times n$ matrix 
	whose $3\times 3$ minors $p_{ijk}$ are nonzero. The condition for six points to lie on a conic  is
	the vanishing of the Pl\"ucker binomial
	\begin{equation}
		\label{eq:coconic}
		q \,\,=\,\, p_{134}\,p_{156}\,p_{235}\,p_{246} - p_{135} \,p_{146}\, p_{234}\, p_{256}.
	\end{equation}
	Two matrices represent the same configuration
	if and only if they differ by left multiplication with ${\rm GL}(3,\C)$ and by scaling 
	the columns with an element in the torus $(\C^*)^n$. Therefore, by fixing a projective basis, 
	every configuration in $Y(3,n)$ is represented by a unique matrix:
	\begin{equation}\label{eq:generalmatrix}
		M \,\, = \,\,
		\begin{bmatrix} 
			1  & 0 & 0 & 1 & 1 & 1 & \cdots & 1\\ 
			0  & 1 & 0 & 1 & x_{1,1} & x_{1,2} & \cdots & x_{1,n-4} \\ 
			0  & 0 & 1 & 1 & x_{2,1} & x_{2,2} & \cdots & x_{2,n-4}\\ 
		\end{bmatrix} .
	\end{equation}

	We note that $Y(3,n)$ is a subset of
	the familiar space $X(3,n) $
	of configurations~in linearly general position.
	Namely, $Y(3,n)$ is obtained from $X(3,n)$
	by removing the $\binom{n}{6}$~divisors of the form $q$.
	The space $X(3,n)$ was studied in detail in \cite{ABF}.
	Both $Y(3,n)$ and $X(3,n)$ are very affine of dimension $2n-8$, by (\ref{eq:generalmatrix}),
	and they have appeared in the physics literature \cite{CEGM,CUZ}.
	
	For $n=6$ and $n=7$, 
	we also employ the parametrization used in \cite{HKT, RSSS, RSS2},
	where the points lie  on a cuspidal cubic. 
	After an automorphism of $\P^2$, we can thus choose the matrix 
	\begin{equation}\label{eq:dimatrix}
		M \,\, = \,\,
		\begin{bmatrix} 
			1 & 1 & 1 & 1 & \cdots& 1\\ 
			d_1  & d_2 & d_3 & d_4 & \cdots & d_n \\ 
			d_1^3 & d_2^3 & d_3^3 & d_4^3 &\cdots & d_n^3\\ 
		\end{bmatrix} .
	\end{equation}
	The $3 \times 3$ minors and the conic conditions factor into linear forms  $d_i - d_j$, $d_i+d_j+d_k$ and $d_{i_1}+d_{i_2}+d_{i_3}+d_{i_4}+d_{i_5}+d_{i_6}$. These 
	are the roots of type $E_6$ and $E_7$. They define
	arrangements of $36$ hyperplanes in $\P^5$ and $63$ hyperplanes in $\P^6$. 
	The complements of these  arrangements are the moduli spaces for $6$ and $7$ points in general position together with a cuspidal cubic through them. Our moduli spaces
	are embedded in $\P^{39}$ and $\P^{134}$ by certain monomials in these linear forms 
	of degree $9$ and $7$, respectively. This yields the
	{\em Yoshida variety} \cite[eqn (6.2)]{RSS1} and the {\em G\"opel variety} \cite[\S 6]{RSSS},
	on which the Weyl groups act by permuting coordinates. 
	
	This paper is structured as follows. In Section \ref{sec2}, we perform a case study on $Y(3,5)$. We  will generalize this analysis to $Y(3,6)$ and $Y(3,7)$ in the rest of the paper. In Section~\ref{sec3} we consider the real del Pezzo surfaces $\mathcal{S}^\circ_n$,
	and we characterize their polygonal regions for $n=5,6$, as summarized in Theorem \ref{thm:subdivision}.
	In Section~\ref{sec4} we pass to the complex setting and we determine the Euler characteristics of both $\mathcal{S}^\circ_{n-1}$ and $Y(3,n)$ for $n=6,7$. Section~\ref{sec5} applies
	numerical methods for computing
	the Euler characteristic.
	Tropical likelihood degenerations
	\cite[Section 8]{ABF} are applied to
	verify theoretical results and offer new insights.
	In Section~\ref{sec6} we study the parametrizations (\ref{eq:dimatrix})
	of the del Pezzo moduli spaces via the reflection hyperplane arrangements 
	$E_6$ and $E_7$.  This centers around the
	Yoshida variety and the G\"opel variety.	
	Section~\ref{sec7} connects our findings with the 
	physics literature on scattering amplitudes.
	
	In Sections~\ref{sec8} and \ref{sec9}, we define the pezzotopes, we
	compute their f-vectors, and we present their perfect $u$-equations.
	The line arrangements in Figure \ref{fig:lines}
	and the graphs in Figure~\ref{fig:drei} reveal 
	remarkable structures that are of independent interest  for
	Weyl group combinatorics.
	Theorem \ref{thm:E7pezzo} 
	characterizes the $E_7$ pezzotope, and it offers
	a view from commutative algebra.
	
	In Section \ref{sec10} we derive
	the regions of $Y(3,n)$ from
	Grassmannians and their tropicalizations.
	The chirotopes in Theorem \ref{thm:chirotope}
	offer a fresh perspective on positive Grassmannians.
	The residues of the canonical form $\Omega$ for  Theorem \ref{thm:Y3nPosGeom}
	match the $10+5$ facets in Figure \ref{fig:zwei}.

	This article relies heavily on software and data. These materials
	are made available in the {\tt MathRepo} collection at MPI-MiS via
	\url{https://mathrepo.mis.mpg.de/positivedelPezzo}.

	\section{Blowing Up Four Points}
	\label{sec2}
	
	This section is a warm-up for the rest of the paper.
	It offers a case study of the surface~$\mathcal{S}_4$,
	which is obtained by blowing up $\P^2$ at four points in general position.
	We present an elementary proof that $\mathcal{S}_4^\circ$ is a {\em positive geometry}
	in the sense of  Arkani-Hamed, Bai and Lam~\cite{ABL}, i.e.~each region of the real 
	surface has a canonical differential form
	which satisfies the recursive axioms in \cite[Section 2]{Lam}.
	Such regions are {\em worldsheet associahedra} \cite[Figure~18]{ABHY}.
	
	Without loss of generality,
	the four points to be blown up are $(1:0:0)$, $(0:1:0)$, $(0:0:1)$ and $(1:1:1)$.
	These points span six lines, and the very affine surface $\mathcal{S}_4^\circ$ is
	the complement of these lines in $\P^2$. In coordinates, $\mathcal{S}_4^\circ$ consists
	of all points $(1:x:y)$ such that
	\begin{equation}
		\label{eq:ineqs}
		xy(1-x)(1-y)(y-x) \,\,\not= \,\,0.
	\end{equation}
	The following well-known fact highlights that this  example is of interest to many readers.
	
	\begin{prop}
		The very affine surface defined by the inequation (\ref{eq:ineqs}) plays various roles:
		$$ \mathcal{S}_4^\circ \,= \, Y(3,5) \, =  \,X(3,5) \, = \,X(2,5) \, = \, \mathcal{M}_{0,5}. $$
		The compact surface $\mathcal{S}_4$ is the tropical compactification of $\mathcal{M}_{0,5}$,
		as seen in \cite[Section~6.5]{MS}.
	\end{prop}
	
	We conclude that our del Pezzo surface $\mathcal{S}_4^\circ$ also plays the
	role of a moduli space in two different ways. It is the moduli space $\mathcal{M}_{0,5}$ 
	of five distinct labeled points on the line $\P^1$, and it is also the moduli space $Y(3,5)$ 
	of del Pezzo surfaces of degree four.
	This double casting of $Y(3,5)$ -- it acts as a del Pezzo surface and 
	as a moduli space of del~Pezzo surfaces --
	is our motivation for dedicating an entire section to  this seemingly minor actor.
	In the trinity of geometries, adulated in the Introduction,  the pictures for $Y(3,5)$ are as follows:
	\begin{itemize}
		\item The complex surface $\mathcal{S}_{\C}$ has Euler characteristic $2$.  \hfill \cite[Figure 1]{ABF} \vspace{-0.2cm}
		\item The real surface  $\mathcal{S}_{\R}$ consists of $12$ pentagons.        \hfill \cite[Figure 11]{Dev} \vspace{-0.2cm}
		\item The tropical surface $\mathcal{S}_{\T}$ is the cone over the Petersen graph. \hfill \cite[Figure 3]{RSS2} {\ }
	\end{itemize}
	
	\noindent The Euler characteristic of $\mathcal{S}_{\C}$ is the number of critical points of the log-likelihood function:
	\begin{equation}
		\label{eq:loglike1}
		L \,=\, s_1 \,{\rm log}(x) \,+\, s_2 \,{\rm log}(y) \,+\, s_3  \,{\rm log}(1-x) \,+\,
		s_4\, {\rm log}(1-y) \,+\, s_5 \,{\rm log}(y-x). \end{equation}
	See  \cite[Section 6]{CHKS} and \cite[Theorem 1]{H}.
	The coefficients $s_i$ are parameters. These are known as Mandelstam invariants in physics,
	and they represent the data in algebraic statistics; see~\cite{ST}. There are two critical points,
	obtained as the solutions of the likelihood equations
	\begin{equation}
		\label{eq:loglike2}
		\frac{s_1}{x} - \frac{s_3}{1-x} - \frac{s_5}{y-x} \,\,=\,\,
		\frac{s_2}{y} - \frac{s_4}{1-y} + \frac{s_5}{y-x} \,\,=\,\, 0. 
	\end{equation}
	We can solve this in radicals, thanks to the quadratic formula.
	If the $s_i$ are positive, then each of the two bounded regions of (\ref{eq:ineqs})
	contains one critical point. For the tropical setting, where the parameters $s_i$ are scalars in
	the valued field $\R \{ \! \{ \epsilon \} \! \}$, we refer to \cite[Section 7]{ABF}.
	
	When interpreting our surface as $\mathcal{M}_{0,5}$, it is customary to write its points as $2 \times 5$ matrices
	$$ M \,\,= \,\,\begin{bmatrix} 1 & 1 & 1 & 1 & 0 \\ 0 & x & y & 1 & 1 \end{bmatrix} .$$
	The $10$ maximal minors $p_{ij}$  of the matrix $M$ are the linear expressions in (\ref{eq:ineqs}) and (\ref{eq:loglike2}), namely
	$$ (p_{12},p_{13},p_{14},p_{15},p_{23},p_{24},p_{25},p_{34},p_{35},p_{45}) \, = \,
	(\,x,\,y,\,1,\,1,y-x,\,1-x,\,1,\,1-y,\,1,\,1\,). $$
	As in \cite[(6.7)]{ABHY} and \cite[(2.6)]{Bro}, we  use the 
	following five cross ratios as coordinates on~$\mathcal{M}_{0,5}$:
	$$ \begin{matrix}
		u_1 =   \frac{p_{25} p_{34}}{p_{35} p_{24}} = \frac{1-y}{1-x},\,
		u_2 =   \frac{p_{13} p_{45}}{p_{14} p_{35}} = y ,\,
		u_3 =   \frac{p_{24} p_{15}}{p_{14} p_{25}} = 1-x ,\,
		u_4 =   \frac{p_{35} p_{12}}{p_{13} p_{25}} = \frac{x}{y}, \,
		u_5 =   \frac{p_{14} p_{23}}{p_{13} p_{24}} = \frac{y-x}{(1-x)y}. \end{matrix}
	$$
	Following \cite[(3)]{AHLT},
	the $u$-coordinates satisfy the following system of quadratic equations:
	\begin{equation}
		\label{eq:ueqns}
		u_1 u_3 + u_2 \,=\,
		u_2 u_4 + u_3 \,=\,
		u_3 u_5 + u_4 \,=\,
		u_4 u_1 + u_5 \,=\,
		u_5 u_2 + u_1 \,=\, 1.
	\end{equation}
	These equations give a convenient embedding of $\mathcal{M}_{0,5}$ as a closed subvariety of
	the torus~$(\C^*)^5$.
	
	To appreciate the beauty of the Laurent polynomial ideal in (\ref{eq:ueqns}), consider its real solutions.
	Among the $64 = 2^5$ possible sign patterns for $(u_1,u_2,u_3,u_4,u_5)$, precisely $12$ are realized:
	\begin{equation}
		\label{eq:usigns}
		\begin{small}
			\begin{matrix}
				(+ + + + +) & (- + + + +) & (+ - + + +) & (+ + - + +) & (+ + + - +) & (+ + + + -) \\
				(- - - - -)  & ( + - + - +) & (+ + - + -) & (- + + - +) & (+ - + + -) & (- + - + +) \\
			\end{matrix}
		\end{small}
	\end{equation}
	These sign vectors label the $12$ regions in $\mathcal{S}_\R$, all of which are the same up to symmetry.
	We focus on the positive region $(+++++)$, where the following tropical equations are valid:
	$$
	{\rm min}(u_1+u_3,u_2) =
	{\rm min}(u_2+u_4,u_3) =
	{\rm min}(u_3+u_5,u_4) =
	{\rm min}(u_4+u_1,u_5) =
	{\rm min}(u_5+u_2,u_1) = 0.
	$$
	The solution to these equations is the cone over a pentagon.
	This is one of the $12$ pentagons in the Petersen graph of $\mathcal{S}_\T$.
	To see the pentagon classically in $\mathcal{S}_\R$, we note that (\ref{eq:ueqns})
	implies $0 \leq u_1,u_2,u_3,u_4,u_5 \leq 1$.
	If $u_i$ is $0$, then both neighboring variables $u_{i-1}$
	and $ u_{i+1}$ must be~$1$, and the remaining two variables
	are nonnegative and sum to $1$. This reveals the pentagon.
	In other words, the set of nonnegative solutions to (\ref{eq:ueqns}) is a curvy pentagon in $(\R_+)^5 \subset (\C^*)^5$.
	
	We are now prepared to prove that this curvy pentagon is a positive geometry by showing that the canonical form 
	for $\mathcal{M}_{0,5} $ is the following $2$-form:
	\begin{equation}
		\label{eq:Omega} \Omega \,= \,
		{{\rm dlog} \bigl( \frac{u_2 u_5}{u_1}   \bigr) \,\wedge \,{\rm dlog} \bigl( \frac{u_4}{u_3 u_5} \bigr) \,\, = \,\,
			\rm dlog} \bigl( \frac{y-x}{1-y}  \bigr) \,\wedge \,{\rm dlog} \bigl( \frac{x}{1-y}  \bigr) \,\,=\,\,
		\frac{{\rm d}x \,{\rm d}y} {x (x-y) (1-y)}.
	\end{equation}
	On the line $u_2 = 0$ we have
	$u_1 = u_3=1$ and $u_4 + u_5 = 1$.
	The residue of $\Omega$ on that line is ${\rm dlog}(u_4/u_5) = {\rm dlog}(u_4/(1-u_4))$,
	which is the canonical form of the line segment.
	A similar calculation works for the residues at the other four boundaries $u_i = 0$.
	Since line segments are positive geometries, we thus verify the
	recursive axioms in \cite[Section 2.1]{ABL}.
	By symmetry, each curvy pentagon endows the surface $\mathcal{M}_{0,5}$
	with  the structure of a positive geometry.
	
	Now, we compute the {\em scattering amplitude} of $\mathcal{M}_{0,5}$  following \cite[(2.12)]{CHY}.  To integrate  $\Omega$, we write $\mathcal{H}_L$ for the Hessian of the log-likelihood function
	(\ref{eq:loglike1}), and we sum 
	$\,\text{det}(\mathcal{H}_L^{-1} )\,x^{-2} (x-y)^{-2} (1-y)^{-2}\,$ over the two solutions of (\ref{eq:loglike2}). The result
	is the amplitude
	\begin{equation}
		\label{eq:smallamplitude}
		\frac{1}{s_1 \,s_4} \,+\, \frac{1}{s_4 \,(s_3 + s_4 +s_5)} \,+\, \frac{1}{(s_3 + s_4 + s_5)\, s_5} \,+ \,
		\frac{1}{s_5 \,(s_1+s_2+s_5)} \,+\, \frac{1}{(s_1 + s_2 + s_5)\, s_1}. 
	\end{equation}
	
	Our discussion illustrates the objective of this article: extending
	the study of $\mathcal{S}_\C$, $\mathcal{S}_\R$, $\mathcal{S}_\T$ to del Pezzo surfaces
	of degree $9-n$ and their moduli spaces $Y(3,n)$ for larger values of~$n$.
	
	\section{Polygons}
	\label{sec3}
	
	In this section we examine polygonal decompositions of real del Pezzo surfaces. We denote the blow up of $\P^2$ at $n$ general 
	real points, as described in the introduction, by the {\em del Pezzo surface}
	$\mathcal{S}_n$.
	The anticanonical divisor of such a surface is very ample for $n\leq 6$. Its sections are
	cubics that vanish at the $n$ points in $\P^2$, and these give an embedding of $\mathcal{S}_n$ into $\P^{9-n}$.
	With this notation, $\mathcal{S}_6$ is a cubic surface in $\P^3$.
	The surface $\mathcal{S}_5$ lives in $\P^4$, where it is an intersection of two quadrics. 
	For $n=4,5,6$, the surface contains $10,16,27$ straight lines.
	For $n=7$, the anticanonical map is $2$-to-$1$
	from $\mathcal{S}_7$ onto $\P^2$. Its branch locus is a quartic curve~$Q$.
	There are $56$ exceptional curves on $\mathcal{S}_7$, and these
	are mapped  onto the $28$ bitangents of $Q$. 
	
	The very affine surfaces  $\mathcal{S}_n^\circ$ result from removing the
	$10,16,27,56$ lines from $\mathcal{S}_n$. The lines are denoted following \cite{RSS2}. We write $F_{ij}$ for the line in $\R\P^2$ spanned by the points $i$ and $j$, and $E_i$ for the exceptional divisor over 
	the point $i$. For $n\geq 6$, we write $G_I$ for the unique conic through the five
	points labeled by $\{1,2\ldots,n\} \setminus I$. For $n=7$, we denote by $H_i$ the unique cubic that passes through all seven points and has a node at point $i$.
	
	We shall study the
	connected components of $\mathcal{S}_n^\circ$, here called {\em polygons}, 
	with focus on their shapes and numbers.
	Each polygon is bounded by a subset of the lines. The {\em face vector} $f = (v,e,p)$ 
	records the number $v$ of vertices, the number $e$ of edges
	and the number $p$ of polygons in this polygonal subdivision of $\mathcal{S}_n$.
	The subdivision for $n=4$ consists of $12$ pentagons, and it has
	$f = (15,30,12)$. We summarize the situation in the following theorem.
	Here, a del Pezzo surface is called {\em general}
	if  every point is contained in at most two lines. 
	
	\begin{thm} \label{thm:subdivision}
		The subdivision of any general real del Pezzo surface $\mathcal{S}_n$
		has $36$, $130$, $806$ polygons for $n=5,6,7$.
		For $n=5$, it has $20$ quadrilaterals and $16$ pentagons, and the
		f-vector is $(40,80,36)$. Each line is incident to
		$5$ quadrilaterals and $5$ pentagons.
		For $n=6$, there are $10$ triangles, $90$ quadrilaterals,
		and $30$ pentagons, and the f-vector is
		$(135,270,130)$.
		Twelve of the $27$ lines on $\mathcal{S}_6$
		are disjoint from the triangles, and these form a 
		Schl\"afli double-six. Each of the other $15$ lines is
		incident to  $2$ triangles, $12$ quadrilaterals and $6$ pentagons. 
	\end{thm}

	\begin{figure}[ht] 
		\centering
		\includegraphics[scale=0.46]{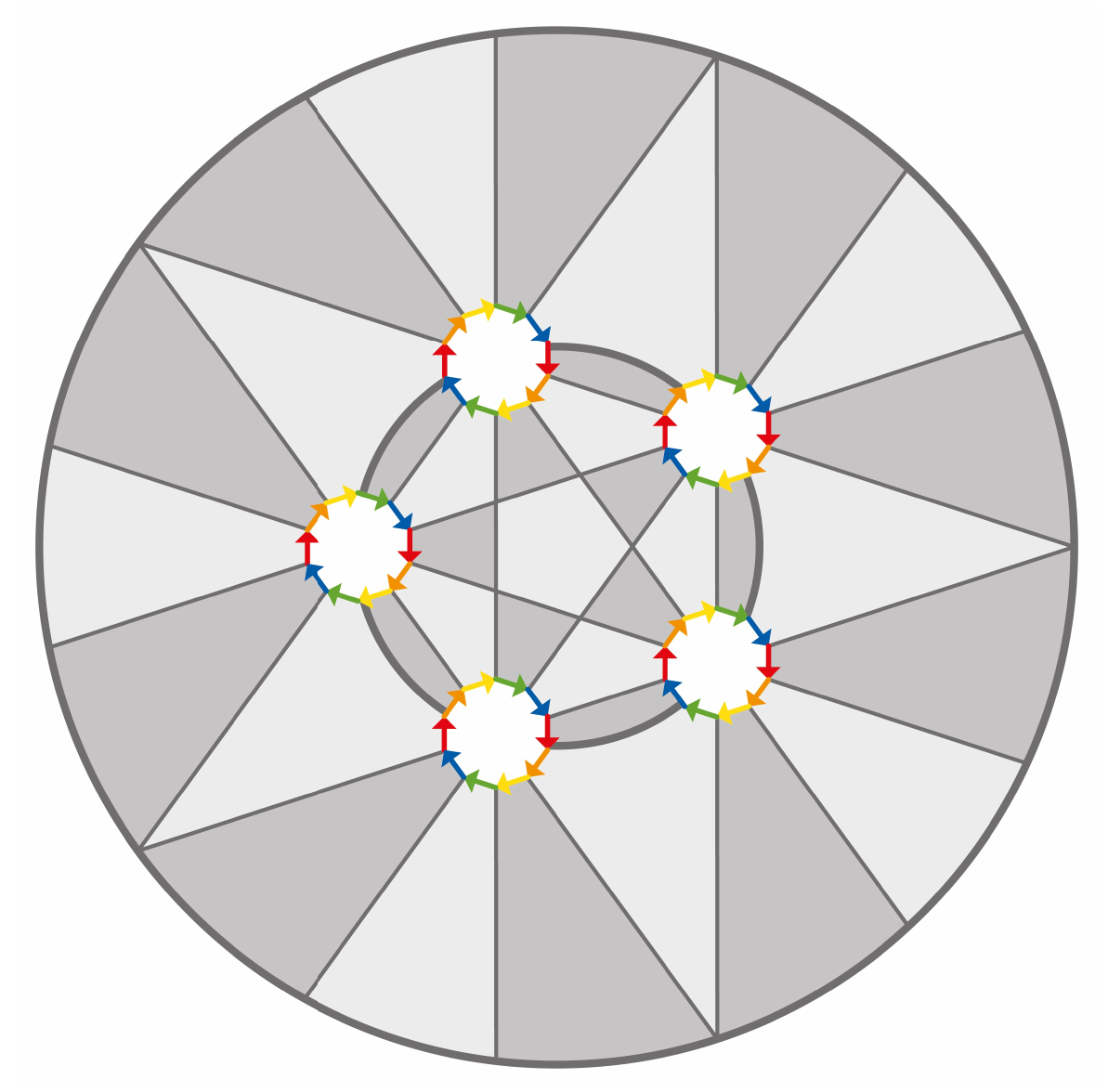} 
		\vspace{-0.21cm}	
		\caption{Subdivision of the del Pezzo surface $\mathcal{S}_5$ into $20$ quadrilaterals (dark) and $16$ pentagons (light).
			Each blown up point is replaced by a decagon with opposite sides identified.
			\label{fig:eins}}
		\vspace{-0.2cm}		
	\end{figure}

	\begin{proof} 
		Consider a polygonal subdivision of a closed surface,
		with $v$ vertices, $e$ edges and $p$ polygons. The Euler
		characteristic of the surface satisfies $\chi = v - e + p$.
		If the surface is $\mathcal{S}_n$, i.e.~the blow up of $\R\P^2$ at $n$ general  points,
		then $\chi = 1-n$. The number of $(-1)$-curves on~$\mathcal{S}_n$
		is $10, 16, 27, 56$ for $n=4,5,6,7$.
		We call these “lines”, in spite of 
		$\R\P^1$ being a~circle.
		
		Each line is incident to the same number of other lines, 
		and that number is $3, 5, 10, 28$. For $n \leq 6$, the intersection of two lines on $\mathcal{S}_n$
		is either empty or a single point.
		For $n=7$, however, each line meets one of the lines twice, and it meets $27$ other lines in one~point.
		
		It follows that the
		number of edges in the  subdivision is the
		product of the number of lines and the number of incident lines.
		We hence find
		that $e$ equals $30, 80, 270, 1624$.
		The number of vertices is half the number of
		edges, i.e. $v = e/2$ equals $15, 40, 135, 812$.
		From this data we compute the number $p$ of
		polygons in each case with the formula
		$$ p \,=\, \frac{e}{2}+\chi. $$
		Therefore, the number of polygons equals  $p= 12, 36, 130, 806$ for $n=4,5,6,7$.
		
		To get more refined information, we discuss each case separately.
		For $n=4$, this was done  in Section \ref{sec2} where we referred to  work of
		Devadoss~\cite{Dev} on $\mathcal{M}_{0,n}$ over $\R$.
		For $n=5$, we offer Figure \ref{fig:eins}. After a projective transformation, we may assume that
		the five points are in convex position, so the conic $G$ that passes through these points is an ellipse.
		The arrangement of ten lines $F_{12},F_{13},\ldots,F_{45}$ and the conic $G$ divides 
		$\R \P^2$ into $36$ regions, and we color them by light gray and dark gray
		in the checkerboard pattern shown in Figure \ref{fig:eins}.
		The blow up replaces the point $i$ by  a projective line  $E_i$, which in the figure is obtained 
		from a decagon by double-covering. Arrows with the same color are identified, and they
		create new edges in the gray regions. Considering these new colorful edges, 
		each of the $16$ light gray regions is a pentagon, and each of the $20$ dark gray
		regions is a quadrilateral.
		
		We now turn to $n=6$. Rather than pointing to a picture, we leave the job to a  computer.
		The input is a random  element in $Y(3,6)$, that is, a random configuration of six points in $\R \P^2$.
		We identify the regions in $\R \P^2$ created by removing the
		$15$ lines $F_{12}, F_{13}, \ldots, F_{56}$ and the six conics $G_1,G_2,G_3,G_4,G_5,G_6$.
		For each region that is incident to a
		point $i$, we identify the $10$ new colorful edges created by the circle~$E_i$.
		The output of the computation is a list of tuples of lines, each delineating one of the $130$ polygons.
		An explicit list is shown in Example \ref{ex:109030}.
		From this data, we can read off interesting properties.
		Recall that the cubic surface $\mathcal{S}_6$ has
		$36$ double-sixes. We found that there is a unique double-six that is disjoint from the $10$ triangular regions. 
		This is identified in the data below.
	\end{proof}
	
	\begin{example} \label{ex:109030}
		Fix the cubic surface $\mathcal{S}_6$ obtained from $\P^2$ by blowing up the six points in
		$$ M \,\, = \,\,\begin{bmatrix}
			\,   -2  &  \phantom{-} 24  &  \,16  &  \phantom{-} 27  & \, 14  & \, 1  \,\,\\
			\,  -25  &   \phantom{-}3  &  \,28  &   \phantom{-}13  &   \,5  &  \,7 \,\, \\
			\,  -26  &  -4  &   \,1  &  -14  &  \, 9  & \, 6  \,\,
		\end{bmatrix}. $$
		The resulting subdivision of $\mathcal{S}_6$ has the ten triangles
		\begin{small}  
			$$
			\begin{matrix}
				E_4 F_{24} G_2  ,   E_4 F_{34} G_3  ,     E_5 F_{15} G_1  ,     E_5 F_{25} G_2  ,     E_6 F_{16} G_1  ,
				E_6 F_{36} G_3  ,   F_{14} F_{25} F_{36}  ,   F_{14} F_{26} F_{35}  ,   F_{15} F_{26} F_{34}  ,   F_{16} F_{24} F_{35}  .
			\end{matrix}$$
		\end{small}
		The $90$ quadrilaterals are
		\begin{small}
			$$
			\begin{matrix}
				E_1 E_5 G_3 G_4  ,      E_1 E_5 G_3 G_6  ,     E_1 E_6 F_{16} G_5  ,      E_1 E_6 G_2 G_4  , 
				E_1 E_6 G_2 G_5  ,      E_1 E_6 G_3 G_4  ,     E_1 F_{12} F_{14} F_{36}  ,    E_1 F_{12} F_{14} G_1  , \\
				E_1 F_{12} F_{15} F_{36}  ,   E_1 F_{13} F_{14} G_1  ,   E_1 F_{13} F_{16} F_{25}  ,    E_1 F_{15} F_{36} G_6  , 
				E_1 F_{16} F_{25} G_5  ,    E_2 E_4 G_1 G_5  ,     E_2 E_4 G_1 G_6  ,       E_2 E_4 G_3 G_6  , \\
				E_2 E_5 F_{25} G_4  ,     E_2 E_5 G_1 G_6  ,     E_2 E_5 G_3 G_4  ,       E_2 E_5 G_3 G_6  , 
				E_2 F_{12} F_{25} F_{34}  ,   E_2 F_{12} F_{26} F_{34}  ,  E_2 F_{12} F_{26} G_2  ,     E_2 F_{15} F_{23} F_{24}  , \\
				E_2 F_{15} F_{24} G_5  ,    E_2 F_{23} F_{26} G_2  ,   E_2 F_{25} F_{34} G_4  ,     E_3 E_4 G_1 G_5  , 
				E_3 E_4 G_1 G_6  ,      E_3 E_4 G_2 G_5  ,      E_3 F_{24} F_{36} G_4  ,     E_4 F_{14} F_{23} F_{45}  , \\
				E_3 F_{13} F_{24} F_{36}  ,   E_3 F_{13} F_{35} G_3  ,   E_3 F_{16} F_{23} F_{34}  ,    E_3 F_{16} F_{23} F_{35}  , 
				E_3 F_{16} F_{34} G_6  ,    E_3 F_{23} F_{35} G_3  ,   E_3 E_6 G_2 G_4  ,       E_3 E_6 G_2 G_5  , \\
				E_4 F_{14} F_{23} F_{46}  ,   E_4 F_{14} F_{45} G_4  ,   E_4 F_{14} F_{46} G_4  ,     E_4 F_{24} F_{45} G_4  , 
				E_4 F_{34} F_{46} G_4  ,    E_5 F_{12} F_{35} F_{45}  ,  E_5 F_{12} F_{35} F_{56}  ,\\  
				E_5 F_{12} F_{45} G_2  ,  
				E_5 F_{15} F_{56} G_5  ,    E_5 F_{35} F_{45} G_5  ,   E_5 F_{35} F_{56} G_5  ,   E_6 F_{13} F_{26} F_{46}  , 
				E_6 F_{13} F_{26} F_{56}  ,   E_6 F_{13} F_{56} G_1  , \\  E_6 F_{26} F_{46} G_6  ,   E_6 F_{26} F_{56} G_6  , 
				E_6 F_{36} F_{46} G_6  ,    F_{12} F_{14} F_{56} G_1  ,  F_{12} F_{15} F_{36} F_{46}  ,  F_{12} F_{25} F_{34} F_{46}  , 
				F_{12} F_{26} F_{35} F_{45}  ,  \\ F_{12} F_{26} F_{45} G_2  ,  F_{13} F_{14} F_{26} F_{56}  ,  F_{13} F_{14} F_{56} G_1  , 
				F_{13} F_{16} F_{24} F_{45}  ,  F_{13} F_{16} F_{25} F_{45}  ,  F_{13} F_{24} F_{36} F_{45}  ,  F_{13} F_{25} F_{36} F_{45}  , \\ 
				F_{13} F_{35} F_{46} G_3  ,   F_{14} F_{23} F_{35} F_{46}  ,   F_{14} F_{25} F_{46} G_4  ,   F_{15} F_{23} F_{24} F_{56}  , 
				F_{15} F_{23} F_{34} F_{56}  ,  F_{15} F_{24} F_{56} G_5  ,   F_{15} F_{26} F_{46} G_6  , \\  F_{15} F_{36} F_{46} G_6  , 
				F_{16} F_{23} F_{34} F_{56}  ,   F_{16} F_{25} F_{45} G_5  ,    F_{16} F_{34} F_{56} G_6  ,   F_{23} F_{26} F_{45} G_2  , 
				F_{23} F_{35} F_{46} G_3  ,    F_{24} F_{35} F_{56} G_5  ,   \\ \!\! F_{24} F_{36} F_{45} G_4  ,   F_{25} F_{34} F_{46} G_4  , 
				E_3 F_{24} G_2 G_4  ,     E_2 F_{34} G_3 G_4  ,   E_2 F_{15} G_1 G_5  ,     E_1 F_{25} G_2 G_5  , 
				E_3 F_{16} G_1 G_6  ,     E_1 F_{36} G_3 G_6  .
			\end{matrix}
			$$
		\end{small}
		Finally, the $30$ pentagons are
		\begin{small}
			$$
			\begin{matrix}
				E_1 E_5 F_{12} F_{15} G_1  ,      E_1 E_5 F_{15} G_1 G_6  ,       E_1 E_5 F_{25} G_2 G_4  , 
				E_1 E_6 F_{13} F_{16} G_1  ,      E_1 F_{13} F_{14} F_{25} F_{36}  ,    E_2 E_4 F_{23} F_{24} G_2  ,  \\
				E_2 E_4 F_{24} G_2 G_5  ,       E_2 E_5 F_{12} F_{25} G_2  ,  E_2 F_{15} F_{23} F_{26} F_{34}  , 
				E_3 E_4 F_{23} F_{34} G_3  ,      E_3 E_4 F_{34} G_3 G_6  ,       E_3 E_6 F_{13} F_{36} G_3  ,  \\
				E_3 E_6 F_{16} G_1 G_5  ,       E_3 E_6 F_{36} G_3 G_4  ,       E_3 F_{13} F_{16} F_{24} F_{35}  , 
				E_4 F_{23} F_{24} F_{45} G_2  ,     E_4 F_{23} F_{34} F_{46} G_3  ,     E_5 F_{12} F_{15} F_{56} G_1  , \\
				E_5 F_{25} F_{45} G_2 G_5  ,      E_6 F_{13} F_{36} F_{46} G_3  ,     E_6 F_{16} F_{56} G_1 G_6  , 
				F_{12} F_{14} F_{25} F_{36} F_{46}  ,   F_{12} F_{14} F_{26} F_{35} F_{56}  ,   F_{12} F_{15} F_{26} F_{34} F_{46}  , \\
				\!\! F_{13} F_{14} F_{26} F_{35} F_{46}  ,   F_{14} F_{23} F_{26} F_{35} F_{45}  ,  F_{14} F_{25} F_{36} F_{45} G_4  , 
				F_{15} F_{26} F_{34} F_{56} G_6  ,    F_{16} F_{23} F_{24} F_{35} F_{56}  ,   F_{16} F_{24} F_{35} F_{45} G_5  .
			\end{matrix}
			$$
		\end{small} 
		We note that the following $12$ symbols do not appear among the triangles:
		$$  \qquad \qquad \qquad
		\begin{matrix}       E_1 & E_2 & E_3 & F_{45} & F_{46} & F_{56} \\
			F_{23} & F_{13} & F_{12} & G_6 & G_5 & G_4 
		\end{matrix} \qquad \longrightarrow \qquad
		\begin{matrix}  5 &  4 & 6 & 5 &  5 &  5  \\  7 & 6 & 6 & 4 & 4 &  3 \end{matrix}
		$$
		These  $12 $ lines form a double-six.
		On the right  we list the numbers of pentagons bounded by each line.
		Each of the other $15$ lines bounds $2$ triangles, $12$ quadrilaterals
		and $6$ pentagons.
	\end{example}
	
	\begin{rmk} \label{rmk:cremona1}
		The Weyl group $W(E_n) $ acts on our data for $n=6,7$. Each group
		is generated by permuting the labels $1,2,\ldots,n$ of $E_\cdot,F_{\cdot \cdot},G_{\cdot \cdot},H_\cdot$
		plus one additional involution that represents the Cremona transformation of $\P^2$ centered
		at the triangle $E_1E_2E_3$.  For $n=6$ this Cremona involution equals
		$(E_1 F_{23}) (E_2 F_{13}) (E_3 F_{12}) (G_4 F_{56}) (G_5 F_{46}) (G_6 F_{45})$.
		For $n=7$ it is \begin{small}
			$\!(E_1 F_{23}) (E_2 F_{13}) (E_3 F_{12})
			(G_{12} H_3) (G_{13} H_2) (G_{23} H_1)
			(\!F_{45} G_{67})
			(\!F_{46} G_{57})
			(\!F_{47} G_{56})
			(\!F_{56} G_{47})
			(\!F_{57} G_{46})
			(\!F_{67} G_{45})
			$. \end{small}
	\end{rmk}
	
	\begin{rmk}
		Each triangle on the cubic surface $\mathcal{S}_6$ shares its three edges
		with pentagons, and its vertices are adjacent to quadrilaterals. For instance,
		in the list above, the triangle $E_4 F_{24} G_2$ shares its edges with the
		pentagons
		$E_2 E_4 F_{23} F_{24} G_2$,
		$E_2 E_4 F_{24} G_2 G_5$,
		$E_4 F_{23}  F_{24} F_{45} G_2$,
		and its vertices are adjacent to the
		quadrilaterals
		$E_3 E_4  G_2  G_5$,
		$E_3 F_{24} G_2 G_4 $,
		$E_4 F_{24} F_{45} G_4$.
		If we move around in the moduli space $Y(3,6)$, then
		the triangle gets inverted when passing through an Eckardt divisor.
		This inversion replaces the pentagons with quadrilaterals and vice versa.
		In particular, the numbers $10, 90$ and $30$ of triangles, quadrilaterals and pentagons
		remain the same when we cross the Eckardt divisor within a
		connected component of $Y(3,6)$.
	\end{rmk}
	
	\begin{example}[Clebsch cubic]
		The most prominent cubic surface is
		the Clebsch cubic. This is obtained by blowing up $\P^2$ at the vertices of a regular
		pentagon and its midpoint, and it has $10$ Eckhardt points. Its subdivision
		consists of $120$ quadrilaterals. Nearby in $Y(3,6)$, as the Eckhardt points become triangles,
		their three adjacent quadrilaterals become pentagons.
	\end{example}
	
	To find the subdivisions of $\mathcal{S}_6$ and $\mathcal{S}_7$, we fix
	ternary forms for all exceptional curves apart from the $E_i$. We then compute the
	stratification of $\R^3 $ that is given by signs $+$ or $-$ of these ternary forms.
	Each region in $\R^3$ is thus labeled by a sign vector of length $21$ or length $49$
	respectively, for $n=6,7$. For the cubic surface $(n=6)$, this computation
	yields $260$ sign vectors. Each of them labels a unique connected component in $\R^3$.
	The quotient map from $\R^3$ to $\P^2$ thus creates $130$ polygons, and the 
	only remaining task is to weave in the exceptional curves $E_i$.
	This is done by  hand, using a careful local analysis.
	
	We now turn to the case $n=7$. New phenomena arise
	because seven cubic curves are needed for subdividing $\P^2$.
	The polygonal subdivision of the surface
	$\mathcal{S}_7$ is created by the $7$ exceptional divisors $E_i$,
	the $21$ lines $F_{ij}$, the $21$ conics $G_{ij}$, and the $7$ cubics~$H_j$.
	While removal of a single $F_{ij}$ leaves $\P^2$ connected, and
	removal of $G_{ij}$ creates two connected components, it is important to note that
	removal of $H_i$ divides $\P^2$ into three connected components.
	The analogous computation in $\R^3$ as above for $n=6$ now yields $1596$ distinct sign vectors.
	This is less than the number $1612 = 2 \cdot 806$, which is expected from Theorem~\ref{thm:subdivision}.
	We believe that this discrepancy arises from the
	tripartition of $\P^2$  by nodal cubics.
	This deserves further~study.

	\section{Euler Characteristic}
	\label{sec4}
	
	We now turn to the complex geometry of very affine varieties.
	The Euler characteristic is the key invariant.
	For statisticians, it yields the  maximum likelihood (ML) degree \cite{CHKS}.
	For physicists, it reveals the number of master integrals of a Feynman diagram \cite[Section~3]{BBKP}.
	
	The Euler characteristic of the configuration space $X(3,n)$ equals $26, 1272, 188112$
	for $n= 6,7,8$. This result is shown in \cite{CUZ} via soft limits and in \cite{ABF} via stratified fibrations. In what follows,
	we determine the analogous numbers for moduli of del Pezzo surfaces, via the method of stratified fibrations. 
	After the completion of this work, we learned that
		 the same results had been proven in \cite{Bergvall, BG} via cohomological methods.
	We know from Section \ref{sec2}
	that $Y(3,5) = X(3,5)$ has Euler characteristic~$2$.

	\begin{thm} \label{thm:euler}
		The Euler characteristic of $Y(3,n)$ equals
		$32, 3600$ for $n=6,7$. The Euler characteristic of $Y(3,8)$ is bounded below by $4884387$.
	\end{thm}
	
	In this section we offer formal proofs for $n=6,7$. These build on \cite[Section 2]{ABF}.
	The case $n=8$ will be treated in Section~\ref{sec5}. For a detailed explanation see Experiment~\ref{expe:sascha}.
	
\begin{rmk}
    Olof Bergvall has computed the Euler characteristic of $Y(3,8)$ to be $4884480$ via a finite field method. This computation and its proof will appear in his upcoming paper.
\end{rmk}

\noindent	We begin by calculating the Euler characteristics of very affine del Pezzo surfaces.
	
	\begin{lem} \label{lem:1690}The general surfaces $\mathcal{S}_5^\circ$ 
		and $\mathcal{S}_6^\circ$ have Euler characteristic 
		$16$ and $90$ respectively. If the cubic surface has $\ell$ Eckhardt points, then
		its Euler characteristic drops to $90- \ell$.
	\end{lem}
	
	\begin{proof}
		The complex projective plane $\C \P^2$ has Euler characteristic $\chi=3$.
		That number increases by $1$ whenever we blow up one point.
		Hence the compact surface $\mathcal{S}_n$ has $\chi=3+n$.
		We pass to $\mathcal{S}_n^\circ$ by inclusion-exclusion.
		The Riemann sphere $\C \P^1$ has $\chi=2$, so we subtract
		$2$ for each line, and we add $1$ for each intersection point of two lines.
		For $n=5$, we remove $16$ lines and we add in $40$ points,
		so we get $\chi = 8- 16 \cdot 2+40 =16$.
		For $n=6$, we remove $27$ lines and we add in $135$ points, so 
		we get $\chi = 9- 27 \cdot 2 +135= 90$. We subtract $1$ for each Eckhardt point, i.e.~each point where three lines meet.
	\end{proof}
	
	\begin{proof}[Proof of  Theorem \ref{thm:euler}]
		We consider the map $\pi_n: Y(3,n) \rightarrow Y(3,n-1)$ given by deleting the last point.
		For $n \leq 6$, this map is a fibration, with each fiber isomorphic to $\mathcal{S}_{n-1}^\circ$.
		Euler characteristic is multiplicative under fibrations. We know from
		Lemma~\ref{lem:1690} that $\mathcal{S}_5^\circ$ has $\chi = 16$.
		This implies that the Euler characteristic of  $Y(3,6)$ is equal to $16 \cdot 2 = 32$.
		
		We now turn to $n=7$. The map $\pi_7$ is not a fibration but only a stratified
		fibration. This is analogous to \cite[Section 5]{ABF}, and we apply the same technique.
		Namely, we identify all strata in $Y(3,6)$ and the fiber over each stratum,
		we compute the Euler characteristic for each piece,  we
		analyze the poset structure of the stratification, and we  use \cite[Lemma~2.3]{ABF}.

		\begin{table}[h]
			\begin{small}
				\begin{tabular}{c|l|c}
					\hline
					EckPts & \makecell[c]{Representative Configuration} & Strata \\ \hline
					1 & $(16)(25)(34)$ & 45\\ \hline
					2 & $(16)(25)(34),(15)(26)(34)$ & 270\\ \hline
					3 & $(16)(25)(34),(15)(24)(36),(14)(26)(35)$ & 240 \\ \hline
					4 & $(16)(25)(34),(15)(24)(36),(14)(26)(35),(15)(26)(34)$ & 720 \\ \hline
					6 & $(16)(25)(34),(15)(26)(34),(12)(35)(46),(12)(36)(45),\!(13)(24)(56),\!(14)(23)(56)\!\!$ & 540 \\ \hline
					9 & $(16)(25)(34),(14)(26)(35),(15)(24)(36),Q_{13}, Q_{46},Q_{32},Q_{65},Q_{21},Q_{54}$ & 40 \\ \hline
					10 & $(16)(25)(34),(15)(24)(36),(14)(26)(35),\!(15)(26)(34),Q_{14},\! Q_{16},\! Q_{24},\! Q_{36},\! Q_{25},\! Q_{35\!\!}$ & 216\\ \hline
					18 & \makecell[l]{$(14)(25)(36),(14)(26)(35),(15)(24)(36),\!(15)(26)(34),\!(16)(24)(35), \! (16)(25)(34),\!\!
						$\\$
						Q_{12},Q_{13},Q_{21},Q_{23},Q_{31},Q_{32},Q_{45},Q_{46},Q_{54},Q_{56},Q_{64},Q_{65}$} & 40\\ \hline
				\end{tabular}
				\caption{All $2111$ strata in $Y(3,6)$, grouped by number of Eckardt points. Data from
					\cite[Lemma 10]{BK} were verified in \texttt{Julia} using \texttt{OSCAR}.
					Column 3 lists the number of strata per type.} 
				\label{table: strata of mod3}
			\end{small}	
		\end{table}
		
		The stratification of $Y(3,6)$ is given by the labels of Eckardt points on the cubic surface,
		i.e.~triples of concurrent lines. In the blow-up picture,
		we see  two types of Eckardt points:
		\begin{itemize}
			\item The $15$ triples of lines $F_{ij} , F_{kl}, F_{mn}$, which we denote by $(ij)(kl)(mn)$. \vspace{-0.15cm}
			\item The $30$ triples $E_i F_{ij} G_j$, where line $F_{ij}$ is tangent to conic $G_j$ at the point $i$ in $\P^2$. 
		\end{itemize}
		These $45$ triples specify the tritangent planes of the general cubic surface.
		The Weyl group $W(E_6)$ acts transitively on this set of $45$ triples.
		Following \cite{BK}, Table \ref{table: strata of mod3} summarizes
		the possible configurations of Eckardt points. They also show the poset structure on the strata. 
		
		Table \ref{table: strata of mod3} is the del Pezzo variant of \cite[Table 3]{ABF}.
		We shall compute $\chi(Y(3,7))$ by using \cite[Lemma 2.3]{ABF}, which we rewrite just like in \cite[eqn (16)]{ABF}.
		In our setting, this takes the form
		\begin{equation}
			\label{eq:chiformula}
			\chi(Y(3,7)) \,\,= \,\,\chi(Y(3,6))\cdot\chi(F_{Y(3,6)}) \,-\, \sum_{S\in\mathbb{S}}\chi(S) \cdot \rho(S),
		\end{equation}
		where $\mathbb{S}$ is the poset of $2111$ strata, and
		$\,\rho(S) = \sum_{S'\in \mathbb{S}, S' \supseteq S} \mu(S,S')\cdot(\chi(F_{Y(3,6)})-\chi(F_{S'}))$.
		By Lemma \ref{lem:1690}, the generic fiber $F_{Y(3,6)}$ has $ \chi = 90$,
		so the product on the left is $ 32 \cdot 90 = 2880$.
		
		To prove Theorem \ref{thm:euler}, we must show that $\sum_{S\in\mathbb{S}}\chi(S) \cdot \rho(S) = -720$.
		By analyzing the poset of strata, we find that  the analog of \cite[Theorem 5.4]{ABF} holds
		in our case: $\rho(T) = 1$ for the top stratum  with one Eckardt point, and
		$\rho(S) = 0$ for all lower-dimensional strata~$S$.
		There are $45$ top strata $T$ all isomorphic. It remains to show that
		$\chi(T) =  -720/45 = -16$.
		
		Consider the map $\pi_6:T\to Y(3,5)$ which deletes point $6$.
		We stratify $Y(3,5)$ as follows:
		\begin{enumerate}
			\item Type $a$. General position. \vspace{-0.15cm}
			\item Type $b_1$. The three points $ \overline{12} \cap \overline{34}, \,\overline{14} \cap \overline{23}$ 
			and $5$ are	collinear. \vspace{-0.15cm}
			\item Type $b_2$. The three points $ \overline{12} \cap \overline{34}, \,\overline{13} \cap \overline{24}$ 
			and $5$ are	collinear. \vspace{-0.15cm}
			\item Type $c$. Setting $S = \overline{12} \cap \overline{34}$, the line $\overline{5S}$ is tangent to the conic through 
			the five points.
		\end{enumerate}
		The poset structure is simple:
		$b_1,b_2$ and $c$ are pairwise disjoint and contained in $a$.
		For the fiber to have Eckardt point $(12)(34)(56)$, point $6$ must be on the line $\overline{5S}$
		but retain general linear and quadratic position with the five points.
		We see that the fiber over $a$ is $F_a = \P^1\setminus \{7 \text{ points}\}$ and
		hence $\chi(F_a) = -5$. Similarly, for  $u \in \{b_1,b_2,c\}$,
		we have $F_u = \P^1\setminus \{6 \text{ points}\}$ and hence $\chi(F_u) = -4$.
		For the strata in  $Y(3,5) = \mathcal{M}_{0,5}$, we find
		$\chi(a) = 2$ and $ \chi(u) = -2$ for $u \in \{b_1,b_2,c\}$. 
		Using \cite[Lemma 2.3]{ABF} again,
		we conclude that $\chi(T) =  2 \cdot (-5) \,-\, 3  \cdot (-2)\cdot 1 \cdot (-5 - (-4)) = -16$.
	\end{proof}
	
	We now examine the set-theoretic difference between $X(3,n)$ and $Y(3,n)$.
	For $n=6$, this difference consists of configurations of
	six points in $\P^2$ in general linear position but lying on a conic. 
	This is isomorphic to $\M_{0,6}$, which is reflected in the Euler characteristics:
	\begin{equation}
		\label{eq:coconic6} \chi(X(3,6)\setminus Y(3,6)) \,\,= \,\,\chi(X(3,6)) - \chi(Y(3,6)) \,\,= 26-32 \,=\, -6 \,\,= \,\,\chi(\M_{0,6}). \end{equation}
	
	For $n=7$, the situation is more complicated.
	The boundary of $Y(3,7)$ in $ X(3,7)$ consists of seven isomorphic strata $a_1,\dots,a_7$, where $a_i$ 
	is the locus where the points $\{1,2,\ldots,7\} \backslash \{i\}$ are on a conic.
	Their pairwise intersection is the stratum $b \simeq \M_{0,7}$ where all points are on a conic.
	We have $\chi(b) = 24$. The following result is found by a stratified fibration. We omit the proof, which is analogous to but lengthier than the proof of Theorem~\ref{thm:euler}. 
	
	\begin{prop}
		For the stratum $a_7$ in $X(3,7)$, parametrizing seven points in $\P^2$ with no three on a line but
		the first six on a conic, the
		Euler characteristic equals $\chi(a_7) = -312$.
	\end{prop}
	
	\noindent As the $n=7$ analogue to (\ref{eq:coconic6}), we find the following identity among Euler characteristics:
	\begin{equation}
		\label{eq:coconic7}
		\chi(X(3,7)\setminus Y(3,7)) \,\,= \,\,
		1272 -3600 \,\,=\,\,
		7 \cdot (-312) - 6 \cdot 24 \,\, = \,\,
		\sum_{i=1}^7 \chi(a_i) - 6 \chi(b).
	\end{equation}
	
	In this section, we explored the complex geometry of 
	some very affine varieties of interest. We 
	showed how their Euler characteristics
	can be derived by means of inclusion-exclusion.

	\section{Numerical Experiments}
	\label{sec5}
	
	We now turn to the numerical solution of likelihood equations.
	See (\ref{eq:loglike2}) for a tiny example. In this section, each item
	is labeled as Experiment, to highlight the
	experimental nature of our investigation.
	Our computation yields a lower bound for the
	Euler characteristic for $Y(3,8)$.

	A complex variety $X$ is {\em very affine} if it is a
	closed subvariety of an algebraic torus $(\C^*)^m$.
	The {\em log-likelihood function} on $X$ is the logarithm of a monomial with unknown exponents:
	\begin{equation}
		\label{eq:loglike}
		{\rm log}\bigl(z_1^{s_1} z_2^{s_2} \cdots z_m^{s_m} \bigr) \,\, = \,\,
		s_1 \, {\rm log}(z_1) \,+\,
		s_2 \, {\rm log}(z_2) \,+\, \cdots \,+\,
		s_m \, {\rm log}(z_m) .
	\end{equation}
	We are interested in the critical points of 
	the multivalued function (\ref{eq:loglike}) on $X$.
	These points are algebraic functions of the parameters $s_1,s_2,\ldots,s_m$.
	Their number is the signed Euler characteristic $|\chi(X)|$, by \cite{H}. 
	In algebraic statistics, this is known as the {\em ML degree} of~$X$.

	Very affine varieties arising in physics are usually presented in the parametric form
	$$
	X \,\, = \,\,
	\bigl\{\, x \in \C^n \,\,: \,\, f_1(x) \cdot f_2(x) \, \cdots \,f_m(x) \,\not=\, 0 \,\bigr\},
	$$
	where $f_1,f_2,\ldots,f_m$ are irreducible polynomials in $n$ variables, and we set
	$z_i = f_i(x)$ in (\ref{eq:loglike}).
	Thus our task is to numerically solve the following system of $n$ equations in $n$ unknowns:
	\begin{equation}
		\label{eq:likelihoodequations} \qquad \qquad
		\sum_{i=1}^m s_i  \, f_i^{-1} \, \frac{\partial f_i}{\partial x_j} \,\,=\,\, 0 \qquad
		\hbox{for} \,\, j = 1,2,\ldots,n.
	\end{equation}
	The left hand side is a rational function.
	A tiny example with two solutions was seen in (\ref{eq:loglike2}). As explained in \cite{ST}, it is essential {\em not} to clear denominators.
	For our computations, we used the software
	\texttt{HomotopyContinuation.jl} \cite{BT}.
	Systems with $|\chi| \leq 10000$
	are easy to solve.
	
	\begin{expe} \label{ex:3600}
		The ML degree  of $Y(3,7)$ is $3600$. To verify this numerically,
		we set
		\begin{equation}
			\label{eq:Mmatrix}
			M \,\,\, = \,\,\, \begin{bmatrix}
				1 & 0 & 0 & 1 & 1& 1 & 1 \\
				0 & 1 & 0 &1 & x_1 & x_2 & x_3 \\
				0 & 0 & 1 & 1 & x_4 & x_5 & x_6 \\
			\end{bmatrix}.
		\end{equation}	
		Here $n=6$ and $m = 28+7 = 35$.
		The polynomials $f_i(x)$ come in two groups.
		First, there are $28$ non-constant $3 \times 3$ minors of $M$.
		And, second there are the seven conic conditions.
		This gives 
		six equations  (\ref{eq:likelihoodequations}) in six variables $x_j$ that depend on $35$ parameters $s_i$.
		It takes \texttt{HomotopyContinuation.jl} around
		$200$ seconds 
		to  compute  $3600$ distinct complex solutions.
	\end{expe}
	
\noindent	The main contribution of this section is the lower bound on the ML degree of 
	$Y(3,8)$.
	
	\begin{expe}[Proof and discussion for Theorem \ref{thm:euler}]
		\label{expe:sascha}
		The computation of the number $4884387$ is the highlight of this section.
		This is analogous to Experiment~\ref{ex:3600} but it is now much harder.
		We will now explain what was done. It all starts with the
		parametrization
		\begin{equation}
			\label{eq:matrix38}
			M' \,\,\, = \,\,\, \begin{bmatrix}
				1 & 0 & 0 & 1 & 1& 1 & 1 &1 \\
				0 & 1 & 0 &1 & x_1 & x_2 & x_3  & x_4\\
				0 & 0 & 1 & 1 & x_5 & x_6 & x_7 &x_8\\
			\end{bmatrix}.
		\end{equation}	
		The torus, in which the very affine variety $Y(3,8)$ lives, is defined by the complement of the vanishing of the polynomials $f_i$. These are given by  the  $48$  non-constant $3\times3$-minors of $M'$, and the $28$ conditions that no six points lie on a conic.
		The resulting system  consists of  eight  rational function equations in eight variables and it depends on $76$ parameters $s_i$.
		
		We ran the software \texttt{HomotopyContiunation.jl} on this system, using the method
		of monodromy loops, as explained in \cite{ST}. Our first serious experiment with this system ran 
		1708.32 hours on a \verb|2x 12-Core Intel Xeon E5-2680 v3| at \verb|2.5 GHz| with \verb|512 GB RAM|. The computation needed over 74.55 GB. We obtained    $ 4884318 $
		distinct complex solutions. The problematic length of the computation is due to the sheer size. For each generated loop in the iteration there are about 4 million that need to be tracked and checked against the existing list of solutions.   We verified that the number achieved by our computations is in the correct order of magnitude using the estimation method of Hauenstein and Sherman \cite{HS}.
		
		At that point we got in touch with Sascha Timme, who is the
		main developer of the software   \texttt{HomotopyContinuation.jl}. Sascha kindly
		took over the computations, and he even used them to improve his
		implementation of the certification method in \cite{BRT}.
		With this new implementation, he established a guarantee that
		there are $4884387$ distinct non-singular solutions. The system is large and fairly ill-conditioned, so it is not surprising that it is missing $ 93 $ solutions compared to
		the number found in Olof Bergvall's forthcoming work.
		
		Sascha performed his computations with  version 2.9.0 of \texttt{HomotopyContiunation.jl}  on a 
		\verb|MacBook Pro| with the \verb|M1 Pro Chip| and \verb|32 GB of RAM|. The computations used all $10$ threads. 
		The total computation time  was around $40$ hours. The final count of $ 4884387 $ was already reached after 
		running the monodromy loops for $16$ hours.
		Just the certification of the solutions can now be performed in less than $30$ minutes, using 
		the new improvements for \cite{BRT}. All the materials for these computations
		are made available at our {\tt MathRepo} page.
	\end{expe}

	To gain additional insight into the structures behind Theorem \ref{thm:euler},
	we turned to tropical geometry.
	Namely,	
	we applied
	the technique developed in \cite[Section 8]{ABF} to compute the tropical critical points of the soft limit degeneration of the log-likelihood function. We discovered identifications as in
	\cite[Table 4]{ABF}  of tropical critical points with stratification data.
	
	To explain our findings, we start with the log-likelihood function for $n=6$:
	\begin{equation}
		\label{eq:loglike6} L \quad = \quad \sum_{1 \leq i < j < k \leq 6} \!\!\! s_{ijk}\cdot {\rm log}(p_{ijk})
		\,\,+\,\, t \cdot  {\rm log}(q). \end{equation}
	The $p_{ijk} $ are the $3 \times 3$ minors of the matrix $M$ in (\ref{eq:Mmatrix}), and 
	$q$ is the condition for the points to lie on a conic, shown explicitly in (\ref{eq:coconic}).
	The coefficients $s_{ijk}$ and $t$ in (\ref{eq:loglike6}) are random real numbers. We know from
	Theorem \ref{thm:euler} that
	$L$ has $32$ critical points $(x_1,x_2,x_3,x_4) \in \C^4$.
	
	We now multiply some of the coefficients by a small positive parameter $\epsilon$.
	The solutions depend on $\epsilon$. Each coordinate
	$\hat x_i$ can be represented by a Puiseux series in $\epsilon$, and 
	ditto for the resulting quantities $\hat p_{ijk}$ and $\hat q$.
	The valuations of $(\hat p,\hat q)$ are the {\em tropical critical points}.
	
	A numerical method for computing tropical critical points
	and their multiplicities was presented in \cite[Algorithm 1]{ABF},
	along with an implementation using {\tt HomotopyContinuation.jl}.
	We adapted this code to find tropical critical points on $Y(3,n)$ for $n=6$ and $n=7$.
	
	\begin{expe}[Soft limits for $n=6$] \label{ex:soft6}
		We multiply all coefficients $s_{ij6}$ and $t$ with $\epsilon$. The zero vector is the unique
		tropical critical point of multiplicity $32=2\cdot 16$, corresponding to the fibration $Y(3,6) \rightarrow Y(3,5)$.
		By contrast, suppose only $t$ gets multiplied by $\epsilon$.
		We obtain two tropical critical points, with multiplicities $26$ and $6$.
		The first group gives the solutions in $X(3,6)$.
		The second group gives solutions in $\mathcal{M}_{0,6}$,
		with  configurations on a~conic.
	\end{expe}
	
	\begin{expe}[Soft limits for $n=7$]  \label{ex:soft7} The $28$ minors $p_{ijk}$ have parameters $s_{ijk}$ and the seven conic conditions have parameters $t_{i_1,...i_6}$. All parameters having last index
		$k=i_6=7$  are multiplied with $\epsilon$. This gives
		$46$ tropical critical points. One of them has all coordinates zero. It has multiplicity 
		$2880 = 90 \cdot 32$ and corresponds to the generic fiber of the map
		$Y(3,7) \rightarrow Y(3,6)$.
		The other $45$ tropical critical points can be sorted into two classes of $15$ and $30$ each. Each of them has multiplicity $16$. They correspond to the 
		$45$ ways of obtaining Eckardt points on a cubic surface. The resulting formula
		mirrors the structure of~\eqref{eq:chiformula}:
		$$ \chi(Y(3,7)) \,=\, 3600 \,\,=\,\,32 \cdot 90 \,+\, 45 \cdot 16. $$  
		This is the del Pezzo analogue to the formula
		$\chi(X(3,7)) =1272 = 26 \cdot 42 + 15 \cdot 12$
		for the stratified fibration $X(3,7) \rightarrow X(3,6)$.
		This was found in \cite{CUZ}  and verified  in \cite[Table 4]{ABF}.
		
		In  a second experiment, we multiply the seven coefficients $t$
		for the conics by $\epsilon$. We~observe
		$$\qquad  \chi(Y(3,7)) \,=\, 3600 \,\, = 1272\,+\,7\cdot312\,+\,144 .$$
		The tropical multiplicities reveal the boundary of  $Y(3,7)$ inside $X(3,7)$,
		as seen  in (\ref{eq:coconic7}).
	\end{expe}
	
	In conclusion, the numerical computation of tropical critical points
	can yield considerable insight into the geometry of very affine varieties.
	Experiments \ref{ex:soft6} and \ref{ex:soft7} document this for del Pezzo moduli,
	and they suggest that tropical likelihood inference deserves further study.

	\section{Weyl Groups, Roots, and their ML Degrees}
	\label{sec6}
	
	The symmetric group acts on the configuration space $X(3,n)$ by
	permuting the $n$ points. Our del Pezzo moduli spaces $Y(3,n)$
	are more symmetric than that, because they admit an action
	of the Weyl group $W(E_n)$. In what follows we make this very explicit 
	for $n=6,7$.
	
	The hyperplane arrangements of type $E_6$ and $E_7$ are defined
	by the linear forms seen from
	the matrix (\ref{eq:dimatrix}). Reflections at these 
	hyperplanes generate the Weyl groups $W(E_6)$ and~$W(E_7)$,
	which act on the key players in this paper. 
	Their orders are $51840$ and $2903040$.
	
	By Remark \ref{rmk:cremona1}, each group is generated by
	one Cremona involution in addition to the symmetric group
	that swaps labels. That extra involution acts on our coordinates as follows:
	\begin{equation}
		\label{eq:cremonad}
		\!\!\!\!\!\!\!
		\begin{small}
			\begin{bmatrix} d_1 \\ d_2 \\ d_3 \\ d_4 \\ d_5 \\ d_6 \end{bmatrix} \!
			\mapsto \frac{1}{3}\!
			\begin{bmatrix}
				\phantom{-}1  & 0 &\! -1\! & \!\! \phantom{-}1 \!& \!  \! \phantom{-}1\! &\!\! -2 \!  \\
				\phantom{-}1  & 3 & \! \!\phantom{-}2 \! &\! \!\phantom{-}1 \! \! &  \!  \!\phantom{-}1\! &\!\! \phantom{-}1\!  \\
				-2  & 0 &\! -1\! &\!\! -2\!\! &\!\! -2\! & \!\!-2\!  \\
				\phantom{-}1  & 0 &\! -1\! &\!  \!\phantom{-}1 \!\!& \!\! -2\! & \!\! \phantom{-}1\!  \\
				\phantom{-}1 & 0 &\! -1\!  &\!\! -2\! & \! \! \phantom{-}1\! &\!\! \phantom{-}1\!  \\
				-2 & 0 & \!-1 \! &\! \!\phantom{-}1\! & \!\!   \phantom{-}1\! & \!\! \phantom{-}1 
			\end{bmatrix} \!\!
			\begin{bmatrix} d_1 \\ d_2 \\ d_3 \\ d_4 \\ d_5 \\ d_6 \end{bmatrix}
			\, {\rm and} \,
			\begin{bmatrix} d_1 \\ d_2 \\ d_3 \\ d_4 \\ d_5 \\ d_6 \\ d_7 \end{bmatrix} \!
			\mapsto\! \frac{1}{3}\!
			\begin{bmatrix}
				\phantom{-}1  & 0 &\! -1\! & \!\! \phantom{-}1 \!& \!  \! \phantom{-}1\! &\!\!-2 \! & 0 \\
				\phantom{-}1  & 3 & \! \phantom{-}2 \! &\!\! \phantom{-}1 \!& \!   \!\phantom{-}1\! &\!\! \phantom{-}1\! & 0 \\
				-2  & 0 &\! -1\! &\!\! -2\! &\!\! -2\! &\! \!-2\! & 0 \\
				\phantom{-}1  & 0 &\! -1\! & \! \!\phantom{-}1 \!&\! \! -2\! & \! \!\phantom{-}1\! & 0 \\
				\phantom{-}1 & 0 &\! -1\!  &\!\! -2\! & \!\!  \phantom{-}1\! &\!\! \phantom{-}1\! & 0 \\
				-2 & 0 & \!-1 \! &\!\! \phantom{-}1\! & \! \!  \phantom{-}1\! & \!\!\phantom{-}1 \!& 0 \\
				\phantom{-} 1 & 0 & \! \phantom{-}2 \! &\! \!\phantom{-}1 \! \!&  \! \phantom{-} 1\! &\!\! \phantom{-}1\! & 3 \\
			\end{bmatrix} \!\!
			\begin{bmatrix} d_1 \\ d_2 \\ d_3 \\ d_4 \\ d_5 \\ d_6 \\ d_7 \end{bmatrix} \!\!
		\end{small}
	\end{equation}
	The $7 \times 7$ matrix on the right in (\ref{eq:cremonad}) is
	obtained from the $7 \times 7$ matrix in
	equation (4-3) on page 337 in \cite{RSSS}, namely
	by the change of coordinates in equation (4-4) on page 338 in \cite{RSSS}.
	Our usage of the notation $d_1,\ldots,d_7$ agrees with \cite{RSSS, RSS2}
	and  \cite[Section 6]{RSS1}.
	The matrix on the left in (\ref{eq:cremonad}) is
	obtained from that on the right by deleting the last row and last column.

	The matrix in (\ref{eq:dimatrix}) gives a parametrization of $Y(3,6)$ and
	$Y(3,7)$ respectively. A natural idea is to explore
	the log-likehood function~(\ref{eq:loglike6}) after
	plugging in this parametrization. For $n=6$, the resulting function in $d$-coordinates equals
	\begin{equation}
		\label{eq:logliked} \sum_{1 \leq i < j \leq 6} \!\!\!s_{ij} \cdot {\rm log}(d_i-d_i)\, \,\,+\!\!\!\!
		\sum_{1 \leq i < j < k \leq 6}\!\!\! t_{ijk} \cdot {\rm log}(d_i+d_j+d_k) \,\,+\,\,
		v \cdot {\rm log} (d_1+d_2+d_3+d_4+d_5+d_6). \,\,\,\,
	\end{equation}
	If the coefficients $s_{ij},t_{ijk},v$ are random real numbers then this
	function has no critical points.
	We need to assume that these $36$ coefficients sum to zero.
	Under that hypothesis, (\ref{eq:logliked}) is the logarithm of
	a rational function on the projective space $\P^{5}$, and
	the number of critical points is the Euler characteristic of
	the complement of the $E_6$ hyperplane arrangement  in $\P^5$, which we denote by $\mathcal{A}(E_6)$. Similarly we write $\mathcal{A}(E_7)$
	for the complement of the $63$ hyperplanes in~$\P^6$.
	These complements are very affine varieties of dimension $5$ and $6$ respectively.
	
	\begin{prop} The ML degrees of $\mathcal{A}(E_6)$ and $\mathcal{A}(E_7)$  are equal to $5040$ and $368640$. 
	\end{prop}
	
	\begin{proof} The characteristic polynomials of the two hyperplane arrangements are
		\begin{equation} \begin{matrix}
				p_{E_6}(t) &=& (t-1)(t-4)(t-5)(t-7)(t-8)(t-11), \phantom{(t-17))} \\ 
				p_{E_7}(t) &=& (t-1)(t-5)(t-7)(t-9)(t-11)(t-13)(t-17).
			\end{matrix}
		\end{equation}
		The zeros of these polynomials are the {\em exponents} of the Coxeter groups $W(E_6)$ and $W(E_7)$.
		If we pass to affine space by declaring one of the hyperplanes to be at infinity,
		then the ML degree is the number of bounded regions in the arrangement.
		This number is computed by evaluating the reduced characteristic polynomial
		$p_{E_n}(t)/(t-1)$ at $t=1$. In our two cases,
		\begin{equation}
			\label{eq:inourtwocases}
			\begin{matrix}
				|\chi( \mathcal{A}(E_6))|  &=& 3 \cdot 4 \cdot 6 \cdot 7 \cdot 10  &=& 5040, \\
				|\chi( \mathcal{A}(E_7))|  &=& 4 \cdot 6 \cdot 8 \cdot 10 \cdot 12 \cdot 16 &= & 368640.
			\end{matrix}
		\end{equation}
		We also verified both ML degrees by solving the likelihood equations, as in Section \ref{sec5}.
	\end{proof}
	
\noindent	We next examine the two morphisms of very affine varieties that are given by (\ref{eq:dimatrix}):
	\begin{equation}
		\label{eq:uniformization} \mathcal{A}(E_6) \rightarrow  Y(3,6) \quad {\rm and} \quad 
		\mathcal{A}(E_7) \rightarrow Y(3,7). 
	\end{equation}
	One might naively think that these maps are
	fibrations. But, this is not the case.
	The fiber of the first map is the parabolic curve,
	which can be realized as the intersection
	of the cubic surface with the quartic surface defined by its Hessian \cite[Section 3]{CGL}.
	This curve is generally smooth, but it is singular for special points in $Y(3,6)$.
	For instance,  the parabolic curve of the Fermat cubic $x_0^3 + x_1^3+x_2^3+x_3^3$
	is reducible and it has $18$ singular points, while for 
	the Clebsch cubic 
	$x_0^3 + x_1^3+x_2^3+x_3^3-(x_0+x_1+x_2+x_3)^3$, it is
	irreducible with $10$ singular points.

	The failure of the maps in (\ref{eq:uniformization})  to be fibrations explains the fact that
	the ML degrees in (\ref{eq:inourtwocases}) are not divisible by those in  Theorem \ref{thm:euler}.
	To reach the desired divisibility, we need additional linear constraints on the
	coefficients $s_{ij},t_{ijk},v$ in (\ref{eq:logliked}). These  constraints
	arise from representing our moduli spaces $Y(3,6)$ and $Y(3,7)$ in high-dimensional projective spaces.
	This goes back to Coble in 1928. 
	We now review the relevant material from \cite{CGL, RSSS, RSS1}.

	The {\em Yoshida variety} $\mathcal{Y}^\circ$
	is the image of $\mathcal{A}(E_6)$ under
	the map $\P^5 \rightarrow \P^{39}$
	that is given by the subroot systems of type $A_2^{\times 3}$.
	Its $40$ coordinates are the $30$ products 
	$p_{125} p_{126} p_{134} p_{234} p_{356} p_{456}$
	and the $10$ products like $p_{123} p_{456} q$, up to relabeling,
	where $p_{ijk}$ denotes $3 \times 3$ minors of (\ref{eq:dimatrix}).
	After dividing by $\prod_{1 \leq i < j \leq 6}(d_i-d_j)$, each coordinate
	is a product of $9$ roots of $E_6$.
	We write $M_6$ for the $36 \times 40$ incidence matrix
	where the rows are labeled by the roots and the columns are labeled by the Yoshida coordinates.
	This matrix has entries in $\{0,1\}$, with nine $1$'s in each column.
	By \cite[Theorem 6.1]{RSS1}, the rank of the matrix $M_6$ equals~$16$.
	
	The {\em G\"opel variety} $\mathcal{G}^\circ$
	is the image of $\mathcal{A}(E_7)$ under
	the map $\P^5 \rightarrow \P^{134}$
	that is given by the $135$ subroot systems of type $A_1^{\times 7}$.
	Its coordinates are the $30$ Fano configurations
	$p_{124} p_{235} p_{346} p_{457} p_{156} p_{267} p_{137}$
	and the $105$ Pascal configurations $p_{127} p_{347} p_{567} q$, up to relabeling.
	After dividing by $\prod_{1 \leq i < j \leq 7}(d_i-d_j)$, each coordinate
	is a product of $7$ roots of $E_7$.
	We write $M_7$ for the $63 \times 135$ incidence matrix
	where the rows are labeled by roots and the columns are labeled by G\"opel coordinates.
	This matrix has entries in $\{0,1\}$, with seven $1$'s in each column.
	By \cite[Theorem 6.1]{RSSS}, the rank of the matrix $M_7$ equals~$36$.
	See \cite[Theorem 6.2]{RSSS} for the homogeneous prime ideal 
	in $135$ variables that defines the G\"opel variety.
	
	\begin{rmk} \label{rmk:YG}
		The two varieties above are our moduli spaces:
		$\mathcal{Y}^\circ = Y(3,6)$ and
		$\mathcal{G}^\circ = Y(3,7)$.
		See \cite[Lemma 3.1]{RSS2} for details on the
		tropical varieties corresponding to $\mathcal{A}(E_7), \mathcal{A}(E_6),\, \mathcal{Y}^\circ,  \mathcal{G}^\circ$.
	\end{rmk}
	
	The virtue of the Yoshida and G\"opel varieties is that
	our Weyl groups act by signed permutations on the
	coordinates of $\,\mathcal{Y}^\circ$ in $\P^{39}\,$ 
	and of $\,\mathcal{G}^\circ$ in $\P^{134}$. These are simply
	the permutations of the subroot systems $A_2^{\times 3} $ in $ E_6$ and
	$A_1^{\times 7} $ in $ E_7$. Explicitly, we
	permute labels and apply the Cremona matrices in 
	(\ref{eq:cremonad}) to products of linear forms in $d_1,d_2,d_3,d_4,d_5,d_6,d_7$.
	
	We now state the 
	linear constraints on the $36$, $63$
	coefficients $s_{ij},t_{ijk},v$ in (\ref{eq:logliked}):
	these coefficient vectors must lie in the column space of the
	matrix $M_6$, $M_7$~respectively. This~hypothesis ensures that
	(\ref{eq:logliked}) is well-defined as a multivalued function
	on $\,\mathcal{Y}^\circ = Y(3,6)$ and~$\mathcal{G}^\circ = Y(3,7)$.
	Geometrically, we  are considering the total spaces of the
	map $\mathcal{A}(E_6) \rightarrow \mathcal{Y}^\circ$
	and $\mathcal{A}(E_7) \rightarrow \mathcal{G}^\circ$.
	We call these total spaces the {\em Yoshida parametrization} and the
	{\em G\"opel parametrization}.
	
	We now apply the numerical methods from Section~\ref{sec5} to
	compute the critical points  of the log-likelihood function (\ref{eq:logliked})
	where the coefficients are taken at random from these column spaces.
	The numbers of these critical points are the ML degrees of the parametrizations~(\ref{eq:uniformization}).
	
	\begin{expe}\label{prop:86400}
		The ML degree of the Yoshida parametrization equals $\,
		2880=90\cdot32$,
		and the ML degree of the  G\"opel parametrization equals 
		$\,     86400=24\cdot 3600$.
	\end{expe}
	
	The factor $24$ in the second product is the size of the  generic fiber of  $\mathcal{A}(E_7) \rightarrow \mathcal{G}^\circ$.
	In other words, it is the number of cuspidal cubics through seven given points in $\P^2$; see \cite[Section~4]{RSSS}. The number $90$ in the first product is more mysterious to us. It should be
	the Euler characteristic of the parabolic curve \cite{CGL} after removal of a collection of special points.
	
	\section{Intermezzo from Physics} \label{sec7}
	
	Our point of departure is the {\em Koba-Nielsen string integral} \cite{KN}, which we write in the form
	\begin{eqnarray}
		\label{eq:KobaNielsen}
		\phi(s) \,\, = \,\,
		\epsilon^{n-3} \int_{\mathcal{M}_{0,n}^+} \frac{1}{p_{12}p_{23}\cdots p_{n1}}\prod_{1 \leq i < j \leq n} p_{ij}^{\epsilon \cdot s_{ij}} dp .
	\end{eqnarray}
	Here $\mathcal{M}_{0,n} $ is the $(n{-}3)$-dimensional moduli space of
	$n$ labeled points on the  line~$\mathbb{P}^1$. This is 
	the quotient of the open Grassmannian ${\rm Gr}(2,n)^\circ$ modulo the 
	torus action by $(\mathbb{C}^*)^n$. We write $p_{12}, p_{13}, \ldots, p_{n-1,n}$
	for the Pl\"ucker coordinates on ${\rm Gr}(2,n)^\circ$. 
	The positive Grassmannian ${\rm Gr}_+(2,n)$ consists of real points 
	whose Pl\"ucker coordinates are positive. Its quotient modulo $(\mathbb{R}^+)^n$ is
	the {\em positive geometry} $\mathcal{M}_{0,n}^+$, 
	to be identified with the
	orthant~$\mathbb{R}^{n-3}_{>0}$.  See 
	\cite[Sec.~3.5]{Lam}.
	
	The physically meaningful quantities in  (\ref{eq:KobaNielsen})
	are the exponents $s_{ij}$. These are known as {\em Mandelstam invariants}.
	Using the {\em spinor-helicity formalism}, we write
	$s_{ij} = {\rm det}(K_i + K_j)$, where $K_1,K_2,\ldots,K_n$ are 
	$2 \times 2$ matrices of rank $1$, i.e.~we have
	points on the lightcone in $\mathbb{R}^{3,1}$.
	The $K_i$ are 
	momenta of $n$ massless particles in a scattering process, so they
	sum to zero. The moduli space  $\mathcal{M}^+_{0,n}$ serves as the {\em open string worldsheet} in 
	the physics literature~\cite{ABHY}.

	Momentum conservation 
	$\sum_{i=1}^n K_i = 0$ translates into
	the  relations  $\sum_{i=1}^n s_{ij} = 0$ for all~$j$,  where $s_{jj} = 0$.
	These ensure that the integrand  in (\ref{eq:KobaNielsen}) is 
	well-defined on $\mathcal{M}_{0,n}$.
	To evaluate the integral (\ref{eq:KobaNielsen}) numerically, one uses local coordinates
	$(0,z_2,z_3,\ldots,z_{n-2},1,\infty)$, so that
	$p_{ij} = -z_i + z_j$, suitably interpreted. 
	We have $\mathcal{M}_{0,n}^+ \simeq \mathbb{R}^{n-3}_{> 0}$
	by assuming $0 {<} z_2 {<} \cdots {<} z_{n-2} {<} 1$.
	
	While $\phi(s)$ is a transcendental function, its limit $m_n$ for $\epsilon \rightarrow 0 $ is a
	rational function. Cachazo, He and Yuan \cite{CHY} computed it
	by summing over critical points of the scattering potential.
	For $n=4$ this is the derivation from
	(\ref{eq:loglike2}) to (\ref{eq:smallamplitude}) in Section \ref{sec2}.
	For $n=6$, we find
	\begin{equation}
		\label{eq:14terms}  \begin{matrix}
			m_6 & =	& \,\,\frac{1}{s_{12} s_{34} s_{56}}
			+\frac{1}{s_{12} s_{56}  s_{123}}
			+\frac{1}{s_{23} s_{56}  s_{123}}
			+\frac{1}{s_{23} s_{56}  s_{234}}
			+\frac{1} {s_{34} s_{56} s_{234}}
			+ \frac{1}{s_{16} s_{23} s_{45}}
			+ \frac{1}{ s_{12} s_{34} s_{345}} \smallskip  \\
			& &   +\, \frac{1}{s_{12} s_{45} s_{123}}
			+ \frac{1}{s_{12} s_{45} s_{345}}
			+\frac{1}{s_{16} s_{23} s_{234}}
			+\frac{1} {s_{16} s_{34} s_{234}}
			+\frac{1}{s_{16} s_{34} s_{345}}
			+\frac{1}{s_{16} s_{45} s_{345}}
			+\frac{1}{s_{23} s_{45} s_{123}}.
		\end{matrix}
	\end{equation}
	This is the {\em biadjoint scalar amplitude} \cite{CEGM, CHY}.
	Here we abbreviate  $s_{ijk} = s_{ij} + s_{ik} + s_{jk}$.
	The $14$ summands in the amplitude (\ref{eq:14terms}) correspond to
	the planar trivalent trees with six labeled leaves, and hence to 
	the vertices of the {\em associahedron} in $\R^3$.
	Summing the expressions (\ref{eq:14terms}) over all
	$60$ cyclic orderings of the set $\{1,\ldots,6\}$, one obtains the
	{\em cubic scalar amplitude}.
	
	We now pass from
	$\mathcal{M}_{0,n} $ to the moduli space
	$X(3,n) = {\rm Gr}(3,n)^\circ/(\C^*)^n$
	of $n$ points in general position in $\mathbb{P}^2$.
	The Pl\"ucker coordinates on ${\rm Gr}(3,n)$ are denoted $p_{ijk}$.
	We write
	$\mathfrak{s}_{ijk}$ for the Mandelstam invariants in 
	CEGM theory \cite{CEGM}.
	In the physical setting, they are
	$\mathfrak{s}_{ijk} = {\rm det}(K_i+K_j+K_k)$
	where $K_1,\ldots,K_n$  are $3 {\times} 3$ matrices of rank $1$
	whose sum has rank~$1$.
	
	As the analogue to (\ref{eq:KobaNielsen}),
	Arkani-Hamed, He and Lam \cite[Equation 6.11]{AHL} introduced the
	{\em stringy integral}
	\begin{equation}
		\label{eq:nima}
		\psi(\mathfrak{s}) \,\, = \,\,
		\epsilon^{2n-8} \int_{X_+(3,n)} \omega_{3,n}\prod_{1 \leq i< j<k \leq n} p_{ijk}^{\epsilon \cdot \mathfrak{s}_{ijk}}.
	\end{equation}
	In the integrand we see
		 the canonical form of $X(3,n)$, which is given in \cite[Equation 6.8]{AHL} as
    $$\omega_{3,n} \,\,=\,\,
     \frac{d^{3\times n}C}{\text{vol} SL(3)\times GL(1)^n}\frac{1}{p_{123}p_{234}\cdots p_{n12}}.$$
     Here,
    $(p_{123}p_{234}\cdots p_{n12})^{-1}$ is the 3-Parke-Taylor factor.  The limit $\epsilon \rightarrow 0$ is a rational function  in the unknowns
	$\mathfrak{s}$ of degree $8-2n = - {\rm dim}(X(3,n))$. This is called the {\em CEGM amplitude}.
	It was computed from the scattering potential on $X(3,n)$ 
	by Cachazo, Early, Guevara and Mizera in \cite{CEGM}.
	Note that the	 integrand  (\ref{eq:nima}) is a well-defined rational function on
	the configuration space $X_+(3,n)$ because
	matrix kinematics requires  $\sum_{j,k=1}^n \mathfrak{s}_{ijk} = 0$ for all $i$.

	The role of the tree space in (\ref{eq:14terms}) is played by 
	the tropicalization of $ X_+(3,n)$, a polyhedral space with very rich combinatorics.
	To see the connection with (\ref{eq:14terms}), we now fix $n=6$, and we
	multiply the integrand  in (\ref{eq:nima}) by  a factor representing the
	conic condition $q$ in (\ref{eq:coconic}).
	This replaces the Parke--Taylor factor by the following rational function in Pl\"ucker coordinates:
	\begin{equation}
		\label{eq:etude}
		\frac{1}{p_{123} p_{234} p_{345} p_{456} p_{156} p_{126}} \bigl[\, \frac{
			p_{123} p_{156} p_{246} p_{345} }{p_{126} p_{135} p_{234} p_{456} } -1 \, \bigr]^{-1} 
		\,\,= \,\,  \frac{p_{135}} {p_{123} p_{345} p_{156} \,q} ,
	\end{equation}
	as in \cite[Section 6]{CE22}. The CEGM amplitude  in this paper
	is an integral over $Y(3,6)$ with the
	integrand (\ref{eq:etude}).
	Thus, we replace  the configuration space $X(3,6)$ by
	the del Pezzo moduli space $Y(3,6)$.
	The analog to (\ref{eq:etude}) for $n=7$ involves two conic factors.
	Here the number two is the
	codimension of the stratum $b \simeq \mathcal{M}_{0,7}$ inside $Y(3,7)$,
	seen in the end of Section \ref{sec4}. 
	
	The resulting CEGM amplitudes are 
	rational functions, to be  computed in Section \ref{sec9}.
	They are highly symmetric, reflecting the 
	Weyl group combinatorics in Section \ref{sec6}.
	As in \cite{CHY}, the computation rests on summing over
	the critical points of the scattering potential $L$, which is shown in
	(\ref{eq:loglike6}).
	The relevant numerical algebraic geometry
	is explained in Section \ref{sec5}.
	
	The integrand on the right hand side (\ref{eq:etude})
	appeared prominently in the developments that led to the
	amplituhedron. We found it explicitly
	in the 2010 article \cite{NVW}, which was inspired 
	{\em ``by Witten's proposal that the $N^{k-2} MHV$ superamplitude should be the integral of an  
		open string current algebra correlator over the space of degree $k-1$ curves in
		supertwistor space $\P^{3|4}$''}.
	By Cauchy's Residue Theorem, the sum of the four residues 
	in $p_{135}/(p_{123} p_{345} p_{156} \,q) $ is equal to zero.
	Nandan, Volovich and Wen \cite{NVW} rewrite this identity of residues as follows:
	\begin{equation}
		\label{eq:cauchy}
		\{123 \} + \{345\} + \{156\} \,\, = \,\, -\, \{\,q \,\}.
	\end{equation}
	The right hand side points to
	positive del Pezzo geometry. The left hand side
	identifies the {\em BCFW triangulation} of the 
	{\em amplituhedron} $\mathcal{A}_{6,1,4}$, which is the image of
	a positive linear~map 
	$$ {\rm Gr}_+(1,6) \,\longrightarrow \,\, {\rm Gr}(1,5). $$
	To be precise, $\mathcal{A}_{6,1,4}$ is a
	cyclic $4$-polytope with $6$ vertices, by the identification
	of cyclic polytopes with totally positive matrices.
	Our  polytope has precisely two  triangulations, each into three
	$4$-simplices. The left hand side of (\ref{eq:cauchy}) is one of these
	triangulations of $\mathcal{A}_{6,1,4}$.
	
	In Section \ref{sec2} we explained the
	positive geometry structure on the moduli space $Y(3,n)$ for $n=5$,
	by identifying this very affine surface with $ \mathcal{M}_{0,5}$.
	The resulting amplitude, shown in (\ref{eq:smallamplitude}),
	is a smaller version of that in (\ref{eq:14terms}).
	A key role was played by the {\em u-equations} in (\ref{eq:ueqns}).
	
	Such $u$-equations define very affine varieties known as  binary geometries, seen recently in
	\cite{AHLT, HLRZ} at the interface of particle physics and geometric combinatorics.
	Given any simple $d$-polytope $P$ with $m$ facets,
	we introduce one variable $u_i$ for each facet. 
	The $u$-equations~are
	\begin{equation}
		\label{eq:ueqns2}
		u_i \,+\, \prod_j u_j^{\beta_{ij}}\,\,= \,\,1  \quad {\rm for} \,\, i = 1,2,\ldots,m,  
	\end{equation}
	where $j$ runs over all facets of
	$P$ that are disjoint from facet $i$, and
	the $\beta_{ij}$ are  positive~integers.
	If there exist $\beta_{ij}$ such that the variety $V$ defined by (\ref{eq:ueqns2}) has dimension $d$ and admits a stratification that induces a combinatorial isomorphism to $P$, then $V$ is a {\em binary geometry}.
	If this happens with $\beta_{ij} = 1$ for all $i,j$,
	then the $u$-equations  are  {\em perfect}.
	The existence of binary geometries is rare,
	as explained in \cite{AHLT}, where
	this was connected to  {\em cluster algebras}.
	He, Li, Raman and Zhang \cite{HLRZ} study the scenarios when $P$
	is a generalized permutohedron. 
	
	
	For $n=5,6$, the regions of a del Pezzo surface are perfect binary geometries, because they
	are triangles, quadrilaterals and pentagons, see Theorem \ref{thm:subdivision}.
	By \cite[Section~VII]{AHLT},
	this would no longer hold if $m$-gons with $m \geq 6$ appeared.
	In Sections \ref{sec8} and \ref{sec9} we introduce $u$-variables for the
	moduli spaces $Y(3,6)$ and $Y(3,7)$, and we argue that these define  perfect binary geometries.
	It is important to note that our approach does
	{\em not} use cluster algebra structures on
	Grassmannians. For instance, the cluster algebra structure on ${\rm Gr}(3,6)$
	is built on the root system $D_4$, while $Y(3,6)$ rests on
	the root system $E_6$.
	We are hopeful that the new perfect binary geometries 
	from del Pezzo surfaces
	will lead to
	new discoveries in physics.
	
	\section{Combinatorics of Pezzotopes}
	\label{sec8}
	
	We now turn to the real geometry of the very affine varieties $Y(3,6)$ and $Y(3,7)$.
	Our aim is to characterize the moduli spaces of del Pezzo surfaces over the real numbers.
	This involves combinatorics and representation theory, which is the main theme in this section,
	as well as geometry and commutative algebra, which will occupy us in Section \ref{sec9}.
	Here is a key result.
	
	\begin{thm} \label{thm:2polytopes}
The {\bf real} moduli space $Y(3,6)$ has $432$ connected components, all $W(E_6)$ equivalent,
		the closure of each is homeomorphic as a cell-complex to a simple
		$4$-polytope with f-vector $(45,90,60,15)$. The real moduli space $Y(3,7)$ has $60480$ connected components, all $W(E_7)$ equivalent,
		the closure of each is homeomorphic as a cell-complex to a simple $6$-dimensional homology ball with f-vector $(579,1737,2000,1105,$ $297,34)$.	
	\end{thm}
	
	By an abuse of terminology, we use the term {\em pezzotope} for the
	connected components and the polytope (in the case of $Y(3,6)$) in this theorem. In the $E_7$ case, we do not yet know a 
	realization of the component as a convex polytope, but we expect it to be. Further justification
	is provided in Theorem \ref{thm:E7pezzo}, where a
	computation in commutative algebra is used to
	show that the boundary of this region in $Y(3,7)$
	is a homology sphere of dimension~$5$.
	
	We now embark on our combinatorial journey to the
	$E_6$ pezzotope and the $E_7$ pezzotope.
	The proof of Theorem~\ref{thm:2polytopes} 
	will be concluded in Section \ref{sec9}.
	A key point is that the
	Weyl groups act transitively on the regions.
	This result is due to Sekiguchi and Yoshida \cite{Sekolder, Sek, SY}.
	Our exposition follows that given by
	Hacking, Keel and Tevelev in \cite[Section~8]{HKT}.
	Our contributions include the f-vectors
	and the convex realization of the $E_6$ pezzotope.
	Also the number $60480$ of regions appears to be new:
	we did not find it in the sources listed above.
	
	We describe the simplicial complexes that are dual to the boundaries of the pezzotopes.
	Let $\mathcal{G}(E_n)$ denote the respective edge graph, with
	$15$ vertices for $n=6$ and with $34$ vertices for $n=7$.
	Each sphere is a flag simplicial complex, i.e.~its
	simplices are the cliques in the graph.
	To present the combinatorics of the pezzotopes,
	it suffices to give the graphs $\mathcal{G}(E_n)$.
	
	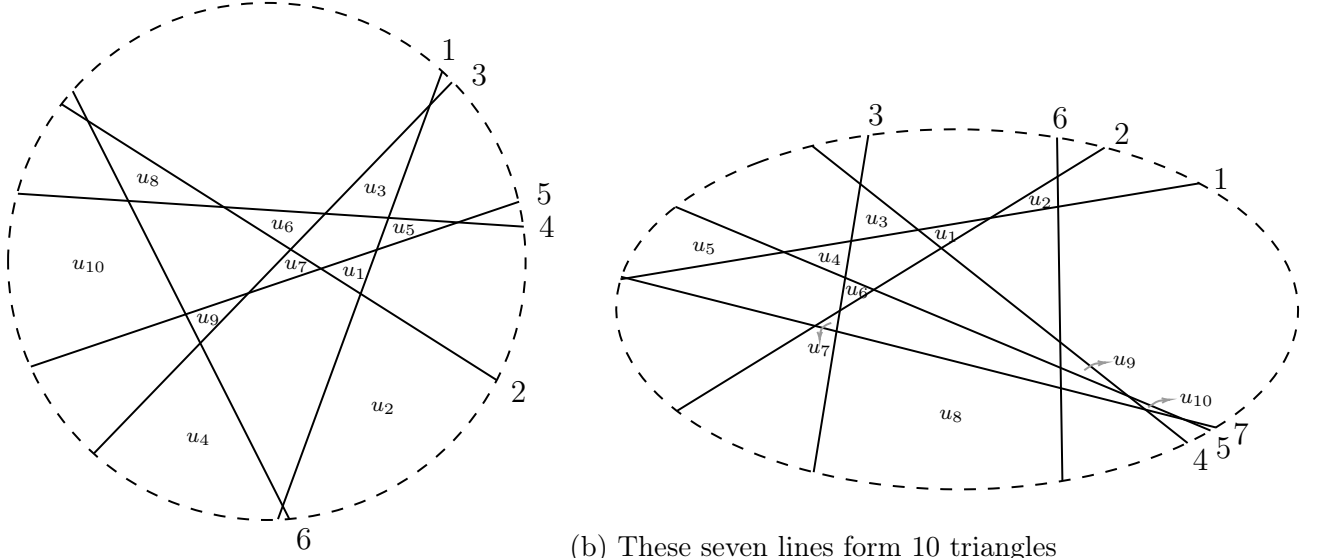
\begin{figure}[ht] 
		\centering
		\!\!\!\!\!\!\!\!\!
		\begin{subfigure}{.43\textwidth}
			\centering
			\tikzset{every picture/.style={line width=0.75pt}} 
			\!\!\!\!\!\!
			\begin{tikzpicture}[x=0.75pt,y=0.75pt,yscale=-1,xscale=1,scale=0.9]
				
				\draw    (205.5,187) -- (489,205.5) ;
				\draw    (443.38,118.25) -- (351.25,369.5) ;
				\draw    (229.88,136.75) -- (473.75,291.5) ;
				\draw    (448.75,124.5) -- (248.25,332.5) ;
				\draw    (212.75,284) -- (486.38,191.5) ;
				\draw    (236.25,129.5) -- (357.63,369.5) ;
				\draw [dash pattern={on 4.5pt off 4.5pt}]   (200,225) .. controls (200,144.92) and (264.92,80) .. (345,80) .. controls (425.08,80) and (490,144.92) .. (490,225) .. controls (490,305.08) and (425.08,370) .. (345,370) .. controls (264.92,370) and (200,305.08) .. (200,225) -- cycle ;
				
				\draw (441,98.9) node [anchor=north west][inner sep=0.75pt]    {$1$};
				\draw (480,290.9) node [anchor=north west][inner sep=0.75pt]    {$2$};
				\draw (458,112.9) node [anchor=north west][inner sep=0.75pt]    {$3$};
				\draw (496,197.9) node [anchor=north west][inner sep=0.75pt]    {$4$};
				\draw (494.5,179.9) node [anchor=north west][inner sep=0.75pt]    {$5$};
				\draw (359.75,372.9) node [anchor=north west][inner sep=0.75pt]    {$6$};
				\draw (385.75,226.5) node [anchor=north west][inner sep=0.75pt]  [font=\scriptsize]  {$u_{1}$};
				\draw (401.75,301.9) node [anchor=north west][inner sep=0.75pt]  [font=\scriptsize]  {$u_{2}$};
				\draw (397.75,179.9) node [anchor=north west][inner sep=0.75pt]  [font=\scriptsize]  {$u_{3}$};
				\draw (298.25,319.4) node [anchor=north west][inner sep=0.75pt]  [font=\scriptsize]  {$u_{4}$};
				\draw (413.75,202) node [anchor=north west][inner sep=0.75pt]  [font=\scriptsize]  {$u_{5}$};
				\draw (345.75,200) node [anchor=north west][inner sep=0.75pt]  [font=\scriptsize]  {$u_{6}$};
				\draw (353.31,222.5) node [anchor=north west][inner sep=0.75pt]  [font=\scriptsize]  {$u_{7}$};
				\draw (270.25,174) node [anchor=north west][inner sep=0.75pt]  [font=\scriptsize]  {$u_{8}$};
				\draw (304,252.5) node [anchor=north west][inner sep=0.75pt]  [font=\scriptsize]  {$u_{9}$};
				\draw (234.25,222.4) node [anchor=north west][inner sep=0.75pt]  [font=\scriptsize]  {$u_{10}$};
			\end{tikzpicture}

			\caption{The variables $u_1,u_2,\ldots,u_{10}$ for $E_6$ \\ are
				the triangles in this line arrangement.
				\label{fig:sixlines}}
		\end{subfigure}
		\begin{subfigure}{.45\textwidth}
			\centering
			
			\tikzset{every picture/.style={line width=0.75pt}} 
			\!\!\!\!\!
			\begin{tikzpicture}[x=0.75pt,y=0.75pt,yscale=-1,xscale=1,scale=0.6]
				
				\draw    (51.38,649.75) -- (538.13,569) ;
				\draw    (98.88,760.5) -- (458.38,539) ;
				\draw    (259.8,528.4) -- (213.8,812) ;
				\draw    (212.13,537.25) -- (528.13,787.5) ;
				\draw    (97.63,588.75) -- (547.38,777) ;
				\draw    (418.38,530.5) -- (422.88,819.25) ;
				\draw    (51.8,647.6) -- (551.8,774.4) ;
				\draw  [dash pattern={on 4.5pt off 4.5pt}] (166.46,551.79) .. controls (294.68,503.17) and (473.81,518.94) .. (566.57,587) .. controls (659.33,655.07) and (630.59,749.65) .. (502.38,798.27) .. controls (374.16,846.89) and (195.02,831.12) .. (102.27,763.06) .. controls (9.51,694.99) and (38.25,600.41) .. (166.46,551.79) -- cycle ;
				\draw [color={rgb, 255:red, 155; green, 155; blue, 155 }  ,draw opacity=1 ]   (495,757.5) .. controls (503.46,751.63) and (501.98,752.16) .. (511.12,751.2) ;
				\draw [shift={(513,751)}, rotate = 173.66] [color={rgb, 255:red, 155; green, 155; blue, 155 }  ,draw opacity=1 ][line width=0.75]    (6.56,-1.97) .. controls (4.17,-0.84) and (1.99,-0.18) .. (0,0) .. controls (1.99,0.18) and (4.17,0.84) .. (6.56,1.97)   ;
				\draw [color={rgb, 255:red, 155; green, 155; blue, 155 }  ,draw opacity=1 ]   (218.9,698.48) .. controls (219.09,693.13) and (220.39,688.73) .. (228.63,686.5) ;
				\draw [shift={(218.88,700.5)}, rotate = 270] [color={rgb, 255:red, 155; green, 155; blue, 155 }  ,draw opacity=1 ][line width=0.75]    (6.56,-1.97) .. controls (4.17,-0.84) and (1.99,-0.18) .. (0,0) .. controls (1.99,0.18) and (4.17,0.84) .. (6.56,1.97)   ;
				\draw [color={rgb, 255:red, 155; green, 155; blue, 155 }  ,draw opacity=1 ]   (441.25,726.25) .. controls (449.71,720.38) and (448.23,720.91) .. (457.37,719.95) ;
				\draw [shift={(459.25,719.75)}, rotate = 173.66] [color={rgb, 255:red, 155; green, 155; blue, 155 }  ,draw opacity=1 ][line width=0.75]    (6.56,-1.97) .. controls (4.17,-0.84) and (1.99,-0.18) .. (0,0) .. controls (1.99,0.18) and (4.17,0.84) .. (6.56,1.97)   ;
				
				\draw (547,554) node [anchor=north west][inner sep=0.75pt]    {$1$};
				\draw (464,516) node [anchor=north west][inner sep=0.75pt]    {$2$};
				\draw (257,500) node [anchor=north west][inner sep=0.75pt]    {$3$};
				\draw (530.13,790.9) node [anchor=north west][inner sep=0.75pt]    {$4$};
				\draw (549.38,780.4) node [anchor=north west][inner sep=0.75pt]    {$5$};
				\draw (412,503) node [anchor=north west][inner sep=0.75pt]    {$6$};
				\draw (564,768.4) node [anchor=north west][inner sep=0.75pt]    {$7$};
				\draw (313,606) node [anchor=north west][inner sep=0.75pt]  [font=\scriptsize]  {$u_{1}$};
				\draw (392,577) node [anchor=north west][inner sep=0.75pt]  [font=\scriptsize]  {$u_{2}$};
				\draw (255,592.9) node [anchor=north west][inner sep=0.75pt]  [font=\scriptsize]  {$u_{3}$};
				\draw (215,625.3) node [anchor=north west][inner sep=0.75pt]  [font=\scriptsize]  {$u_{4}$};
				\draw (109.55,615.7) node [anchor=north west][inner sep=0.75pt]  [font=\scriptsize]  {$u_{5}$};
				\draw (238,652.9) node [anchor=north west][inner sep=0.75pt]  [font=\scriptsize]  {$u_{6}$};
				\draw (206,702) node [anchor=north west][inner sep=0.75pt]  [font=\scriptsize]  {$u_{7}$};
				\draw (317.15,756.9) node [anchor=north west][inner sep=0.75pt]  [font=\scriptsize]  {$u_{8}$};
				\draw (463.1,711.45) node [anchor=north west][inner sep=0.75pt]  [font=\scriptsize]  {$u_{9}$};
				\draw (519,743.95) node [anchor=north west][inner sep=0.75pt]  [font=\scriptsize]  {$u_{10}$};

			\end{tikzpicture}

			\caption{These seven lines form 10 triangles \\ corresponding to the 10 vertices of the \\ tetradiagram in Figure \ref{fig:tetra}. }\label{fig:sevenlines}
		\end{subfigure}
		\caption{The triangles in these line arrangements are the vertex labels in Figure \ref{fig:drei}.}\label{fig:lines}
	\end{figure}

	We begin with $n=6$. The $15$ vertices of $\mathcal{G}(E_6)$
	are denoted $u_1,u_2,\ldots,u_{15}$. They are
	root subsystems that are described by the colorful Petersen graph in Figure \ref{fig:penta}. Each
	vertex label $ijk$ refers to the root $d_i+d_j+d_k$.
	The roots $d_i-d_j$ will be denoted by pairs $ij$.
	
	The first ten vertices are root subsystems $A_1$, so they
	correspond to the vertex labels:
	\begin{equation} \label{eq:u1to10} 
		\begin{small}
			\begin{matrix}
				u_1 \!:\! 125 ,\,\,
				u_2 \!:\! 126 , \,\,
				u_3 \!:\! 134 , \,\,
				u_4\!:\! 136 , \,\,
				u_5\!:\! 145 , \,\,
				u_6\!:\! 234 , \,\,
				u_7\!:\! 235 , \,\,
				u_8\!:\! 246 , \,\,
				u_9\!:\! 356, \,\,
				u_{10}\!:\! 456 . 
		\end{matrix} \end{small}
	\end{equation}
	These variables correspond to the ten triangles formed by the
	arrangement of six lines in $\P^2$ that is shown in Figure \ref{fig:sixlines}.
	This correspondence was given by Sekiguchi in \cite[Figure III]{Sek}.

	The other five vertices of $\mathcal{G}(E_6)$
	are root subsystems $A_2^{\times 3}$, one for each color
	class of edges in Figure \ref{fig:penta}. The three edges
	in a color class are the factors $A_2$, with three roots
	$ij, ikl, jkl$:
	\begin{equation} \label{eq:u11to15} 
		\!\!\!\begin{small} 
			\begin{matrix}
				u_{11}\!:\! \{\underbar{12},134,234, \, \underbar{56},125,126 ,  \, \underbar{34},356,456 \}  ,&
				u_{12}\! :\! \{\underbar{13},125,235, \, \underbar{46},134,136 ,  \, \underbar{25},246,456 \} ,\\
				u_{13}\! :\! \{\underbar{14},126,246, \,   \underbar{35},134,145 ,\,\underbar{26},235,356 \} ,  \\
				u_{14} \!:\! \{\underbar{15},136,356, \, \underbar{24},125,145, \,\underbar{36},234,246 \} , &
				u_{15} \!:\! \{\underbar{16},145,456, \, \underbar{23}, 126,136, \, \underbar{45},234,235 \}. &  \\
		\end{matrix} \end{small}
	\end{equation}
	The $15$ edges of the Petersen graph are partitioned into
	five triples. Each triple of edges is pairwise disjoint in the graph, and its
	underlined labels are  also disjoint. See \cite[Figure~2]{SY}.
	
	The graph $\mathcal{G}(E_6)$ has $60$ edges, which come in two groups,
	namely $30$ edges of type $\{A_1,A_1\}$ and $30$ edges of type $\{A_1,A_2^{\times 3}\}$.
	Those of type $\{A_1,A_1\}$ are the non-edges in the Petersen graph.
	For instance, $\{u_1,u_3\}$ is an edge of $\mathcal{G}(E_6)$ because
	$u_1{:}125$ and  $u_3{:}134$ are not adjacent in the Petersen graph.
	The edges of type $\{A_1,A_2^{\times 3}\}$ arise from the $30 = 5 \cdot 3 \cdot 2$ inclusions
	of a root $ijk$ in a system $A_2^{\times 3}$. For instance,
	$\{u_1,u_{14}\}$ is such an edge because
	$u_1{:} 125$ is among the nine roots in $u_{14}$.
	We find that $\mathcal{G}(E_6)$ has $90$ cliques of
	size three and $45$ cliques of size four, and no larger cliques,
	which yields the first f-vector in Theorem \ref{thm:2polytopes}.
	
	\begin{figure}
		\centering
		\begin{subfigure}{.5\textwidth}
			\centering
			\includegraphics[width=1.1\linewidth]{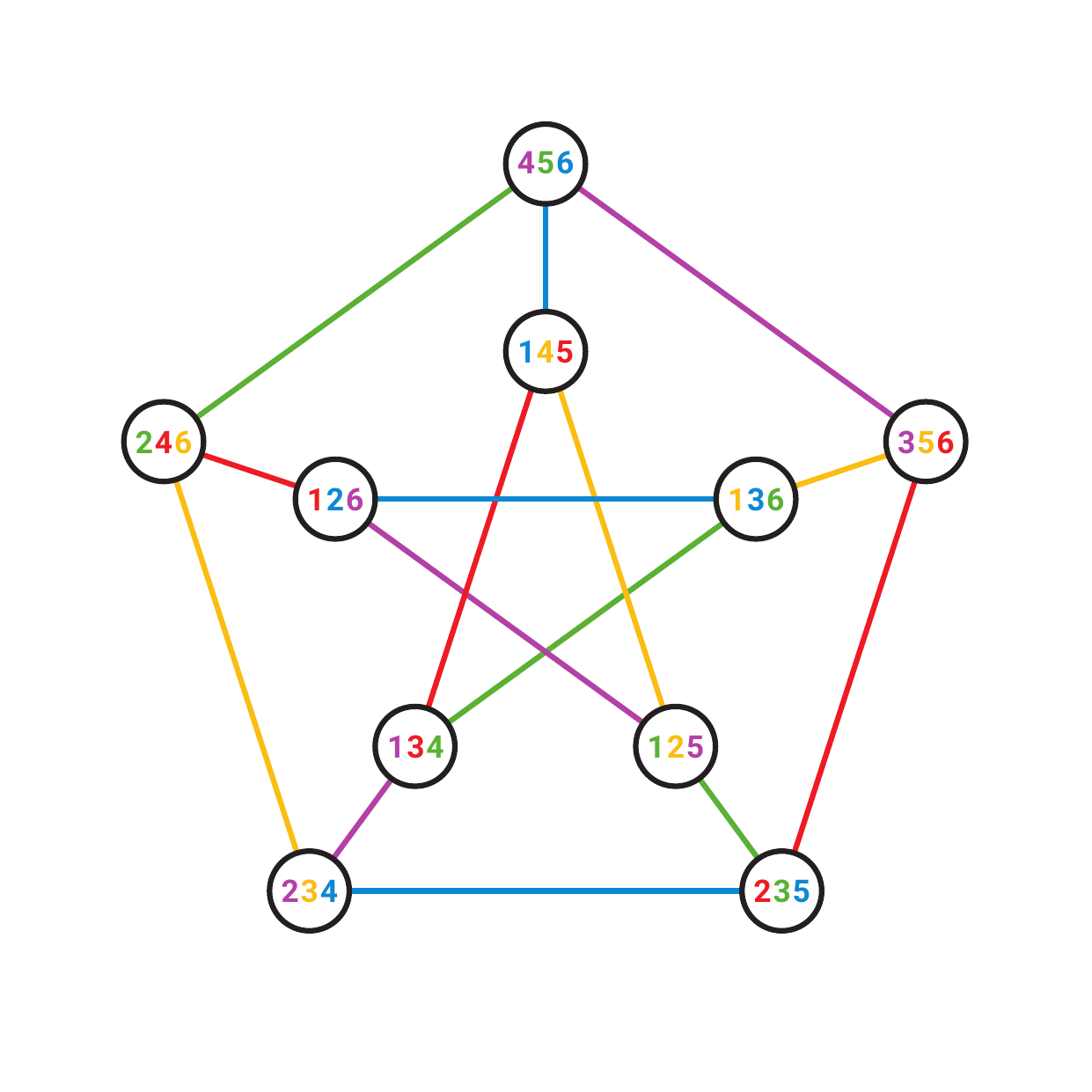}
			\vspace{-1cm}
			\caption{}
			\label{fig:penta}
		\end{subfigure}%
		\begin{subfigure}{.5\textwidth}
			\centering
			\includegraphics[width=1.0\linewidth]{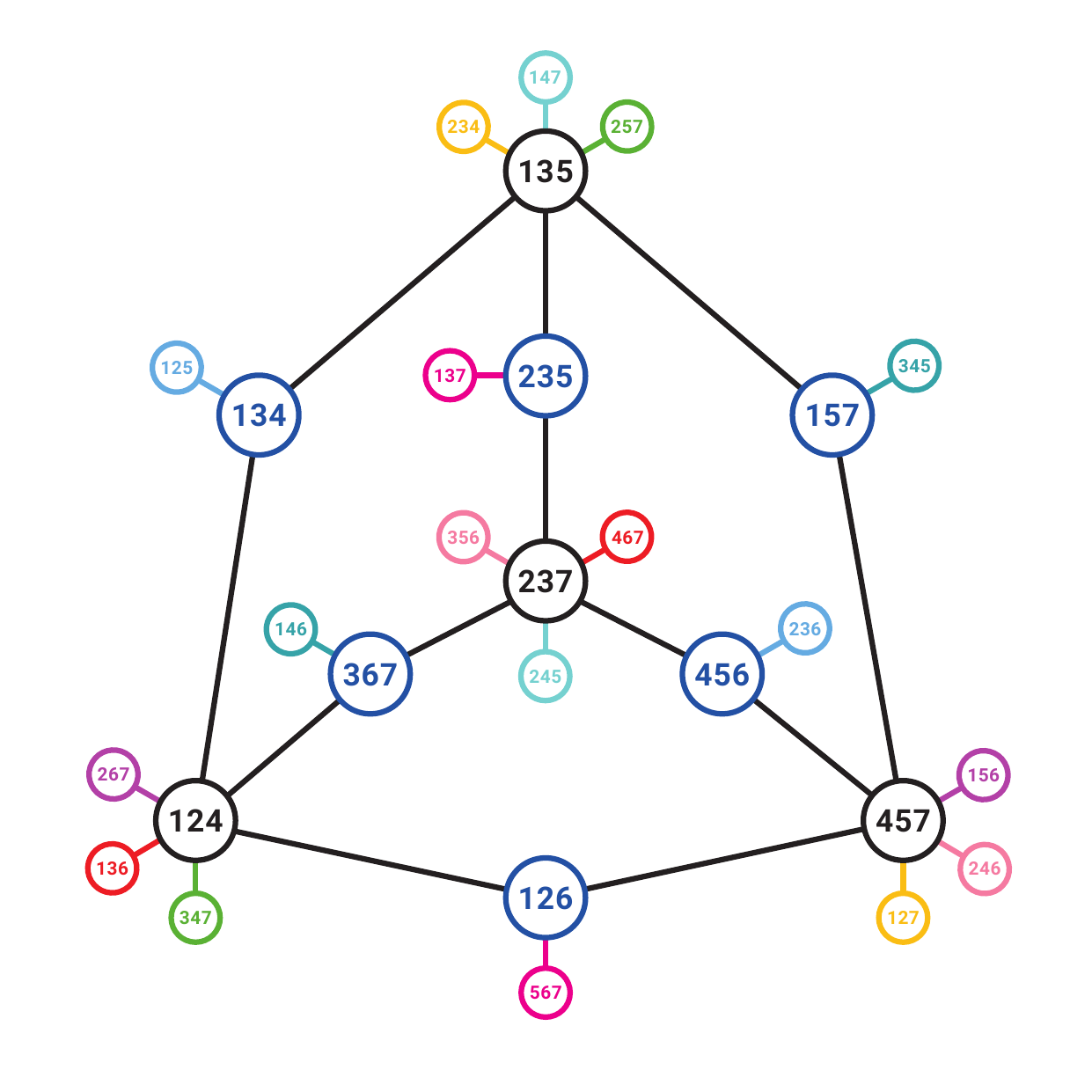}
			\vspace{-1cm}
			\caption{}
			\label{fig:tetra}
		\end{subfigure}
		\caption{Two graphs that reveal the combinatorics of the pezzotopes
			for $n=6$ (left) and $n=7$ (right).  These diagrams are color-enhanced
			reproductions of those in \cite[Figure 3]{HKT}.}
		\label{fig:drei}
	\end{figure}
	
	We now offer two presentations of our pezzotope
	in the setting of the physics literature discussed
	in Section \ref{sec7}. The first is a variety
	as in (\ref{eq:ueqns2}). The second is
	an amplitude as in~(\ref{eq:14terms}).
	
	\begin{thm}\label{thm: E6amplitude}
			The following
			$u$-equations define a perfect binary geometry
			(cf.~\cite{AHLT}):
			\begin{equation}
				\label{eq:perfectuE6}  \begin{small} \begin{matrix}
						u_1 + u_2 u_5 u_7 u_{13} u_{15} \,=\,
						u_2 + u_1 u_4 u_8 u_{12} u_{14} \,=\,
						u_3 + u_4 u_5 u_6 u_{14} u_{15} \,=\,
						u_4 + u_2 u_3 u_9 u_{11} u_{13} \,= \\
						u_5 + u_1 u_3 u_{10} u_{11} u_{12}\, =\,
						u_6 + u_3 u_7 u_8 u_{12} u_{13} \,=\,
						u_7 + u_1 u_6 u_9 u_{11} u_{14} \,=\,
						u_8 + u_2 u_6 u_{10} u_{11} u_{15} \,= \,\\
						u_9 + u_4 u_7 u_{10} u_{12} u_{15} \,\,= \,\,
						u_{10} + u_5 u_8 u_9 u_{13} u_{14} \,\,=\,\,
						u_{11} + u_4 u_5 u_7 u_8 u_{12} u_{13} u_{14} u_{15}\,\, = \\
						u_{12} + u_2 u_5 u_6 u_9 u_{11} u_{13} u_{14} u_{15} \quad = \quad
						u_{13} + u_1 u_4 u_6 u_{10} u_{11} u_{12} u_{14} u_{15} \quad=\,\, \\ \qquad
						u_{14} + u_2 u_3 u_7 u_{10} u_{11} u_{12} u_{13} u_{15} \quad = \quad
						u_{15} + u_1 u_3 u_8 u_9 u_{11} u_{12} u_{13} u_{14} \quad = \quad 1.
				\end{matrix} \end{small}
			\end{equation}
			The solution set in $\mathbb{R}_{\geq 0}^{15}$ is
			combinatorially isomorphic to the $E_6$ pezzotope. 
			The  facets given by setting $u_1,u_2,\ldots,u_{10}$ to $0$ are associahedra, and the  facets 
			for $u_{11} ,u_{12},\ldots,
			u_{15}$ are cubes. 
			The $E_6$ amplitude $\mathcal{A}_{E_6}$ is the following sum over $45$ terms, one for each vertex of the
			pezzotope:
			\begin{equation}
				\label{eq:E6amplitude} \!\!\!\!\! \!\! \begin{small} \begin{matrix}
						{\frac {1}{{s_1}\,{s_3}\,{s_8}\,{s_9}}}+{\frac {1}{{s_1
								}\,{s_3}\,{s_8}\,{s_{12}}}}+{\frac {1}{{s_1}\,{s_3}\,{
									s_9}\,{s_{11}}}}+{\frac {1}{{s_1}\,{s_3}\,{s_{10}}\,{s_{11}}}}+{\frac {1}{{s_1}\,{s_3}\,{s_{10}}\,{s_{12}}}}+{\frac {1
							}{{s_1}\,{s_4}\,{s_6}\,{s_{10}}}}+{\frac {1}{{s_1}\,{
									s_4}\,{s_6}\,{s_{14}}}}+{\frac {1}{{s_1}\,{s_4}\,{s_8}\,{
									s_{12}}}}\smallskip \\ +{\frac {1}{{s_1}\,{s_4}\,{s_8}\,{s_{14}}}}+{
							\frac {1}{{s_1}\,{s_4}\,{s_{10}}\,{s_{12}}}}+{\frac {1}{{
									s_1}\,{s_6}\,{s_9}\,{s_{11}}}}+{\frac {1}{{s_1}\,{s_6}\,{
									s_9}\,{s_{14}}}}+{\frac {1}{{s_1}\,{s_6}\,{s_{10}}\,{
									s_{11}}}}+{\frac {1}{{s_1}\,{s_8}\,{s_9}\,{s_{14}}}}+{\frac {1
							}{{s_2}\,{s_3}\,{s_7}\,{s_{10}}}}{+}{\frac {1}{{s_2}\,{
									s_3}\,{s_7}\,{s_{13}}}}\smallskip \\ +{\frac {1}{{s_2}\,{s_3}\,{s_9}\,{
									s_{11}}}}+{\frac {1}{{s_2}\,{s_3}\,{s_9}\,{s_{13}}}}+{
							\frac {1}{{s_2}\,{s_3}\,{s_{10}}\,{s_{11}}}}+{\frac {1}{{
									s_2}\,{s_5}\,{s_6}\,{s_9}}}+{\frac {1}{{s_2}\,{s_5}\,{
									s_6}\,{s_{15}}}}+{\frac {1}{{s_2}\,{s_5}\,{s_7}\,{s_{13}
						}}}+{\frac {1}{{s_2}\,{s_5}\,{s_7}\,{s_{15}}}}+{\frac {1}{{
									s_2}\,{s_5}\,{s_9}\,{s_{13}}}}\smallskip \\ +{\frac {1}{{s_2}\,{s_6}
								\,{s_9}\,{s_{11}}}}+{\frac {1}{{s_2}\,{s_6}\,{s_{10}}\,{
									s_{11}}}}+{\frac {1}{{s_2}\,{s_6}\,{s_{10}}\,{s_{15}}}}+{
							\frac {1}{{s_2}\,{s_7}\,{s_{10}}\,{s_{15}}}}+{\frac {1}{{
									s_3}\,{s_7}\,{s_8}\,{s_{12}}}}+{\frac {1}{{s_3}\,{s_7}\,{
									s_8}\,{s_{13}}}}{+}{\frac {1}{{s_3}\,{s_7}\,{s_{10}}\,{
									s_{12}}}} {+} {\frac {1}{{s_3}\,{s_8}\,{s_9}\,{s_{13}}}} \smallskip \\ +{\frac {1
							}{{s_4}\,{s_5}\,{s_6}\,{s_{14}}}}+{\frac {1}{{s_4}\,{
									s_5}\,{s_6}\,{s_{15}}}}+{\frac {1}{{s_4}\,{s_5}\,{s_7}\,{
									s_8}}}+{\frac {1}{{s_4}\,{s_5}\,{s_7}\,{s_{15}}}}+{\frac 
							{1}{{s_4}\,{s_5}\,{s_8}\,{s_{14}}}}+{\frac {1}{{s_4}\,{
									s_6}\,{s_{10}}\,{s_{15}}}}+{\frac {1}{{s_4}\,{s_7}\,{s_8
								}\,{s_{12}}}}+{\frac {1}{{s_4}\,{s_7}\,{s_{10}}\,{s_{12}}}} \smallskip \\ +{
							\frac {1}{{s_4}\,{s_7}\,{s_{10}}\,{s_{15}}}}+{\frac {1}{{\it 
									s_5}\,{s_6}\,{s_9}\,{s_{14}}}}+{\frac {1}{{s_5}\,{s_7}\,{
									s_8}\,{s_{13}}}}+{\frac {1}{{s_5}\,{s_8}\,{s_9}\,{s_{13}
						}}}+{\frac {1}{{s_5}\,{s_8}\,{s_9}\,{s_{14}}}}.
				\end{matrix} \end{small} 
			\end{equation}
			The same data
			($u$-equations and CEGM amplitude) are available
			for the $E_7$ pezzotope -- and displayed in Theorem \ref{thm:E7pezzo} --
			but we do not yet know that they define a perfect binary geometry.
		\end{thm}
		
		\begin{proof}
			Each equation in (\ref{eq:perfectuE6}) has a variable $u_i$ and a monomial that is the product over all $u_j$
			not adjacent to $u_i$ in the graph $\mathcal{G}(E_6)$.
			One checks computationally that the ideal generated by (\ref{eq:perfectuE6})
			is prime, and its variety in $\C^{15}$ has dimension $4$ and degree $192$. The variety 
			in $\mathbb{R}_{\geq 0}^{15}$ admits a stratification that induces a combinatorial isomorphism to the pezzotope $E_6$. For example, the stratum obtained by restricting to $u_1 =0$ is a curvy associahedron. Restricting further to $u_3 = 0$ reveals a curvy pentagonal face 
			(cf.~Section \ref{sec2}) of that associahedron. 
			
			We compute the CEGM amplitude by summing (\ref{eq:etude})
			over all $32$ critical points of 
			\eqref{eq:loglike6}:
			\begin{eqnarray}\label{eq: CEGMVeroneseAmp}
				\sum_{c\in \text{Crit}(L)}^{32}\frac{1}{\det(\Phi)}\left(\frac{p_{135}}{p_{123}p_{345}p_{561}q}\right)^2\bigg\vert_{c}.
			\end{eqnarray}
			Here, $\Phi$ is the toric Hessian of the scattering potential $L$,
			and   the coefficients $s_{ijk},t$ in  \eqref{eq:loglike6}
			are replaced by $s_1,s_2,\ldots,s_{15}$ using the rule in
			Remark \ref{rmk:substitute}.
			Then (\ref{eq: CEGMVeroneseAmp}) is a rational function in $s_1,s_2,\ldots,s_{15}$.
			We show by numerical evaluations that this rational function equals
			(\ref{eq:E6amplitude}).
			
			The derivation for $n=7$ is analogous, using the data below. We refer to Theorem
			\ref{thm:E7pezzo}.
		\end{proof}
		
		We now turn to $n=7$. The graph $\mathcal{G}(E_7)$ has $34$ vertices:
		ten of type $A_1$, twelve of type $A_2$,
		nine of type $A_3^{\times 2}$, and three of type $A_7$.
		The first ten are the  labels in Figure \ref{fig:tetra}:
		\begin{equation} \label{eq:u1to10E7}
			\begin{small}
				\begin{matrix}
					u_1 \!:\! 124, \,\,
					u_2 \!:\!126 ,\,\,
					u_3 \!:\!134 , \,\,
					u_4 \!:\!135 ,\,\,
					u_5 \!:\!157 ,\,\,
					u_6 \!:\!235 ,\,\,
					u_7 \!:\!237 ,\,\,
					u_8 \!:\!367 ,\,\,
					u_9 \!:\!456,\,\,
					u_{10} \!:\! 457.
			\end{matrix} \end{small}
		\end{equation}
		The twelve vertices $A_2$ correspond to the pairs from (\ref{eq:u1to10E7}) that are connected by an edge:
		\begin{equation} \label{eq:u11to22E7}
			\begin{small}
				\begin{matrix}
					u_{11} : \{ \underbar{12} , 135 , 235 \} , & &
					u_{12} : \{ \underbar{14} , 157 , 457 \} , & &
					u_{13} : \{ \underbar{23} , 124 , 134 \} , \\
					u_{14} : \{ \underbar{26} , 237 , 367 \} , & &
					u_{15} : \{ \underbar{37} , 135 , 157 \} , & &
					u_{16} : \{ \underbar{45} , 134 , 135 \} ,\\
					u_{17} : \{ \underbar{46} , 124 , 126 \} , &&
					u_{18} : \{ \underbar{57} , 235 , 237 \} , &&
					u_{19} : \{ \underbar{67} , 456 , 457 \}, \\
					u_{20} : \{ \underbar{1} , 237 , 456 \} , &&
					u_{21} : \{ \underbar{3} , 126 , 457 \} , &&
					u_{22} : \{ \underbar{5} , 124 , 367 \} .
				\end{matrix}
			\end{small}
		\end{equation}
		In the last row,  \underbar{$i$} is the root $-d_i + \sum_{j=1}^7 d_j$.
		Each root subsystem $A_3^{\times 2}$ has $12$ roots. These nine vertices of $\mathcal{G}(E_7)$
		come in two types.
		The first six come from pairs of trivalent nodes:
		\begin{equation} \label{eq:u23to28E7}
			\begin{small}
				\begin{matrix}
					u_{23}: \{\underbar{124} ,367,5,347,46,126,\,    
					\underbar{135}, 235,12,257,37,157
					\} \\
					u_{24}: \{ \underbar{124}, 134,23,136,46,126  ,\, \underbar{237}, 235,57,467,1,456  \} \\
					u_{25}: \{ \underbar{124},  134,23,267,5,367, \, \underbar{457}, 157,14,156,67,456 \} \\
					u_{26}: \{\underbar{135},   157,37,147,45,134 ,\,  \underbar{237}, 456,1,245,26,367 \} \\
					u_{27}: \{\underbar{135},   235,12,234,45,134, \, \underbar{457}, 456,67,127,3,126 \} \\
					u_{28}: \{\underbar{237},   235,57,356,26,367, \, \underbar{457}, 157,14,246,3,126 \} \\
				\end{matrix}
			\end{small}
		\end{equation}
		We remark that label $167$ in \cite[Figure 5]{HKT} is incorrect. We corrected it to $156$ in $u_{25}$ above.
		
		The remaining three vertices of type $A_3^{\times 2}$ come from pairs of antipodal
		bivalent nodes:
		\begin{equation} \label{eq:u29to31E7}
			\begin{small}
				\begin{matrix}
					u_{29}: \{\underbar{126} ,  124,46,567,\,3,457,
					\,        \underbar{235}, 135,12,137,57,237  \} \\
					u_{30}: \{\underbar{134} ,  124,23,125,45,135,\,   
					\underbar{456},  237,1,236,67,457\} \\
					u_{31}: \{\underbar{157} , 135,37,345,14,457,\,
					\underbar{367}, 124,5,146,26,237 \} \\
				\end{matrix}
			\end{small}
		\end{equation}
		Finally, the graph $\mathcal{G}(E_7)$ has three vertices that correspond to root subsystems 
		of type $A_7$:
		\begin{equation} \label{eq:u32to34E7}
			\begin{scriptsize}
				\begin{matrix}
					u_{32}: \{ \underbar{126}, 124, \underbar{134},135,\underbar{235},237 ,\underbar{456}, 457,
					567, 46, 136, 23, 125, 45, 234, 12, 137, 57, 467, 1, 236, 67, 127, 3, 13, 56, 2, 47
					\}\smallskip \\u_{33}: \{ \underbar{126},457, \underbar{157},135, \underbar{235}, 237,\underbar{367},124,
					567, 3, 246, 14, 345, 37, 257, 12, 137, 57, 356, 26, 146, 5, 347, 46, 35, 24, 7, 16
					\} \smallskip \\
					u_{34}: \{\underbar{134},135,\underbar{157},457,237, \underbar{456},  \underbar{367},124,
					125, 45, 147, 37, 345, 14, 156, 67, 236, 1, 245, 26, 146, 5, 267, 23, 15, 36, 4, 27
					\}
				\end{matrix}
			\end{scriptsize}
		\end{equation}
		The first eight roots in each $A_7$ form a cycle of length $8$ in Figure~\ref{fig:tetra}.

		The graph $\mathcal{G}(E_7)$ has $297$ edges, to be divided into several groups. There are $33$ edges of type  $\{A_1,A_1\}$, corresponding to 
		non-edges in Figure~\ref{fig:tetra}. The $24$ edges of type $\{A_2,A_2\}$ arise from the inclusion 
		of the pair of root systems in a common $A_3^{\times 2}$, but not in a common $A_3$. The edges $\{A_1, A_2\}$ come in two groups: 
		$24$ from the inclusion of a root system $ijk$ in $A_2$ and $60$ from the pair not being contained in a common $A_3^{\times 2}$. There are $6$ pairs of type $\{A_1, A_3^{\times 2}\}$ not contained in a common $A_7$ and such that any $A_2$ containing $A_1$ is disjoint from $A_3^{\times 2}$.
		Finally, we have five more groups of edges $\{A_1, A_3^{\times 2}\}$, $\{A_1, A_7\}$, $\{A_2, A_3^{\times 2}\}$, $\{A_2, A_7\}$ and
		$\{A_3^{\times 2}, A_7\}$, respectively  of size $54$, $24$, $36$, $24$ and $12$, from the inclusion of the 
		corresponding pairs. The clique complex of   $\mathcal{G}(E_7)$ yields the second f-vector
		presented in Theorem \ref{thm:2polytopes}.
		
		\section{Geometry of Pezzotopes}
		\label{sec9}
		
		In Section \ref{sec8} we offered a detailed combinatorial description of the two pezzotopes.
		We now return to the geometry of the moduli spaces they represent, 
		starting with the algebraic expressions in Theorem \ref{thm: E6amplitude}.
		As always, our guiding principle is the trinity of real, complex and tropical shapes
		in algebraic geometry.
		This section also harbors the proof of Theorem~\ref{thm:2polytopes}.
		
		In Theorem \ref{thm: E6amplitude}
		we displayed an embedding of $Y(3,6)$ as a very affine variety in
		$(\C^*)^{15}$.
		The next result gives a parametric representation of our moduli space
		in that realization.
		
		\begin{prop} \label{prop:uparaE6}
			The variety defined by the u-equations  in (\ref{eq:perfectuE6}) has the parametrization
			$$
			\begin{matrix}
				\! u_1 = \frac{-q}{p_{126}p_{135}p_{234}p_{456}},
				&\!\! \!u_2 = \frac{p_{134}p_{156}p_{235}p_{246}}{p_{135}p_{146}
					p_{234}p_{256}} ,&  u_3 = \frac{p_{134}p_{356}}{p_{135}p_{346}}, & 
				\! u_4 = \frac{p_{136}
					p_{145}}{p_{135}p_{146}}, &     
				\!\! u_5 = \frac{p_{125}p_{136}p_{246}p_{345}}{p_{126}p_{135}p_{245} p_{346}}, \smallskip \\
				u_6 = \frac{p_{136}p_{235}}{p_{135}p_{236}}, & \!\! \! u_7 = \frac{p_{123}p_{145}p_{246}
					p_{356}}{p_{124}p_{135}p_{236}p_{456}} ,& u_8 = \frac{p_{125}p_{356}}{p_{135}p_{256}},
				& \! u_9 = \frac{p_{125}p_{134}}{p_{124}p_{135}}, &
				u_{10} = \frac{p_{145}p_{235}}{p_{135} p_{245}} , \smallskip \\
				u_{11} = \frac{p_{135}p_{234}}{p_{134}p_{235}}, &
				u_{12} = \frac{p_{135}p_{456}}{p_{145} p_{356}}, & \!
				\!\!\!\! u_{13} = \frac{p_{124}p_{135}p_{256}p_{346}}{p_{125}p_{134}p_{246} p_{356}}, \! &
				\!\! u_{14} = \frac{p_{126}p_{135}}{p_{125}p_{136}} ,&
				\!\!   u_{15} = \frac{p_{135}p_{146}p_{236}  p_{245}}{p_{136}p_{145}p_{235}p_{246}}.
			\end{matrix}
			$$
			Replacing the Pl\"ucker coordinates $p_{ijk}$ with the $3 \times 3$ minors of the matrix (\ref{eq:dimatrix}),
			we obtain
			$$ \begin{scriptsize}
				\begin{matrix}  
					u_1 & = &
					(d_6{-}d_3) (d_2{-}d_5) (d_1{-}d_4) (d_1{+}d_2{+}d_3{+}d_4{+}d_5{+}d_6)/((d_4{+}d_5{+}d_6)(d_2{+}d_3{+}d_4)(d_1{+}d_3{+}d_5)(d_1{+}d_2{+}d_6)) , \\
					u_2 & = & (d_1{+}d_3{+}d_4) (d_1{+}d_5{+}d_6) (d_2{+}d_3{+}d_5) (d_2{+}d_4{+}d_6)/((d_2{+}d_5{+}d_6)(d_2{+}d_3{+}d_4)(d_1{+}d_4{+}d_6)(d_1{+}d_3{+}d_5)) , \\
					u_3 & = & (d_1{-}d_4) (d_1{+}d_3{+}d_4) (d_5{-}d_6) (d_3{+}d_5{+}d_6)/((d_4{-}d_6)(d_3{+}d_4{+}d_6)(d_1{-}d_5)(d_1{+}d_3{+}d_5)) , \\
					u_4 & = & (d_3{-}d_6) (d_1{+}d_3{+}d_6) (d_4{-}d_5) (d_1{+}d_4{+}d_5)/((d_4{-}d_6)(d_1{+}d_4{+}d_6)(d_3{-}d_5)(d_1{+}d_3{+}d_5)) , \\
					u_5 & = & (d_1{+}d_2{+}d_5) (d_1{+}d_3{+}d_6) (d_2{+}d_4{+}d_6) (d_3{+}d_4{+}d_5)/((d_3{+}d_4{+}d_6)
					(d_2{+}d_4{+}d_5)(d_1{+}d_3{+}d_5)(d_1{+}d_2{+}d_6)) , \\ 
					u_6 & = & (d_1{-}d_6) (d_1{+}d_3{+}d_6) (d_2{-}d_5) (d_2{+}d_3{+}d_5)/((d_2{-}d_6)(d_2{+}d_3{+}d_6)(d_1{-}d_5)(d_1{+}d_3{+}d_5)) , \\
					u_7 & = & (d_1{+}d_2{+}d_3) (d_1{+}d_4{+}d_5) (d_2{+}d_4{+}d_6) (d_3{+}d_5{+}d_6)/((d_4{+}d_5{+}d_6)(d_2{+}d_3{+}d_6)(d_1{+}d_3{+}d_5)(d_1{+}d_2{+}d_4)) , \\
					u_8 & = & (d_1{-}d_2) (d_1{+}d_2{+}d_5) (d_3{-}d_6) (d_3{+}d_5{+}d_6)/((d_2{-}d_6)
					(d_2{+}d_5{+}d_6)(d_1{-}d_3)(d_1{+}d_3{+}d_5)) , \\
					u_9 & = & (d_2{-}d_5) (d_1{+}d_2{+}d_5) (d_3{-}d_4) (d_1{+}d_3{+}d_4)/((d_3{-}d_5)
					d_1{+}d_3{+}d_5)(d_2{-}d_4)(d_1{+}d_2{+}d_4)) , \\
					u_{10} & = & (d_1{-}d_4) (d_1{+}d_4{+}d_5) (d_2{-}d_3) (d_2{+}d_3{+}d_5)/((d_2{-}d_4)(d_2{+}d_4{+}d_5)
					(d_1{-}d_3)(d_1{+}d_3{+}d_5)) , \\
					u_{11} & = & (d_1{-}d_5) (d_1{+}d_3{+}d_5) (d_2{-}d_4) (d_2{+}d_3{+}d_4)/((d_2{-}d_5)
					(d_2{+}d_3{+}d_5)(d_1{-}d_4)(d_1{+}d_3{+}d_4)) , \\
					u_{12} & = & (d_1{-}d_3) (d_1{+}d_3{+}d_5) (d_4{-}d_6) (d_4{+}d_5{+}d_6)/((d_3{-}d_6)(d_3{+}d_5{+}d_6)
					(d_1{-}d_4)(d_1{+}d_4{+}d_5)) , \\
					u_{13} & = & (d_1{+}d_2{+}d_4) (d_1{+}d_3{+}d_5) (d_2{+}d_5{+}d_6) (d_3{+}d_4{+}d_6)/
					((d_3{+}d_5{+}d_6)(d_2{+}d_4{+}d_6)(d_1{+}d_3{+}d_4)(d_1{+}d_2{+}d_5)) , \\
					u_{14} & = & (d_2{-}d_6) (d_1{+}d_2{+}d_6) (d_3{-}d_5) (d_1{+}d_3{+}d_5)/
					((d_3{-}d_6)(d_1{+}d_3{+}d_6)(d_2{-}d_5)(d_1{+}d_2{+}d_5)) , \\
					u_{15} & = & (d_1{+}d_3{+}d_5) (d_1{+}d_4{+}d_6) (d_2{+}d_3{+}d_6) (d_2{+}d_4{+}d_5)
					/((d_2{+}d_4{+}d_6)(d_2{+}d_3{+}d_5)(d_1{+}d_4{+}d_5)(d_1{+}d_3{+}d_6)).
				\end{matrix}
			\end{scriptsize}
			$$
		\end{prop}
		
		\begin{proof}
			The proof is by computation. One checks that  that the 
			above expressions for $u_1,\ldots,u_{15}$
			satisfy the $15$ equations in (\ref{eq:perfectuE6}). Further, their
			Jacobian matrix has rank $4$, so they parametrize an irreducible
			$4$-dimensional variety in $(\C^*)^{15}$.
			This variety is $Y(3,6)$ by 
			Theorem \ref{thm: E6amplitude}.
		\end{proof}
		
		\begin{quest}
			We expressed each $u_i$
			as a ratio of products of four roots of $E_6$.
			Such expressions were studied by Hacking, Keel 
			and Tevelev \cite[Theorem 8.7]{HKT}, who called them
			{\em $D_4$-units}. They derive $D_4$-units also for
			$E_7$. How to use their results for parametrizing
			the inclusion $Y(3,7) \subset (\C^*)^{34}$ that is suggested by the
			$u$-equations in (\ref{eq:U37eqns})
			for the $E_7$ pezzotope?
		\end{quest}
		
		We now turn to the CEGM amplitudes arising from del Pezzo moduli.
		For $E_6$, the expression
		as a rational function in $s_1,\ldots,s_{15}$ was displayed in (\ref{eq:E6amplitude}).
		For $E_7$, the formula is analogous but much larger, with 
		$579$ summands of degree $-6$ in $34$ unknowns $s_1,\ldots,s_{34}$. Our readers may
		rapidly generate it with the {\tt Macaulay2} code that is shown in Theorem \ref{thm:E7pezzo}.
		
		The unknowns in the amplitude (\ref{eq:E6amplitude}) are the parameters in the log-likelihood function $L$
		on $Y(3,6)$. However, now we use 
		embedding of $Y(3,6)$ into $(\C^*)^{15}$ via  the $u$-equations (\ref{eq:perfectuE6}):
		\begin{equation}
			\label{eq:loglikeu}
			L \,\, = \,\, s_1 \cdot {\rm log}(u_1) \,+ \,s_2 \cdot {\rm log}(u_2) \, + \, \cdots \,+\,
			s_{15} \cdot {\rm log}(u_{15}).
		\end{equation}
		We are interested in the critical points of (\ref{eq:loglikeu}) on the
		variety defined by the $15$ equations in~(\ref{eq:perfectuE6}).
		From a maximum likelihood perspective, we now face a
		constrained optimization problem. 
		We already know (from Theoren \ref{thm:euler})
		that this  problem has $32$
		complex critical points, but it is very hard to find these using
		Lagrange multipliers in the unconstrained formulation.
		It is easier to compute the $32$ critical points of  (\ref{eq:loglikeu})
		with the unconstrained formulation that is obtained from the parametrization in (\ref{prop:uparaE6}).
		That computation is equivalent to the one in
		local coordinates in (\ref{eq:loglike6}) and  also to the one via
		lifting to $\mathcal{A}(E_6)$ given in (\ref{eq:logliked}).
		We next present the transformation which makes these correspondences completely explicit.
		
		\begin{rmk} \label{rmk:substitute}
			If we substitute the Pl\"ucker parametrization from Proposition \ref{prop:uparaE6} into (\ref{eq:loglikeu}) then
			we obtain the expression for $L$ in (\ref{eq:loglike6}). However, now each Mandelstam invariant $s_{ijk}$
			is given as an integer linear combination of $s_1,\ldots,s_{15}$.
			Explicitly, we obtain the formulas 
			$$ s_{123} = s_7, s_{124} = -s_7 + s_9 + s_{13}, 
			s_{125} = s_5+s_8+s_9-s_{13}-s_{14}, \ldots, s_{456} = -s_1-s_7+s_{12} ,\, t = s_1.
			$$
			These $21$ linear forms in $15$  unknowns $s_1,\ldots,s_6$
			satisfy six independent linear constraints, and these 
			must be satisfied when using (\ref{eq:loglike6}) in Pl\"ucker coordinates. For instance, we have
			$$ s_{125}+s_{135}+s_{145}+s_{156}+s_{235}+s_{245}+s_{256}+s_{345}+s_{356}+s_{456}+2 t \,=\, 0 .$$
			Similarly, if we substitute the $D_4$-units from
			Proposition \ref{prop:uparaE6} into (\ref{eq:loglikeu}) then
			we obtain the expression in (\ref{eq:logliked}), along with the linear constraints 
			its coefficients $s_{ij}, t_{ijk},v$ must satisfy. Recall that these had been expressed
			by the matrix $M_6$ when we stated Proposition \ref{prop:86400}.
			Everything in this remark extends directly to the case $n=7$, but all expressions are larger.
		\end{rmk}
		
\noindent		We now complete the thread started at the beginning of the previous section.
		
		\begin{proof}[Proof and discussion of Theorem \ref{thm:2polytopes}]
			The relevant structures were discovered~by Sekiguchi and Yoshida 
			\cite{Sekolder, Sek, SY}.
			They discussed the facets of the pezzotopes, but they did not
			give a list of faces. Also, Sekiguchi states the $n=7$ result without proof.
			Their work was extended by Hacking, Keel and Tevelev \cite{HKT},
			whose description we relied on in Section~\ref{sec8}.
			We computed
			the number of connected components ($432$ resp.~$60480$)
			as the order of $W(E_n)$, modulo center,
			divided by the order of the automorphism group
			($S_5$ resp.~$S_4$) of the labeled diagrams in Figure~\ref{fig:drei}.
			That the Weyl group $W(E_n)$ acts transitively on the connected
			components of $Y(3,n)$ can be found in \cite[Theorem 2]{SY}
			for $n=6$ and in \cite[Theorem~1]{Sekolder} for $n=7$.			
			
			We also derived the pezzotopes from
			data in \cite{RSSS, RSS1}.
			The Yoshida variety $\mathcal{Y}$ lives in $\P^{39}$,
			and the~G\"opel variety~$\mathcal{G}$ lives in $\P^{134}$,
			namely in the dense tori, by Remark~\ref{rmk:YG}.			
			The Weyl groups act 
			by permuting coordinates. Their tropicalizations
			were computed in \cite[Lemma 3.1]{RSS2}.
			Namely, ${\rm trop}(\mathcal{Y}) $ is a $4$-dimensional fan
			with $76$ rays and $1275$ maximal cones, while
			${\rm trop}(\mathcal{G}) $ is a $6$-dimensional fan with
			$1065$ rays and $547155$ maximal cones.
			The fibers of the surjection
			${\rm trop}(\mathcal{G}) \rightarrow {\rm trop}(\mathcal{Y})$
			are the tropical cubic surfaces that appear
			in  \cite[Section~5]{RSS2}.
			
			We explored the positive tropical varieties  of $\mathcal{Y}$
			and~$\mathcal{G}$.
			Linear and binomial generators for their ideals are listed in
			\cite{RSSS} and \cite{RSS1}. These furnish
			necessary conditions for membership in ${\rm trop}_+(\mathcal{Y}) $ and $ {\rm trop}_+(\mathcal{G})$ 
			as follows: if a point $w$  is in the  positive tropical variety, then
			\begin{equation}
				\forall \text{ generators }  f=f_+-f_- \text{ with } f_-,f_+\in\R_{\geq 0}[x_1,\ldots,x_m]:  \text{trop}(f_-)(w)=\text{trop}(f_+)(w).\label{eq:membership}
			\end{equation}
			We also know the tropical linear spaces
			${\rm trop}(\mathcal{A}(E_6))$ and
			${\rm trop}(\mathcal{A}(E_7))$, which map
			to 
			${\rm trop}(\mathcal{Y}) $ and
			${\rm trop}(\mathcal{G}) $ 
			via the horizontal maps in
			\cite[eqn (3.1)]{RSS2}. These are the linear maps given by the
			matrices $M_6 \in \{0,1\}^{36 \times 40}$ and $M_7 \in \{0,1\}^{63 \times 135}$ 
			which we constructed in 
			Section \ref{sec6}. 
			For $n=6,7$,
			we computed the images of the rays of ${\rm trop}(\mathcal{A}(E_n))$ under the map $M_n$.
			Applying the criterion \eqref{eq:membership} to these images, we found $15$ and $34$
			rays expected to live in $ {\rm trop}_+(\mathcal{Y})$ and ${\rm trop}_+(\mathcal{G}) $, respectively.
			To identify higher-dimensional cones,
			we applied the criterion to positive linear combinations of these rays.
			The f-vector of the output agrees with that stated in Theorem~\ref{thm:2polytopes}.
			These and further computational checks provide overwhelming evidence that
			${\rm trop}_+(\mathcal{Y}) $ and $ {\rm trop}_+(\mathcal{G})$
			are simplicial fans which are
			dual to our pezzotopes.		
			
			Let us be even more explicit for $n=6$. 
			Our input is
			the data after Theorem 6.1 in \cite{RSS1}.
			The fan ${\rm trop}(\mathcal{A}(E_6))$
			has $36$ rays of type $A_1$ and $120$ rays of type $A_2$.
			The former map to $36$ distinct rays in $\mathbb{N}^{40}$.
			The latter map to $40$ distinct rays in $\mathbb{N}^{40}$.
			Each fiber in this $3$-to-$1$ map corresponds to 
			a color class in a Petersen graph like
			the one in~Figure \ref{fig:drei}. For the membership
			criterion \eqref{eq:membership} we used the $270$ four-term linear relations
			that cut out the $9$-dimensional linear space in 
			\cite[Theorem 6.1]{RSS1}. The binomials
			are automatically satisfied because our $76$ rays
			lie in the row space of $M_6$.
			Of the $36$ rays of type $A_1$, ten passed the criterion.
			Of the $40$ rays of type $A_2$, five passed.
			Up to the $W(E_6)$-action,
			these correspond to the root subsystems in
			(\ref{eq:u1to10}) and (\ref{eq:u11to15}).
			Among the $\binom{15}{2}=105$ cones spanned by two rays,
			precisely $60$ passed the necessary criterion to be contained in
			${\rm trop}_+(\mathcal{Y})$,
			and similarly for triples and quadruples.
			
			\smallskip
			
			Two other computational methods for independent verification of Theorem \ref{thm:2polytopes}
			will be presented in Section \ref{sec10}. Their input is the familiar data  for
			the  {\em tropical Grassmannian} \cite[Section 4.3]{MS}.
			They rely on modifications (Experiment \ref{expe:method1G36})
			and chirotopes  (Theorem \ref{thm:chirotope}).
			
			\smallskip

			It remains to show that our simplicial fans are normal 
			fans of  convex polytopes. This turned out to be
			more difficult than we had anticipated,
			and the issue is unresolved for $n=7$.
			
			For $n=6$ we obtained help from Moritz Firsching
			who found the following realization of the dual pezzotope:
			\setcounter{MaxMatrixCols}{20}
			\begin{equation}
				\label{eq:firsching}
				\begin{bmatrix}
					\phantom{-}4 & -4 &  \phantom{-}2 &  \phantom{-}2 & -4 & -4
					& \phantom{-} 0 & \phantom{-} 2 & -4 & \phantom{-} 2 &  -1
					& \phantom{-} 4 &  -2 & \phantom{-} 0 & -2 \,\,\, \\
					\phantom{-}0 & -4 & -4 & \phantom{-} 2 & \phantom{-} 4
					&  \phantom{-}2 &  -4 & \phantom{-} 2 &\phantom{-} 2 & -4
					& -1 & -1 & \phantom{-} 0 & \phantom{-} 4 & \phantom{-} 0 \,\,\, \\
					\phantom{-}0 & \phantom{-} 0 & \phantom{-}4 & -4 & \phantom{-} 0 &  -4
					& \phantom{-} 0 & \phantom{-} 4 & \phantom{-} 4 &  -4 & \phantom{-} 0
					& \phantom{-} 0 & \phantom{-}4 & \phantom{-}0 & -4\,\, \,\\
					\phantom{-}0 & \phantom{-} 0 &\phantom{-} 0 & \phantom{-} 4 &\phantom{-}  3
					&  -4 & \phantom{-} 4 & \phantom{-} 4 &  -4 & \phantom{-} 0
					&  -3 & \phantom{-} 3 &\phantom{-} 1 &\phantom{-} 0 &\phantom{-} 1\,\,\, \\
				\end{bmatrix}.
			\end{equation}
			One checks by direct computation that the convex hull of the $15$ columns 
			is a simplicial polytope with f-vector $(15,60,90,45)$,
			and the $45$ facets are precisely the 45 summands in (\ref{eq:E6amplitude}); this polytope is \textit{dual} to the $E_6$ pezzotope. We refer to Theorem \ref{thm:pos_surj_par} and Corollary \ref{cor: MS of E6 pezzotope} for the geometric realization of the $E_6$ pezzotope, as computed directly from a parametrization of one of the connected components of $Y(3,6)$.  
			
			For $n=7$ we would need a $6 \times 34$ matrix with the analogous
			property, but presently we have no such matrix.
			We did verify that our simplicial complex on $34$ vertices
			and $579$ facets is a homology $5$-sphere, so the
			term curvy $6$-polytope is appropriate for its dual.
			Ideally, the polytopality of the pezzotopes should
			follow from general facts about positive tropical varieties,
			or from an argument about
			perfect binary geometries. But this is still missing.
		\end{proof}
		
		We now take a closer look at the $E_7$ pezzotope.
		The following result mirrors the combinatorial content of
		Theorem \ref{thm: E6amplitude}. However, the algebraic content is weaker.
		We believe that the $u$-equations below define a perfect binary geometry,
		but we currently have no proof. In particular, we conjecture that the
		$34$ equations define a $6$-dimensional subvariety of $(\mathbb{C}^*)^{34}$.
		This seems like an excellent challenge for the next
		generation in algebraic geometry software.
		
		\begin{thm} \label{thm:E7pezzo}
			The $E_7$ pezzotope is characterized combinatorially by the
			$34$ $u$-equations 
			\begin{equation}
				\label{eq:U37eqns} \begin{scriptsize} \!\!
					\begin{matrix}
						u_1 + u_2 u_{21} u_{22} u_{23} u_{24} u_{28} u_3 u_{30} u_{32} \,\,=\,\,
						u_2 + u_1 u_{19} u_{20} u_{25} u_{26} u_{27} u_{29} u_{31} u_{33} u_8 =  \\
						u_3 + u_1 u_{11} u_{13} u_{14} u_{16} u_{18} u_{25} u_{33} u_4 u_6 \,\,= \,\,
						u_4 + u_{12} u_{17} u_{19} u_{21} u_{29} u_3 u_5 u_7 u_8 \\
						u_5 + u_{13} u_{20} u_{26} u_{28} u_{30} u_{31} u_{32} u_4 u_6 u_9 \,\,=\,\,
						u_6 + u_{10} u_{12} u_{14} u_{17} u_{19} u_{21} u_{22} u_{24} u_{25} u_{26} u_{27} u_{28} u_{29} u_3 u_{34} u_5 u_8 = \\
						u_7 + u_{10} u_{14} u_{22} u_{24} u_{25} u_{26} u_{27} u_{28} u_{34} u_4 \,\,=\,\,
						u_8 + u_{11} u_{13} u_{14} u_{16} u_{18} u_2 u_{21} u_{22} u_{23} u_{24} u_{25} u_{28} u_{30} u_{32} u_{33} u_4 u_6 = \\
						u_9 + u_{11} u_{14} u_{16} u_{19} u_{21} u_{22} u_{23} u_{24} u_{25} u_{27} u_{29} u_{33} u_5 \,\,=\,\,
						u_{10} + u_{15} u_{16} u_{23} u_{29} u_{30} u_{31} u_{33} u_6 u_7 = \\
						u_{11} + u_{15} u_{17} u_{22} u_{26} u_{27} u_{28} u_{29} u_3 u_{30} u_{31} u_{34} u_8 u_9 \,=\,
						u_{12} + u_{14} u_{15} u_{16} u_{22} u_{23} u_{24} u_{25} u_{26} u_{27} u_{28} u_{29} u_{30} u_{31} u_{33} u_{34} u_4 u_6 =  \\
						u_{13} + u_{14} u_{15} u_{16} u_{17} u_{19} u_{21} u_{22} u_{23} u_{24} u_{25} u_{26} u_{27} u_{28} u_{29} u_3 u_{30} u_{31} u_{33} u_{34} u_5 u_8 = \\
						u_{14} + u_{12} u_{13} u_{17} u_{19} u_{20} u_{21} u_{26} u_{28} u_{29} u_3 u_{30} u_{31} u_{32} u_6 u_7 u_8 u_9 \,\,=\,\,
						u_{15} + u_{10} u_{11} u_{12} u_{13} u_{19} u_{20} u_{21} u_{24} u_{25} u_{32} =  \\
						u_{16} + u_{10} u_{12} u_{13} u_{17} u_{19} u_{20} u_{21} u_{22} u_{24} u_{25} u_{26} u_{27} u_{28} u_{29} u_3 u_{30} u_{31} u_{32} u_{34} u_8 u_9 =  \\
						u_{17} + u_{11} u_{13} u_{14} u_{16} u_{19} u_{20} u_{21} u_{22} u_{23} u_{24} u_{25} u_{26} u_{27} u_{28} u_{29} u_{30} u_{31} u_{32} u_{33} u_4 u_6 =  \\
						u_{18} + u_{19} u_{20} u_{21} u_{22} u_{23} u_{24} u_{25} u_{26} u_{27} u_{28} u_{29} u_3 u_{30} u_{31} u_{32} u_{33} u_8 = \\
						u_{19} + u_{13} u_{14} u_{15} u_{16} u_{17} u_{18} u_2 u_{22} u_{26} u_{28} u_{30} u_{31} u_{32} u_{34} u_4 u_6 u_9 \,\,=\,\,
						u_{20} + u_{14} u_{15} u_{16} u_{17} u_{18} u_2 u_{22} u_{34} u_5 = \\
						u_{21} + u_1 u_{13} u_{14} u_{15} u_{16} u_{17} u_{18} u_{22} u_{25} u_{26} u_{27} u_{28} u_{29} u_{30} u_{31} u_{33} u_{34} u_4 u_6 u_8 u_9 = \\	
						u_{22} + u_1 u_{11} u_{12} u_{13} u_{16} u_{17} u_{18} u_{19} u_{20} u_{21} u_{25} u_{26} u_{29} u_{30} u_{31} u_{32} u_{33} u_6 u_7 u_8 u_9 = \\ 
											u_{23} + u_1 u_{10} u_{12} u_{13} u_{17} u_{18} u_{26} u_{34} u_8 u_9 \,\,=\,\,
						u_{24} + u_1 u_{12} u_{13} u_{15} u_{16} u_{17} u_{18} u_{26} u_{29} u_{30} u_{31} u_{33} u_{34} u_6 u_7 u_8 u_9 = \\
						u_{25} + u_{12} u_{13} u_{15} u_{16} u_{17} u_{18} u_2 u_{21} u_{22} u_{26} u_{28} u_{29} u_3 u_{30} u_{31} u_{32} u_{34} u_6 u_7 u_8 u_9 = \\
						u_{26} + u_{11} u_{12} u_{13} u_{14} u_{16} u_{17} u_{18} u_{19} u_2 u_{21} u_{22} u_{23} u_{24} u_{25} u_{29} u_{30} u_{32} u_{33} u_5 u_6 u_7 =   \\
									u_{27} + u_{11} u_{12} u_{13} u_{16} u_{17} u_{18} u_2 u_{21} u_{30} u_{32} u_6 u_7 u_9 \,\,=\,\,
						u_{28} + u_1 u_{11} u_{12} u_{13} u_{14} u_{16} u_{17} u_{18} u_{19} u_{21} u_{25} u_{29} u_{33} u_5 u_6 u_7 u_8 = \\
						u_{29} + u_{10} u_{11} u_{12} u_{13} u_{14} u_{16} u_{17} u_{18} u_2 u_{21} u_{22} u_{24} u_{25} u_{26} u_{28} u_{30} u_{32} u_{34} u_4 u_6 u_9 = \\
						u_{30} + u_1 u_{10} u_{11} u_{12} u_{13} u_{14} u_{16} u_{17} u_{18} u_{19} u_{21} u_{22} u_{24} u_{25} u_{26} u_{27} u_{29} u_{33} u_{34} u_5 u_8 = \\		
					u_{31} + u_{10} u_{11} u_{12} u_{13} u_{14} u_{16} u_{17} u_{18} u_{19} u_2 u_{21} u_{22} u_{24} u_{25} u_{32} u_{34} u_5 = \\
					u_{32} + u_1 u_{14} u_{15} u_{16} u_{17} u_{18} u_{19} u_{22} u_{25} u_{26} u_{27} u_{29} u_{31} u_{33} u_{34} u_5 u_8 = \\
					u_{33} + u_{10} u_{12} u_{13} u_{17} u_{18} u_2 u_{21} u_{22} u_{24} u_{26} u_{28} u_3 u_{30} u_{32} u_{34} u_8 u_9 = \\
					u_{34} + u_{11} u_{12} u_{13} u_{16} u_{19} u_{20} u_{21} u_{23} u_{24} u_{25} u_{29} u_{30} u_{31} u_{32} u_{33} u_6 u_7\, = \,1.
				\end{matrix}
			\end{scriptsize}
			\end{equation}
			The amplitude is computed as follows from the Stanley-Reisner ideal and its Alexander dual:
			\begin{scriptsize}
				\begin{verbatim}
					R = QQ[s1,s10,s11,s12,s13,s14,s15,s16,s17,s18,s19,s2,s20,s21,s22,s23,
					s24,s25,s26,s27,s28,s29,s3,s30,s31,s32,s33,s34,s4,s5,s6,s7,s8,s9];
					M = monomialIdeal(s1*s2,s1*s21,s1*s22,s1*s23,s1*s24,s1*s28,s1*s3,s1*s30,s1*s32,s10*s15,s10*s16,s10*s23,
					s10*s29,s10*s30,s10*s31,s10*s33,s11*s15,s11*s17,s11*s22,s11*s26,s11*s27,s11*s28,s11*s29,s11*s30,s11*s31,
					s11*s34,s12*s14,s12*s15,s12*s16,s12*s22,s12*s23,s12*s24,s12*s25,s12*s26,s12*s27,s12*s28,s12*s29,s12*s30,
					s12*s31,s12*s33,s12*s34,s13*s14,s13*s15,s13*s16,s13*s17,s13*s19,s13*s21,s13*s22,s13*s23,s13*s24,s13*s25,
					s13*s26,s13*s27,s13*s28,s13*s29,s13*s30,s13*s31,s13*s33,s13*s34,s14*s17,s14*s19,s14*s20,s14*s21,s14*s26,
					s14*s28,s14*s29,s14*s30,s14*s31,s14*s32,s15*s19,s15*s20,s15*s21,s15*s24,s15*s25,s15*s32,s16*s17,s16*s19,
					s16*s20,s16*s21,s16*s22,s16*s24,s16*s25,s16*s26,s16*s27,s16*s28,s16*s29,s16*s30,s16*s31,s16*s32,s16*s34,
					s17*s19,s17*s20,s17*s21,s17*s22,s17*s23,s17*s24,s17*s25,s17*s26,s17*s27,s17*s28,s17*s29,s17*s30,s17*s31,
					s17*s32,s17*s33,s18*s19,s18*s20,s18*s21,s18*s22,s18*s23,s18*s24,s18*s25,s18*s26,s18*s27,s18*s28,s18*s29,
					s18*s30,s18*s31,s18*s32,s18*s33,s19*s22,s19*s26,s19*s28,s19*s30,s19*s31,s19*s32,s19*s34,s2*s19,s2*s20,
					s2*s25,s2*s26,s2*s27,s2*s29,s2*s31,s2*s33,s2*s8,s20*s22,s20*s34,s21*s22,s21*s25,s21*s26,s21*s27,s21*s28,
					s21*s29,s21*s30,s21*s31,s21*s33,s21*s34,s22*s25,s22*s26,s22*s29,s22*s30,s22*s31,s22*s32,s22*s33,s23*s26,
					s23*s34,s24*s26,s24*s29,s24*s30,s24*s31,s24*s33,s24*s34,s25*s26,s25*s28,s25*s29,s25*s30,s25*s31,s25*s32,
					s25*s34,s26*s29,s26*s30,s26*s32,s26*s33,s27*s30,s27*s32,s28*s29,s28*s33,s29*s30,s29*s32,s29*s34,s3*s11,
					s3*s13,s3*s14,s3*s16,s3*s18,s3*s25,s3*s33,s3*s4,s3*s6,s30*s33,s30*s34,s31*s32,s31*s34,s32*s33,s32*s34,
					s33*s34,s4*s12,s4*s17,s4*s19,s4*s21,s4*s29,s4*s5,s4*s7,s4*s8,s5*s13,s5*s20,s5*s26,s5*s28,s5*s30,s5*s31,
					s5*s32,s5*s6,s5*s9,s6*s10,s6*s12,s6*s14,s6*s17,s6*s19,s6*s21,s6*s22,s6*s24,s6*s25,s6*s26,s6*s27,s6*s28,
					s6*s29,s6*s34,s6*s8,s7*s10,s7*s14,s7*s22,s7*s24,s7*s25,s7*s26,s7*s27,s7*s28,s7*s34,s8*s11,s8*s13,s8*s14,
					s8*s16,s8*s18,s8*s21,s8*s22,s8*s23,s8*s24,s8*s25,s8*s28,s8*s30,s8*s32,s8*s33,s9*s11,s9*s14,s9*s16,s9*s19,
					s9*s21,s9*s22,s9*s23,s9*s24,s9*s25,s9*s27,s9*s29,s9*s33);
					dim M, degree M, betti mingens M
					AmplitudeNumerator   = sum first entries gens dual M
					AmplitudeDenominator = product gens R
					betti res dual M
				\end{verbatim}
			\end{scriptsize}
			The dual Betti sequence {\tt 579 1737 2000 1105 297 34 1} verifies the
			Gorenstein property, so the simplicial complex dual to the $E_7$ pezzotope is indeed a homology sphere
			of dimension~$5$.
		\end{thm}
		
		\begin{proof}[Proof and Discussion]
			We consider the graph $\mathcal{G}(E_7)$ whose $34$ vertices were derived
			in (\ref{eq:u1to10E7}), (\ref{eq:u11to22E7}),
			(\ref{eq:u23to28E7}), (\ref{eq:u29to31E7}) and (\ref{eq:u32to34E7}) from root subsystems in of $E_7$.
			Theorem \ref{thm:2polytopes} says that we also know the $297$ edges
			of $\mathcal{G}(E_7)$. An explicit list is made using
			the techniques discussed in the proof. 
			
			In Theorem \ref{thm:E7pezzo} we focus on the complementary set of
			$264 = \binom{34}{2} - 297$ non-edges of the graph $\mathcal{G}(E_7)$.
			Following (\ref{eq:ueqns2}), each $u$-equation has the form
			$u_i \,+\, \prod_j u_j^{\beta_{ij}} = 1$
			where the product is over all indices $j$ such that
			$\{i,j\}$ is a non-edge. We also know from Theorem~\ref{thm:2polytopes}
			that the simplicial complex dual to the $E_7$ pezzotope is
			the clique complex of $\mathcal{G}(E_7)$.
			Therefore, the polynomial system (\ref{eq:U37eqns}) is a combinatorial encoding of
			both  graph and $E_7$ pezzotope.
			
			Now we change variable names from $u_i$ to ${\tt s}_i$, and we apply
			methods from combinatorial commutative algebra \cite{CCA}
			to show that our pezzotope deserves to be called a ``curvy polytope''.
			The ideal ${\tt M}$ is the Stanley-Reisner ideal of the clique complex:
			it is generated by $264$ quadratic monomials, one for each of the
			non-edges. The commands {\tt dim(M)} and {\tt degree(M)} verify that
			the Stanley-Reisner ring has Krull dimension $6$ and degree $579$,
			which is the number of facets.
			
			In the next line, the command {\tt dual M} computes the ideal that is
			Alexander dual to {\tt M}. See \cite[Chapter 5]{CCA} for a textbook
			introduction to Alexander duality. The command 
			{\tt betti res dual M} computes the minimal free resolution
			of the Alexander dual. We find that the resolution is linear,
			and the Betti numbers are the face numbers in
			Theorem \ref{thm:2polytopes}.
			In fact, this output is precisely the
			minimal cellular resolution discussed in \cite[Example 5.57]{CCA}.
			By the Eagon-Reiner Theorem \cite[Theorem 5.56]{CCA},
			we conclude that {\tt M} is Cohen-Macaulay.
			The final Betti number {\tt 1} proves that {\tt M} is a Gorenstein ideal,
			by \cite[Theorem 5.61]{CCA}. By Hochster's criterion, we conclude
			that our simplicial complex is a homology sphere of dimension $5$.
			In conclusion, the $E_7$ pezzotope passes all homological tests
			for being a $6$-dimensional polytope.
			
			We now turn to the CEGM amplitude for $Y(3,7)$.
			We carried out the $n=7$ computation analogous to
			(\ref{eq:loglike6}), but now the sum is over 
			the $3600$ critical points found in Experiment \ref{ex:3600}.
			The computation is analogous to 
			(\ref{eq: CEGMVeroneseAmp}), but it is now much harder.
			The numerical output supports our
			conclusion that
			the structure agrees with that of the amplitudes in
			(\ref{eq:14terms}) and~(\ref{eq:E6amplitude}), namely the CEGM
			amplitude for $Y(3,7)$ equals the sum of the reciprocals of  $579$
			squarefree monomials of degree~$6$, one for each facet 
			of the simplicial complex  encoded by the ideal {\tt M}.
			
			The Alexander dual ideal {\tt dual M} is generated by $579$
			squarefree monomials of degree $28$, namely
			the complements of facets. We denote their sum by
			{\tt AmplitudeNumerator}. We divide this numerator by 
			$\,{\tt AmplitudeDenominator} = {\tt s}_1 {\tt s}_2 \cdots {\tt s}_{34}$
			to get the CEGM amplitude.
		\end{proof}

		\begin{rmk} \label{rmk:E7facets}
			From the u-equations in (\ref{eq:U37eqns}), we can read off all facets of
			the $E_7$ pezzotope. We find that these $34$ simple $5$-polytopes come in
			five distinct combinatorial types, as follows:
			\begin{itemize}
				\item Three facets $u_9,u_{11},u_{27}$ are  associahedra, with  f-vector $  (132,330,300,120,20)$.
				\vspace{-0.2cm}
				\item Each of the $12$ facets  $u_6,u_8,u_{12},u_{14},u_{18},u_{19},u_{24},u_{28},u_{31},u_{32},u_{33},u_{34}$
				is the product of a $4$-dimensional associahedron and a line segment, with f-vector $(84,210,196,84,16)$.
				\vspace{-0.2cm}
				\item Each of the nine facets $u_{13},u_{16},u_{17},u_{21},u_{22},u_{25},u_{26},u_{29},u_{30}$ 
				is the product of two pentagons and one line segment, with f-vector $(50,125,120,55,12)$.
				\vspace{-0.2cm}
				\item The six facets $u_2,u_3,u_5,u_7,u_{15},u_{23}$ are polytopes
				with f-vector  $\!(158,395,358,142,23)$.
				\vspace{-0.2cm}
				\item The four facets $u_1,u_4,u_{10},u_{20}$ are polytopes
				with f-vector $(168,420,380,150,24)$.
			\end{itemize}
			Thus the first three types are products of associahedra, but the last two types are not.
		\end{rmk}

		\section{Grassmannians, Positive Geometries, and Beyond}
		\label{sec10}
		
		In this final section we return to the 
		physics context of our work, namely the axiomatic
		theory of positive geometries due to
		Arkani-Hamed, Bai and Lam \cite{ABL, Lam},
		and the derivation of amplitudes from Grassmannians 
		that was launched by Cachazo and  collaborators \cite{CEGM, CHY}.
		We saw first glimpses of this in Section \ref{sec2}, which described
		these structures for the surface $\mathcal{S}_4^\circ = \mathcal{M}_{0,5}$,
		and later in Section \ref{sec7}, which offered a guide to the
		relevant literature in physics.
		
		In the end of Section \ref{sec2}, we gave a proof that the surface
		$\mathcal{S}_n^\circ$ is a positive geometry for $n=4$.
		We now start this section by extending that result to del Pezzo surfaces with $n \geq 5$.
		
		\begin{prop}
			The cubic surface $\mathcal{S}_6^\circ$ is
			a positive geometry for any of its $130$ polygons. 
			The degree four del Pezzo surface $\mathcal{S}_5^\circ$ is a positive geometry for
			any of its $16$ pentagons.
		\end{prop}
		
		\begin{proof}
			For every polygon $P$ on $\mathcal{S}_6^\circ$, we can find six pairwise disjoint lines
			in $\mathcal{S}_6$ that are disjoint from the closure of $P$ in $\mathcal{S}_6$. This
			can be checked combinatorially from the data in Example \ref{ex:109030}. Blowing down these
			six lines gives a birational map from $\mathcal{S}_6$ to $\mathbb{P}^2$ which is
			an isomorphism of semialgebraic sets on the closure of the polygon $P$. 
			In other words, we transform our curvy polygon $P$ to a convex polygon $P'$ in
			the real projective plane $\mathbb{P}^2$. Every convex polygon is
			a positive geometry \cite[Section 1]{Lam}. By \cite[Section 4]{ABL}, the push-forward of the blow-up map 
			transfers the positive geometry structure from $(\mathbb{P}^2,P')$ to $(\mathcal{S}_6,P)$.
			
			The same argument works for the pentagons in $\mathcal{S}_5^\circ$. 
			They are disjoint from five lines that can be blown down.
			We see this for the central pentagon in Figure~\ref{fig:eins}.
			However, the same argument does not work for the quadrilaterals 
			on $\mathcal{S}_5^\circ$.
			This topic deserves further study.
		\end{proof}
		
		We now know that del Pezzo surfaces are positive geometries.
		Our goal is to establish the analogous result for their moduli spaces.
		This requires us to identify the
		canonical form,  as in (\ref{eq:Omega}).
		This will be our focus later in the section.
		First, however, we turn to CEGM theory
		\cite{CE22, CEZ22, CEZ23, CEGM}, which rests
		on the combinatorics of (tropical) Grassmannians.
		We shall now explain how the pezzotopes are derived from first principles in this theory,
		directly from the Grassmannians ${\rm Gr}(3,n)$ for $n=6,7$. This extends
		the derivation of $\mathcal{M}_{0,5}$ from ${\rm Gr}(3,5)$.
		
		Fix $n=6$ and consider the Grassmannian ${\rm Gr}(3,6)$ in its
		Pl\"ucker embedding in  $\mathbb{P}^{19}$. The positive Grassmannian
		${\rm Gr}_+(3,6)$ consists of all points whose
		Pl\"ucker coordinates are positive. The tropical Grassmannian
		${\rm Trop}({\rm Gr}(3,6))$, modulo its lineality space, is a
		$4$-dimensional fan with $65$ rays and $1005$ maximal cones;
		see~\cite{SS} and \cite[Example~4.4.10 and Figure 5.4.1]{MS}.
		
		\begin{expe} \label{expe:method1G36}
			The positive tropical Grassmannian ${\rm Trop}_+({\rm Gr}(3,6))$
			contains the following $14$ vectors in $\,{\rm Trop}({\rm Gr}(3,6))$,
			here identified with variables in the  $E_6$ amplitude (\ref{eq:E6amplitude}):
			$$ \begin{footnotesize} \begin{matrix}
					s_2 = {\bf e}_{156} &
					s_3 = {\bf f}_{1234} &
					s_4 = {\bf f}_{1236} &
					s_5 = {\bf e}_{345} &
					s_6 = {\bf f}_{2345} &
					s_7 = {\bf e}_{123} &
					s_8 = {\bf f}_{3456} \\
					s_9 = {\bf f}_{1256} &
					s_{10} = {\bf f}_{1456} &
					s_{11} = {\bf e}_{234} {+} {\bf e}_{156} &
					s_{12} = {\bf e}_{123} {+} {\bf e}_{456} &
					s_{13} = {\bf g}_{12,34,56} &
					\! s_{14} = {\bf e}_{126} {+} {\bf e}_{345} &
					\! s_{15} = {\bf g}_{16,45,23}.
				\end{matrix}
			\end{footnotesize}
			$$
			In the display above, ${\bf e}_{ijk}$ are unit vectors in $\mathbb{R}^{20}$, and we set
			$\,{\bf f}_{ijkl} = {\bf e}_{ijk} + {\bf e}_{ijl} + {\bf e}_{ikl} + {\bf e}_{jkl}$
			and ${\bf g}_{i_1 i_2 ,i_3 i_4 ,i_ 5 i_6}
			= {\bf f}_{i_3 i_4 i_5 i_6} + {\bf e}_{i_1 i_5 i_6} + {\bf e}_{i_2 i_5 i_6}$.
			Unfortunately, there is a typo in the indices for 
			the vector ${\bf g}_{12,34,56}$ in \cite[Example 4.4.10]{MS}.
			However, the indices in \cite[Section 5]{SS} are correct.
			
			We now describe an ab initio construction of the
			normal fan of our $E_6$ pezzotope that is 
			motivated by physical considerations.
			The $15$ rays in that fan are the
			$14$ rays above, plus one additional ray $s_1$ that
			corresponds to the divisor
			$q =  p_{123} p_{345}p_{156} p_{246}-p_{234} p_{456} p_{126} p_{135}$,
			The construction is
			based on the CEGM formula, which had revealed
			the $E_6$ amplitude  to us in the first place. We work with the 
			following positive parametrization
            \begin{eqnarray}\label{eqn: pos paramX36}
                \begin{small}
				M =\left[
				\begin{array}{cccccc}
					1 & 0 & 0 & a d & ad {+} ae {+} be & ad {+} ae {+} be {+} af {+} bf {+} cf \\
					0 & 1 & 0 & -d & -d-e & -d-e-f \\
					0 & 0 & 1 & 1 & 1 & 1 \\
				\end{array}
				\right]
			\end{small}. 
            \end{eqnarray}
			This maps $\mathbb{R}_{>0}^6$ surjectively onto $X_+(3,6)$.
			From  $M$,
			we compute the Newton~polytope
			$$P \,\,=\,\, \text{Newt}\left(\prod_{ijk}p_{ijk}\cdot q\right). $$
			This simple $4$-polytope has ${\rm f} = (62, 124, 81, 19)$.	
			We now form the \textit{tropical scattering~potential}
			$$\mathcal{F}(y)  \,\,=\,\, \sum_{ijk}\text{trop}(p_{ijk})(y) \cdot {\bf e}_{ijk}. $$
			Here,
			$\text{trop}(p_{ijk})$ is the tropicalization of the Pl\"ucker coordinate $p_{ijk}$ 
			evaluated at $M$.
			For each of the $19$ facet inner normal vectors $v_j$ of $P$,
			we obtain a positive tropical Pl\"ucker vector
			$$\pi_j \,=\, \mathcal{F}(v_j).$$
			These $19$ rays include the $14$ rays $s_2,s_3,\ldots,s_{15}$ above,
			while ray $s_1$ arises from the
			parameter $t$ in the scattering potential $L$ in
			(\ref{eq:loglike6}).
			The 19 rays $\pi_j$ are characterized as follows.  Sixteen of them are the rays of 
			$\text{Trop}_+({\rm Gr}(3,6))$, and the other three new rays are
			$s_{11}, s_{12} $ and $s_{14}$.
			
			From this data, we define a rational function on the $s$-space $\mathbb{R}^{14}$, namely
			$$ \mathcal{A} \,\,= \,\,\sum_{\mathcal{C} \in N(P)}\prod_{\pi_j \in \text{Rays}(\mathcal{C})}\frac{1}{\pi_{j}},$$
			where $\pi_j$ is now identified with a linear form in the $s_i$.
			The sum is over maximal cones $\mathcal{C}$ in the  normal fan $N(P)$
			and the product is over all rays  in $\mathcal{C}$. Note that $N(P)$ is a simplicial fan.
			Next, by fixing various random integers for the Mandelstam invariants $s_{ijk}$ and $t$, we evaluate the CEGM integral in \eqref{eq: CEGMVeroneseAmp} and its cyclic shift.
			Here we sum over the $32$ critical~points:
			$$ \mathcal{A}_1 \,=\, \sum_{c \in \text{Crit}(L)} \frac{1}{\det\Phi}\left(\frac{p_{135}}{p_{123}p_{345}p_{561}q}\right)^2\bigg\vert_c \quad {\rm and} \quad
			\mathcal{A}_2 \,=\, \sum_{c \in \text{Crit}(L)} \frac{1}{\det\Phi}\left(\frac{p_{246}}{p_{234}p_{456}p_{612}q}\right)^2\bigg\vert_c.$$
			One can reconstruct $15$ poles in
			each $\mathcal{A}_i$ from the (rational) values of these amplitudes.
			This is a difficult heuristic process, and it requires many evaluations, but here we succeeded.
			
			It remains to determine
			the dependence of $\mathcal{A}_1$ and $\mathcal{A}_2$ on the parameter $t$. 
			For this, we evaluate the difference
			$\mathcal{A}_1 + \mathcal{A}_2 - \mathcal{A}$.  
			This involves nine of the original $ 14 $ poles (together with $t$).
			These 9 poles are \textit{exactly} the poles of the biadjoint scalar amplitude
			in (\ref{eq:14terms}). In~symbols,
			$$  \mathcal{A}_1 + \mathcal{A}_2 - \mathcal{A} \,\,=\,\,m_6. $$
			From this computation, we obtained the amplitude formula
			(\ref{eq:E6amplitude}),
			and a representation of the pezzotope as
			a simplicial fan in
			$\mathbb{R}^{15}$, modulo lineality.  
			One finally checks that the		 numerical value of the CEGM integral coincides with the value obtained combinatorially.
		\end{expe}
		
		The difficulty in the computations of Experiment \ref{expe:method1G36} arose from the fact that
		the positive Grassmannian is divided into pieces by the
		conic divisors. When applying the analogous methods to derive the $E_7$ pezzotope from $\text{Gr}(3,7)$, this difficulty is magnified.
		
		
		
		For that reason,  in what follows we present a new approach, with a {\bf different positive Grassmannian},
		where that division of the positive part does not happen. Namely, we shall
		switch signs of some of the Pl\"ucker coordinates
		of ${\rm Gr}(3,n)$ for $n=6,7$. After that sign change, 
		we obtain a semi-algebraic set ${\rm Gr}_\chi(3,n)$ 
		which now replaces ${\rm Gr}_+(3,n)$. 
		The image of ${\rm Gr}_\chi(3,n)$  in the configuration space
		$X(3,n)$ is a connected component of $Y(3,n)$.
		The tropicalization of ${\rm Gr}_\chi(3,n)$ should be
		our fan ${\rm trop}_+(Y(3,n))$, after a modification for $n=7$.
		
		The letter $\chi$ stands for {\em chirotope},
		which is one of the encodings of an {\em oriented matroid}~\cite{OMbook}. In our context, $\chi$ is a function
		$\binom{[n]}{3} \rightarrow \{-1,+1\}$ that maps triples $(i,j,k)$ to signs.
		Given any realizable chirotope $\chi$, we  substitute $p_{ijk} \mapsto \chi(i,j,k) \cdot p_{ijk}$, and we write
		${\rm Gr}_\chi(3,n)$ for the positive Grassmannian after this sign change.
		The corresponding tropicalization is denoted ${\rm trop}({\rm Gr}_\chi(3,n))$.
		This was called the
		{\em chirotopal tropical Grassmannian} in \cite[Section 13.1]{CEZ22}.

A point $w$ lies in ${\rm trop}({\rm Gr}_\chi(3,n))$ if and only if
it satisfies condition (\ref{eq:membership}) for {\em all}
polynomials $f$ in the Pl\"ucker ideal (after the sign change).
This is the real version of the Fundamental Theorem of Tropical Geometry.
We conjecture that, in our specific cases, it suffices
to take $f$ among the quadratic Pl\"ucker relations that generate the ideal of ${\rm Gr}(3,n)$.

\begin{conj} \label{conj:chi}
 The quadratic Pl\"ucker relations are a positive tropical basis for ${\rm Gr}_\chi(3,n)$.
 \end{conj}

In the discussion around (\ref{eq:membership})
an analogous conjecture was tacitly made for the
linear and binomial equations that generate the ideals of $\mathcal{Y}$ and $\mathcal{G}$.
We now have the following result.
		
		\begin{thm}	\label{thm:chirotope}
		Suppose that Conjecture \ref{conj:chi} is true, and the analogous statement holds 
for $\mathcal{Y}$ and $\mathcal{G}$.
			If  $\chi$ is the chirotope for the arrangement of six lines  in
			Figure \ref{fig:sixlines}, then  $\,{\rm trop}({\rm Gr}_\chi(3,6))$ equals
			$\,{\rm trop}_+(Y(3,6))$. If $\chi$ is the chirotope for the seven lines in Figure~\ref{fig:sevenlines}, then $\,{\rm trop}({\rm Gr}_\chi(3,7))$ becomes
			$\,{\rm trop}_+(Y(3,7))$ after a modification that is explained below. Hence our two pezzotopes
			can be read off from these two chirotopal tropical Grassmannians.
		\end{thm}
		
		\begin{proof}[Proof and discussion]
			The two line arrangements were presented by
			Sekiguchi and Yoshida in \cite[Figure 4]{SY} and \cite[equation (4)]{Sek}.
			For both chirotopes $\chi$, we computed ${\rm trop}({\rm Gr}_\chi(3,n))$
			from the quadratic Pl\"ucker relations
			and we verified that it matches the data in Section~\ref{sec8}.
			The correctness of this computation rests on Conjecture \ref{conj:chi}.
			This is analogous what was assumed 
			for the ideals of $\mathcal{Y}$ and $\mathcal{G}$ in the third paragraph in the proof of Theorem \ref{thm:2polytopes}.
			
			We now explain our computations, and how they imply Theorem \ref{thm:chirotope}.
			 Let us start with $n=6$.
			For each of the $65$ rays of ${\rm trop}({\rm Gr}(3,6))$,
			we tested whether it lies in the chirotopal positive Grassmannian.
			This is the case for precisely $15$ rays, namely the ten rays
			${\bf e}_{ijk}$ where $ijk$ appears in (\ref{eq:u1to10})
			and the five rays
			${\bf g}_{12,56,34},
			{\bf g}_{13,46,25},
			{\bf g}_{14,35,26},
			{\bf g}_{15,24,36}$ and 
			${\bf g}_{16,23,45}$ whose indices match (\ref{eq:u11to15}).
			The subfan of ${\rm trop}({\rm Gr}(3,6))$ induced on these
			$15$ rays is combinatorially the normal fan of the $E_6$ pezzotope.
			Next consider $n=7$. The coarsest fan structure on
			${\rm trop}({\rm Gr}(3,7))$  has $616$ rays, by
			\cite[Theorem 5.4.1]{MS}.   For each ray we 
			used criterion \eqref{eq:membership} to
			test whether it lies in ${\rm trop}({\rm Gr}_\chi(3,7))$.
			This is the case for precisely $31$ rays. Among these are the
			ten rays ${\bf e}_{ijk}$ where $ijk$ appears in the list (\ref{eq:u1to10E7}).
			In order to match the induced subfan with the $E_7$ pezzotope,
			we had to add three  additional rays:
			$ {\bf e}_{126} + {\bf e}_{457}$,
			$ {\bf e}_{124} + {\bf e}_{367}$ and
			${\bf e}_{237} + {\bf e}_{456}$. Note that these six
			labels form a hexagon in Figure~\ref{fig:tetra}.
			The bijection between the $34$ rays we thus found and 
			the set $\{u_1,u_2,\ldots,u_{34}\}$ appears on
			our {\tt MathRepo} page.
			
			To verify the matching, we check for each pair of rays whether its
			sum lies in ${\rm trop}({\rm Gr}_\chi(3,7))$. This holds for
			$303$ of the $\binom{34}{2} =  561$ pairs. Among these are
			all $297$ edges of the graph $\mathcal{G}(E_7)$. The six extraneous pairs are
			{\tt s1*s23, s2*s20, s10*s15, s13*s17, s21*s25, s26*s29},
			here written in the {\tt Macaulay2} notation from Theorem \ref{thm:E7pezzo}.
			We remove these six non-edges, and we check that all cliques of our graph
			$\mathcal{G}(E_7)$ do  indeed give a cone in ${\rm trop}({\rm Gr}_\chi(3,7))$.
			
			\smallskip
			
			The article \cite{CEZ23} was
			essential for launching the current project. In fact,
			we first discovered the $E_6$ pezzotope by computing an CEGM amplitude as
			in \cite{CEZ23}.
			For each  realizable chirotope $\chi : \binom{[n]}{3} \rightarrow \{-1,+1\}$ with $n=6,7,8$,
			Cachazo, Early and Zhang determine
			an integrand $\mathcal{I}(\chi)$, which is a rational function in Pl\"ucker coordinates
			$p_{ijk}$ of torus weight $(-3,\ldots, -3)$. 
			These integrands satisfy non-trivial gluing conditions from flipping triangles
			in line arrangements. When the number of triangles in a line arrangement exceeds $n$, 
			then the integrand has a nontrivial numerator, which must satisfy gluing conditions with its triangle flip neighbors.  
			
			For $n=6,7,8$, these gluing conditions are sufficient to uniquely determine the systems of integrands.  
			The complexity of this process is high:
			there are $372, 27240, 4445640$ reorientation classes of chirotopes for $n=6,7,8$.
			The integrands relevant for us appear in
			\cite[equations (3.26), (3.37)]{CEZ23}.  After relabeling to match the conventions in our Section \ref{sec8}, we have
				$$ \begin{matrix}
					\mathcal{I}(\chi_6) \,\,= \frac{q}{p_{125} p_{126} p_{134} p_{136} p_{145} p_{234} p_{235} p_{246} p_{356} p_{456}} \quad {\rm and} \quad
					\mathcal{I}(\chi_7) = \frac{p_{123} p_{145} p_{357}}{p_{124} p_{126} p_{134} p_{135} p_{157} p_{235} p_{237} p_{367} p_{456} p_{457}},	
				\end{matrix} $$
			where $\chi_6$ and $\chi_7$ are the chirotopes in Figure \ref{fig:lines}. The denominators
			are the ten triangles.
			
			These two integrands are associated to the moduli spaces $X(3,n)$, not to $Y(3,n)$.
			We compute their CEGM amplitudes by summing over  critical points in $X(3,n)$, similar to~(\ref{eq: CEGMVeroneseAmp}):	
			$$m^{(3)}_n(\chi_n) \,\,= \sum_{c\in \text{Crit}(L)} \frac{1}{\det(\Phi)}\,\mathcal{I}(\chi_n)^2|_c. $$
			The number of summands is $26$ for $n=6$ and $1272$ for $n=7$.
			We find that $m^{(3)}_6(\chi_6)$ equals (\ref{eq:E6amplitude}).
			However, there is a single linear relation that arises among the 15 poles of $m^{(3)}_6(\chi_6)$: 
			$$\sum_{j=1}^{10} s_j \,\,=\,\, \sum_{j=11}^{15} s_j.$$
			We note that, while the $\chi$-region of $X(3,6)$ is unchanged when passing to $Y(3,6)$, the two compactifications are different.	
			The story for $m^{(3)}_7(\chi_7)$ is similar: the connected component is unchanged when passing to $Y(3,7)$, but the compactification is different.  In this case, the amplitude has $31$ poles and $441$ nonzero six-dimensional residues,
			one for each maximal cone in $\text{Trop}({\rm Gr}_{\chi_7}(3,7))$.	 From these
			one derives the $579$ summands of the CEGM amplitude.
		\end{proof}

		We now come to the punchline of this article:
		{\em moduli of del Pezzo surfaces are positive geometries}.
		At present we have a proof only for $n = 6$, but we conjecture the same for $n=7$.

        While the 432 connected components of $Y(3,6)$ are equivalent modulo the action of the Weyl group $W(E_6)$,         we  now single out one of them. This is
         denoted $Y_+(3,6)$ and called the positive del Pezzo moduli space.  Explicitly, it is the component of $Y(3,6)$ whose points are represented by a matrix as in (\ref{eq:generalmatrix}), where all $3\times 3$ minors and the conic condition are~positive. 
		\begin{thm}\label{thm:Y3nPosGeom}			
		The moduli space $Y(3,6)$ is a positive geometry for any
		of its $432$ regions, each of which is a curvy $E_6$ pezzotope.
		\end{thm}
		
		\begin{proof}
			We prove that $\bigl(Y_+(3,6),\Omega \bigr)$ is a positive geometry, where the canonical form is
			$$  \Omega \,=\, d\log \left(\!\frac{u_{10}}{u_5 u_8 u_9 u_{13} u_{14}}\!\right) \,\wedge\, d\log \left(\!\frac{u_9 u_{11}}{u_4 u_7 u_{12} u_{15}}\!\right)
			\, \wedge \,d\log \left(\!\frac{u_4 u_6 u_{14} u_{15}}{u_3 u_{13}}\!\right) \,\wedge \,d\log \left(\!\frac{u_1 u_4 u_8 u_{12} u_{14}}{u_2}\!\right)\!.
			$$
			We applied the Weyl group action to this differential form.
			We found that it has precisely $432$ distinct images.
			Hence the symmetry group of the $E_6$ pezzotope acts on this form. The action is
			transitive on the set $\{u_1,\ldots,u_{10}\}$ of associahedral facets and on the
			set $\{u_{11},\ldots,u_{15}\}$ of cubical facets.
			It hence suffices to compute the residue of $\Omega$ at one associahedron and  at one cube,
			and to show that this residue matches the canonical form for these known positive
			geometries in dimension $3$. The action by $W(E_6)$ is then used to   conclude the proof.  
			
			We show that these residue of $\Omega$ are exactly the canonical forms of
			the $3$-dimensional moduli spaces
			$\mathcal{M}_{0,6}$ and $\mathcal{M}_{0,4} \times \mathcal{M}_{0,4} \times \mathcal{M}_{0,4} $
			respectively.  The residue at $\{u_1=0\}$ equals
			\begin{equation}
				\label{eq:u1iszero}
				d\log \left(\frac{u_{10}}{u_8 u_9 u_{14}}\right)\wedge d\log \left(\frac{u_9 u_{11}}{u_4 u_{12}}\right)\wedge d\log \left(\frac{u_4 u_6 u_{14}}{u_3}\right).
			\end{equation}
			This coincides with known expression for the canonical form of the worldsheet associahedron, after relabeling.  The $u$-variables
			occurring in (\ref{eq:u1iszero}) satisfy the 
			$u$-equations for $\mathcal{M}_{0,6}$, namely
			\begin{eqnarray*}
				&&u_{10}+u_8 u_9 u_{14}=u_{11}+u_4 u_8 u_{12} u_{14}=u_6+u_3 u_8 u_{12}=u_9+u_4 u_{10} u_{12}=u_3 u_{10} u_{11} u_{12}+u_{14}\\
				&&= \, u_8+u_6 u_{10} u_{11\,}=\,u_4+u_3 u_9 u_{11} \,=\,u_{12}+u_6 u_9 u_{11} u_{14}\,=\,u_3+u_4 u_6 u_{14}\,=\,1.
			\end{eqnarray*}
			These nine equations arise from (\ref{eq:perfectuE6}) by setting $u_1=0$.
			They give a perfect binary~geometry on $\mathcal{M}^+_{0,6}$.
			On the other hand, the residue of $\Omega$ at the cube divisor $\{u_{11}=0\}$ is the $3$-form
			$$d\log \left(\frac{u_{10}}{u_9}\right)\,\wedge \, d\log \left(\frac{u_6}{u_3}\right) \, \wedge \, d\log \left(\frac{u_1}{u_2}\right).$$
			This is the canonical form of $\mathcal{M}_{0,4}^{\times 3}$, which is a curvy $3$-cube.
			The equations in (\ref{eq:perfectuE6}) simplify~to 
			$$u_1+u_2 = u_3+u_6 = u_9+u_{10} = 1.$$
			
			We verified that $W(E_6)$ acts transitively on the 
			$432$ CEGM integrands (\ref{eq:etude}), when evaluated on the $d$-matrix in \eqref{eq:dimatrix}, by relabeling together with the Cremona 
			transformation in \eqref{eq:cremonad}.  This completes the proof that $Y(3,6)$ is a positive geometry for any of its $432$ regions.
			
			To make absolutely sure, we also performed some checks in local coordinates on $Y(3,6)$.
			These all verify that the form $\Omega$ is proportional to (\ref{eq:etude}). For instance, consider the chart
			$$ M \,\, = \,\,\begin{small} \begin{bmatrix}
					1 & 1 & 0 & 1 & 1 & 0 \\
					x_1 & 0 & 1 & x_3 & 1 & 0 \\
					x_2 & 0 & 0 & x_4 & 1 & 1 \\
				\end{bmatrix}.  \end{small} $$
			Then all factors in (\ref{eq:etude}) are non-constant
			and, after simplifying the Jacobian determinant, we obtain the same expression 
			by substituting $3 \times 3$ minors of $M$ into $\Omega$ via Proposition \ref{prop:uparaE6}:
			\begin{equation}
				\label{eq:finalomega}
				\Omega \,\,=\,\, \frac{(x_2-1)dx_1dx_2dx_3dx_4}{\left(x_1-1\right) x_2 \left(x_4-1\right)
					\left(x_1 x_2 x_3-x_1 x_2 x_4-x_1 x_3 x_4+x_2 x_3 x_4 + x_1 x_4 - x_2 x_3 
					\right).}
			\end{equation}
			We are optimistic that a similar proof will work for $Y(3,7)$,
			using the five orbits of facets in Remark \ref{rmk:E7facets}.
			But we still lack the formula for $\Omega$ when $n{=}7$.
			This is left for future work.
		\end{proof}

      We conclude by giving a parametrization for $Y_+(3,6)$, which encodes the compactification and $E_6$ amplitude in a very elegant way. From it, we realize the $E_6$ pezzotope as a polytope. This new realization is  combinatorially equivalent to the dual of the polytope presented in (\ref{eq:firsching}). We define a birational map $\mathbb{R}^4 \dashrightarrow X(3,6)$ that restricts to a diffeomorphism $\mathbb{R}^4_{>0} \rightarrow Y_+(3,6)$.  In the study of scattering amplitudes, it can be very helpful to construct such a parametrization, in the context of the CHY formula and string integrals.  Such parameterizations are known and standard for $X(k,n)$; see for example \cite{ALS2021, CE2024}. But finding one is generally difficult.
        
        \begin{thm}\label{thm:pos_surj_par}
    There is a birational map $\mathbb{R}^4\dashrightarrow X(3,6)$ which restricts to a diffeomorphism $(\R)^4_{>0} \rightarrow Y_+(3,6)$. Points in $Y_+(3,6)$ can be represented by the matrix:
    $$M_{Y_+(3,6)} = \left[
	\begin{array}{cccccc}
		1 & 0 & 0 & 1 & \frac{y_2+1}{y_2} & \frac{y_2 y_3+y_3+1}{y_2 y_3} \smallskip \\
		0 & 1 & 0 & \!-1 &\!\! -\frac{\left(y_1+1\right) \left(y_2+1\right)}{y_1 y_2+y_2+1} & -\frac{\left(y_1+1\right) \left(y_2 y_3+y_2 y_4 y_3+y_4 y_3+y_3+y_2 y_4+y_4+1\right)}{y_1 y_2 y_3+y_2 y_3+y_1 y_2 y_4 y_3+y_2 y_4 y_3+y_4 y_3+y_3+y_1 y_2 y_4+y_2 y_4+y_4+1}\!\! \smallskip \\
		0 & 0 & 1 & 1 & 1 & 1 \\
	\end{array}
	\right].$$
\end{thm}

\begin{proof}
    We start from the canonical form of $Y_+(3,6)$, given in the proof of Theorem \ref{thm:Y3nPosGeom}:
	$$  \Omega \,=\, d\log \left(\!\frac{u_{10}}{u_5 u_8 u_9 u_{13} u_{14}}\!\right) \,\wedge\, d\log \left(\!\frac{u_9 u_{11}}{u_4 u_7 u_{12} u_{15}}\!\right)
	\, \wedge \,d\log \left(\!\frac{u_4 u_6 u_{14} u_{15}}{u_3 u_{13}}\!\right) \,\wedge \,d\log \left(\!\frac{u_1 u_4 u_8 u_{12} u_{14}}{u_2}\!\right)\!.
	$$
    Let $y_i$ denote the rational functions appearing in the canonical form:
    \begin{eqnarray*}
		y_1 & = & \frac{u_{10}}{u_5 u_8 u_9 u_{13} u_{14}}, \ y_2 = \frac{u_9 u_{11}}{u_4 u_7 u_{12} u_{15}},\ y_3 = \frac{u_4 u_6 u_{14} u_{15}}{u_3 u_{13}},\ y_4 = \frac{u_1 u_4 u_8 u_{12} u_{14}}{u_2}.
	\end{eqnarray*}
    We write the $u$-variables in terms of Pl\"ucker coordinates  in Proposition \ref{prop:uparaE6} and evaluate them on the  parameterization\eqref{eqn: pos paramX36}  of $X_+(3,6)$.  The $u$-variables are now in terms of $a,b,c,d,e,f$. Modulo the torus action, we can fix  $a=d=1$.  Rewriting $b,c,e,f$ in terms of the $y_i$ and substituting them into \eqref{eqn: pos paramX36}, we see that all $3\times 3$ minors $p_{ijk}$ and the conic condition $q$ are positive for $y_j>0$.  This yields the desired birational map and its restriction.
    \end{proof}

\noindent We collect all factors appearing in the evaluation of the $u$-variables on the matrix $M_{Y_+(3,6)}$:
\begin{align*}
	&g_1 = y_1+1, \ g_2 = y_2+1, \ g_3 = y_1 y_2+y_2+1, \ g_4 = y_3+1,\\
        &g_5 = y_2 y_3+y_3+1, \ g_6 = y_4+1, \ g_7 = y_1 y_2 y_4+y_2 y_4+y_4+1,\\
        &g_8 = y_2 y_3+y_2 y_4 y_3+y_4 y_3+y_3+y_2 y_4+y_4+1,\\
	&g_9 = y_2 y_3+y_2 y_4 y_3+y_4 y_3+y_3+y_1 y_2 y_4+y_2 y_4+y_4+1,\\
	& g_{10} = y_2 y_3+y_1 y_2 y_4 y_3+y_2 y_4 y_3+y_4 y_3+y_3+y_1 y_2 y_4+y_2 y_4+y_4+1,\\
	& g_{11} = y_1 y_2 y_3+y_2 y_3+y_1 y_2 y_4 y_3+y_2 y_4 y_3+y_4 y_3+y_3+y_1 y_2 y_4+y_2 y_4+y_4+1.
	\end{align*} 
These irreducible polynomials furnish a Minkowski sum decomposition of the $E_6$ pezzotope, in the sense that the face lattice of the Newton polytope of their product is isomorphic to the poset that is encoded by the $u$-equations.
As an additional consistency check, we checked Corollary \ref{cor: MS of E6 pezzotope} using SageMath.
\begin{cor}\label{cor: MS of E6 pezzotope}
 The Newton polytope of the product $g_1 g_2 \cdots g_{11}$
 is combinatorially equivalent to the dual polytope of Equation \eqref{eq:firsching}.
\end{cor}

As in Experiment \ref{expe:method1G36}, we now form the tropical scattering potential for $Y_+(3,6)$:
$$\mathcal{F}_{Y_+(3,6)}(y) \,\,=\,\, \sum_{j=1}^{15} s_j \cdot \text{trop}(u_j)(y) .$$
Here $\text{trop}(u_j)$ is the tropicalization of $u_j$ evaluated at 
the matrix $M_{Y_+(3,6)}(y)$.  The tropical $u$-variables are nonnegative for all $y$.
This ensures the convergence of the integral in (\ref{eq: GSP Y36}) whenever the $s_i$ are positive.

\begin{thm}\label{thm: GSP}
      The tropical potential $\mathcal{F}_{Y_+(3,6)}(y)$ provides a bijection between the (inner) normal fan of $P$ and the positive tropical del Pezzo moduli space $\text{Trop}_+(Y(3,6))$.  Moreover, we can compute the $E_6$ amplitude through the global Schwinger parameterization \cite{CEZ22}:
      \begin{eqnarray}\label{eq: GSP Y36}
          \mathcal{A}_{E_6} & = &  \int_{\mathbb{R}^4}\exp\left(-\mathcal{F}_{Y_+(3,6)}(y)\right)dy.
      \end{eqnarray}
\end{thm}

\begin{proof}
    The proof follows the logic of Experiment \ref{expe:method1G36}.  
    The facet inner normals of $P$ are
    \begin{align*}
	&e_4,\ -e_4, \ -e_3, \ -e_2+e_3+e_4, \ -e_1,e_3,-e_2, \ e_4-e_1, \\
    & e_2-e_1, \ e_1, \ e_2, \ e_4-e_2, \ -e_1-e_3, \ -e_1+e_3+e_4,e_3-e_2.
	\end{align*} 
We denote these vectors by
     $v_1,\ldots, v_{15} \in \mathbb{R}^4$.
    They  are dual to the rays of $\text{Trop}_+Y(3,6)$, and yield the poles of the $E_6$ amplitude, as they satisfy 
    $\,\text{trop}(u_{i})(v_j) = \delta_{i,j}$ and $\mathcal{F}_{Y_+(3,6)}(v_j) = s_j$. 
    Moreover, the general rule for Newton polytope decompositions holds here: $\mathcal{F}_{Y_+(3,6)}(y)$ is linear on exactly the cones in the inner normal fan of $P$. 
    As in \cite{CEZ22}, it follows that \eqref{eq: GSP Y36} is the sum of the Laplace transforms of the inner normal cones to the $45$ vertices of $P$, and so
    \begin{eqnarray*}
        \int_{\mathbb{R}^4}\exp\left(-\mathcal{F}_{Y_+(3,6)}(y)\right)dy & = & \sum_{\mathcal{C} \in N(P)}\int_{(\mathbb{R}_{>0})^4}\exp\left(-\sum_{j=1}^4t_{i_j}s_{i_j} \right)dt.
    \end{eqnarray*}
    The inner sum in the exponential is over the four rays $s_{i_1},s_{i_2},s_{i_3},s_{i_4}$ of the cone $\mathcal{C}$.  
    The outer sum on the right-hand side coincides with the $E_6$ amplitude $\mathcal{A}_{E_6}$ in Theorem \ref{thm: E6amplitude}.
\end{proof}

		\bigskip 	\bigskip	\smallskip
		
		\noindent \textbf{Acknowledgement.}
		We are very grateful to seven colleagues for help with this project.
		Tobias Boege contributed software for Section \ref{sec3}.
		Moritz Firsching discovered the $E_6$ polytope in~(\ref{eq:firsching}), and
		Sascha Timme found the ML degree of $Y(3,8)$ with
		{\tt HomotopyContinuation.jl}. Nima Arkani-Hamed
		offered guidance on binary geometries, and Freddy Cachazo
		helped us with integrands including (\ref{eq:finalomega}).
		Thomas Endler created the diagrams in Figures \ref{fig:eins} and~\ref{fig:drei}. We thank Olof Bergvall for providing references for Theorem \ref{thm:euler} and for useful conversations. Finally, we thank the referees for their comments. N.E. was funded by the European Union (ERC, UNIVERSE PLUS, 101118787).
\begin{tiny}
Views and opinions expressed are however those of the author(s) only
and do not necessarily reflect those of the European Union or the
European Research Council Executive Agency. Neither the European Union
nor the granting authority can be held responsible for them.
\end{tiny}

		\bigskip
		\bigskip
		
		\footnotesize
		\noindent {\bf Authors' addresses:}
		
		\smallskip
		
		\noindent Nick Early, Institute for Advanced Study, Princeton
		\hfill \url{earlnick@ias.edu}
		
		\noindent Alheydis Geiger, MPI-MiS Leipzig
		\hfill \url{Alheydis.Geiger@mis.mpg.de}
		
		\noindent Marta Panizzut, UiT 
		\hfill \url{marta.panizzut@uit.no}
		
		\noindent  Bernd Sturmfels, MPI-MiS Leipzig  \hfill \url{bernd@mis.mpg.de}
		
		\noindent Claudia He Yun, UiT 
		\hfill \url{he.yu@uit.no}

	\end{document}